\documentclass[12pt]{article}
\usepackage{graphicx}
\usepackage{amsmath,amsthm,amssymb,enumerate}
\usepackage{euscript,mathrsfs}
\usepackage{nicefrac}
\usepackage{bm}
\usepackage{cite}
\usepackage{morefloats}
\usepackage[citecolor=blue,colorlinks=true,linkcolor=blue]{hyperref}
\usepackage[left=2cm,right=2cm,top=3.5cm,bottom=3.5cm]{geometry}
\usepackage{color}
\usepackage{multirow}

\usepackage{soul}
\usepackage[normalem]{ulem}

\catcode`\@=11 \@addtoreset{equation}{section}

\catcode`\@=12

\allowdisplaybreaks

\usepackage[labelformat=simple]{subcaption}

\newtheorem{Theorem}{Theorem}[section]
\newtheorem{Proposition}[Theorem]{Proposition}
\newtheorem{Lemma}[Theorem]{Lemma}
\newtheorem{Corollary}[Theorem]{Corollary}

\theoremstyle{definition}
\newtheorem{Definition}[Theorem]{Definition}

\newtheorem{Remark}[Theorem]{Remark}

\newcommand{\bTheorem}[1]{
\begin{Theorem} \label{T#1} }
\newcommand{\eT}{\end{Theorem}}

\newcommand{\bProposition}[1]{
\begin{Proposition} \label{P#1}}
\newcommand{\eP}{\end{Proposition}}

\newcommand{\bLemma}[1]{
\begin{Lemma} \label{L#1} }
\newcommand{\eL}{\end{Lemma}}

\newcommand{\bCorollary}[1]{
\begin{Corollary} \label{C#1} }
\newcommand{\eC}{\end{Corollary}}

\newcommand{\bRemark}[1]{
\begin{Remark} \label{R#1} }
\newcommand{\eR}{\end{Remark}}

\newcommand{\bDefinition}[1]{
\begin{Definition} \label{D#1} }
\newcommand{\eD}{\end{Definition}}

\newcommand{\K}{\mathcal{K}}

\newcommand{\bU}{\bm{U}}

\newcommand{\tvE}{\tilde{E}}

\newcommand{\vrh}{\vr_h}

\newcommand{\vmh}{\bm{m}_h}

\newcommand{\tvm}{\tilde{\bm{m}}}
\newcommand{\bfphi}{\boldsymbol{\varphi}}
\newcommand{\bfPhi}{\boldsymbol{\Phi}}

\newcommand{\Eh}{E_h}

\newcommand{\ds}{\,\mathrm{d}S(x)}

\newcommand{\bFormula}[1]{
\begin{equation} \label{#1}}
\newcommand{\eF}{\end{equation}}

\newcommand{\grid}{\mathcal{T}}

\newcommand{\vuh}{\vu_h}
\newcommand{\mh}{\vc{m}_h}

\newcommand{\intSh}[1] {\int_{\sigma} #1 \ds }

\newcommand{\Ov}[1]{\overline{#1}}
\newcommand{\av}[1]{ \overline{#1}}

\newcommand{\aleq}{\stackrel{<}{\sim}}

\newcommand{\vr}{\varrho}

\newcommand{\tvr}{\widetilde \vr}

\newcommand{\vt}{\vartheta}
\newcommand{\vu}{\bm{u}}
\newcommand{\vm}{\bm{m}}

\newcommand{\vc}[1]{{\bm #1}}

\newcommand{\Div}{{\rm div}_x}
\newcommand{\Grad}{\nabla_x}

\newcommand{\dx}{\,{\rm d} {x}}

\newcommand{\dt}{\,{\rm d} t }

\newcommand{\dxdt}{\dx \ \dt}
\newcommand{\intO}[1]{\int_{\Omega} #1 \ \dx}

\newcommand{\intOh}[1]{\int_{\Omega_h} #1 \ \dx}

\renewcommand{\vm}{\bm{m}}
\renewcommand{\vu}{\bm{u}}

\def\softd{{\leavevmode\setbox1=\hbox{d}%
          \hbox to 1.05\wd1{d\kern-0.4ex{\char039}\hss}}}
\definecolor{Cgrey}{rgb}{0.85,0.85,0.85}
\definecolor{Cblue}{rgb}{0.50,0.85,0.85}
\definecolor{Cred}{rgb}{1,0,0}
\definecolor{fancy}{rgb}{0.10,0.85,0.10}

\newcommand\Cbox[2]{%
    \newbox\contentbox%
    \newbox\bkgdbox%
    \setbox\contentbox\hbox to \hsize{%
        \vtop{
            \kern\columnsep
            \hbox to \hsize{%
                \kern\columnsep%
                \advance\hsize by -2\columnsep%
                \setlength{\textwidth}{\hsize}%
                \vbox{
                    \parskip=\baselineskip
                    \parindent=0bp
                    #2
                }%
                \kern\columnsep%
            }%
            \kern\columnsep%
        }%
    }%
    \setbox\bkgdbox\vbox{
        \color{#1}
        \hrule width  \wd\contentbox %
               height \ht\contentbox %
               depth  \dp\contentbox
        \color{black}
    }%
    \wd\bkgdbox=0bp%
    \vbox{\hbox to \hsize{\box\bkgdbox\box\contentbox}}%
    \vskip\baselineskip%
}

\begin{document}


\title{
Computing oscillatory solutions of the Euler system via $\K$-convergence}

\author{Eduard Feireisl\thanks{The research of E.F. and B.S. leading to these results has received funding from the
Czech Sciences Foundation (GA\v CR), Grant Agreement
18--05974S. The Institute of Mathematics of the Czech Academy of Sciences is supported by RVO:67985840.\newline
\hspace*{1em} $^\spadesuit$M.L. has been funded by the Deutsche Forschungsgemeinschaft (DFG, German Research Foundation) - Project number 233630050 - TRR 146 as well as by  TRR 165 Waves to Weather. \newline
\hspace*{1em}$^\diamondsuit$Y.W.~acknowledges the support provided by President Foundation of
CAEP(YZJJLX2016009). \newline
This research has been also partially supported by the Sino-German Center funded by the
Deutsche Forschungsgemeinschaft and National Natural Science Foundation of China. M.L. and B.S. would like to thank
Jiequan Li and Beijing
Institute of Applied
Physics and Computational Mathematics for their hospitality. }
$^{, \clubsuit}$\and M\' aria Luk\' a\v cov\' a -- Medvi\softd ov\' a$^{\spadesuit}$ \and
Bangwei She$^{*}$ \and Yue Wang$^{\diamondsuit}$
}

\date{}

\maketitle

\bigskip

\centerline{$^*$ Institute of Mathematics of the Czech Academy of Sciences}
\centerline{\v Zitn\' a 25, CZ-115 67 Praha 1, Czech Republic}
\centerline{feireisl@math.cas.cz, she@math.cas.cz}

\bigskip

\centerline{$^\clubsuit$ Institute of Mathematics, TU Berlin}
\centerline{Strasse des 17. Juni, Berlin, Germany}

\bigskip
\centerline{$^\spadesuit$ Institute of Mathematics, Johannes Gutenberg-University Mainz}
\centerline{Staudingerweg 9, 55 128 Mainz, Germany}
\centerline{lukacova@uni-mainz.de}

\bigskip
\centerline{$^\diamondsuit$ Institute of Applied Physics and Computational Mathematics}
\centerline{Fenghao East road 2, Haidian, Beijing, China}
\centerline{wang\_yue@iapcm.ac.cn}

\begin{abstract}

We develop a method to compute effectively the Young measures associated to sequences of
numerical solutions of the compressible Euler system. Our approach is based on the concept of $\mathcal{K}-$convergence
adapted to sequences of parametrized measures. The convergence is strong  in space and time (a.e.~pointwise or in certain $L^q$ spaces) whereas the
measures converge narrowly or in the Wasserstein distance to the corresponding limit.

\end{abstract}


\tableofcontents

\section{Introduction}
\label{I}

The initial--boundary value problems to certain systems of nonlinear conservation laws are ill--posed in the class of weak solutions.
An iconic example is the \emph{Euler system} describing the time evolution of the density $\vr = \vr(t,x)$, the momentum
$\vm = \vm(t,x)$, and the energy $E = E(t,x)$ of a compressible inviscid fluid:
\begin{equation} \label{i1}
\begin{split}
\partial_t \vr + \Div \vm &= 0,\\
\partial_t \vm + \Div \left( \frac{\vm \otimes \vm}{\vr} \right) + \Grad p &= 0,\\
\partial_t E + \Div \left[ \left( E + p \right) \frac{\vm}{\vr} \right] &= 0.
\end{split}
\end{equation}
Here, $p$ is the pressure related to $\vr, \vm, E$ through a suitable equation of state. The fluid is confined to a spatial
domain $\Omega \subset R^d$, $d = 1,2,3$, on the boundary of which we impose the impermeability condition
\begin{equation} \label{i2}
\vm \cdot \vc{n}|_{\partial \Omega} = 0.
\end{equation}
The initial state of the system is given through initial conditions
\begin{equation} \label{i3}
\vr(0, \cdot) = \vr_0, \ \vm(0, \cdot) = \vm_0, \ E(0, \cdot) = {E}_0.
\end{equation}

The first numerical evidence that  indicated ill--posedness of the Euler system was presented  by Elling \cite{Elling}.  As proved later rigorously \cite{Chiod, ChiDelKre, ChiKre, ChKrMaSwI, dlsz2, FKKM17}, the initial--boundary value problem \eqref{i1}--\eqref{i3} admits infinitely many weak solutions
on a given time interval $(0,T)$
for a rather vast class of initial data. Moreover, these solutions obey the standard entropy admissibility condition
\begin{equation} \label{i4}
\partial_t S + \Div \left( S \frac{\vm}{\vr} \right) \geq 0,
\end{equation}
where $S$ is the total entropy of the system.

The ill--posedness of the Euler system is intimately related to the lack of compactness of the set of functions satisfying
\eqref{i1} and \eqref{i4}.
Indeed bounded sequences of solutions may develop uncontrollable oscillations and/or concentrations. This
phenomenon is then naturally inherited by their numerical approximations, see e.g. Fjordholm, Mishra, and Tadmor \cite{FjKaMiTa, FjMiTa1}. This fact
motivated the renewed interest in the measure--valued (MV) solutions introduced in the context of the incompressible Euler system
by DiPerna and Majda \cite{Diperna_Majda}. The exact values of physical densities and fluxes, originally numerical functions of the
physical variables $(t,x)$, are replaced by a family of \emph{probability measures} (Young measure)
\[
(t,x) \mapsto \mathcal{V}_{t,x} \in \mathcal{P} \left( [\vr, \vm, S,E] \in R^{d + 3} \right).
\]
The values of the observable fields $[\vr, \vm, S, E]$ are interpreted as the expected values:
\[
\vr(t,x) = \left< \mathcal{V}_{t,x}; \tvr \right>,\ \vm(t,x) = \left< \mathcal{V}_{t,x}; \tvm \right>,\
S(t,x) = \left< \mathcal{V}_{t,x}; \widetilde{S} \right>,\
E(t,x) = \left< \mathcal{V}_{t,x}; \tvE \right>.
\]

In view of the observed fact that certain numerical schemes fail to provide convergent sequences
of approximate solutions, Fjordholm et al. \cite{FjKaMiTa, FjMiTa1} proposed a way how to compute the associated Young measure.
Note that, in the context of numerical simulations, the convergence towards the limit (exact) solution is required to be pointwise (a.e.). On the other hand, however, the theoretical studies available so far in the literature, see e.g.  \cite{FjKaMiTa, FjMiTa1}, provide only weak-$(^*)$ convergence in $(t,x)$.
This is of little practical interest as wildly oscillating output data are difficult to
analyze in such a way.

In the present paper, we propose a new method to compute the Young measures associated to sequences of numerical solutions
based on the concept of $\mathcal{K}$-convergence. We start by showing a remarkable property of the Euler system \eqref{i1} that
can be roughly stated as follows: either a consistent numerical approximation converges strongly (pointwise a.e.) or
its (weak) limit \emph{is not} a weak solution of the Euler system, see Section \ref{w}. In the context of numerical analysis, this property might be seen as a sharp version
of the well-known \emph{Lax-Wendroff theorem}. In view of these facts, the concept of measure--valued
or other kind of  generalized solution is necessary whenever the approximate sequence exhibits oscillations.

Following Balder \cite{Bald}, we associate to an approximate sequence $\{ \vrh, \vmh, S_h, \Eh \}_{h \searrow 0}$ the Young measure
\[
\mathcal{V}^h_{t,x} = \delta_{ [\vrh(t,x), \vmh(t,x), S_h, \Eh(t,x)] },
\]
where $\delta_Y$ is the Dirac mass concentrated at $Y$. As already pointed out, \emph{observable} limits of $\mathcal{V}^h_{t,x}$ for
$h \to 0$ must be given in terms of pointwise convergence. Unfortunately, the standard basic result of the theory of Young
measures, see, e.g., Ball \cite{Ball2} or Pedregal \cite{PED1}, provides only the weak-(*) convergence (up to a suitable subsequence):
\begin{equation} \label{i5}
\mathcal{V}^h \to \mathcal{V} \ \mbox{weakly-(*) in}\ L^\infty(Q; \mathcal{P}(R^{d + 3})), \ Q \equiv (0,T) \times \Omega.
\end{equation}
Thus, similarly to the approximate solutions $[\vr_h, \vm_h, S_h, \Eh]$, the family $\{ \mathcal{V}^h \}_{h \searrow 0}$ may exhibit oscillations
with respect to the physical space $Q$. From the practical point of view, therefore, the piece of information provided by the convergence \eqref{i5} is negligible.

The concept of $\mathcal{K}$-convergence, developed in the framework of Young measures by Balder \cite{Bald, Bald1}, converts the weak-(*) convergence to the desired pointwise one by replacing sequences by (suitable) Ces\`aro averages in the spirit of the original work by
Koml\'os \cite{Kom}:

{\it A sequence $\{ f_{h} \}_{h  \searrow 0}$ of functions uniformly bounded in $L^1$ contains a subsequence $\{f_{h'}\}_{h'  \searrow 0}$ such that}
\[
\frac{1}{N} \sum_{n = 1}^N f_{h'_n} \to f \ \mbox{a.e. as}\ N \to \infty.
\]
{\it Any further subsequence of $ \{f_{h'}\} $ enjoys the same property.}

\noindent Rephrased in terms of sequences of probability measures, Balder \cite{Bald1} obtained a remarkable extension of the Prokhorov theorem:

{\it If a parametrized family of probability measures $\{ \mathcal{V}^h_{t,x} \}_{h > 0}$ is tight, meaning,}
\[
\frac{1}{|Q|} \int_Q \mathcal{V}^h_{t,x} \ \dxdt \ \mbox{is tight},
\]
{\it then, for a suitable subsequence,}
\begin{equation} \label{nar}
\frac{1}{N} \sum_{n=1}^N \mathcal{V}^{h_n}_{t,x} \to \mathcal{V}_{t,x} \ \mbox{narrowly in}\ \mathcal{P}(R^{d + 3})
\ \mbox{as}\ N \to \infty \ \mbox{for a.e.}\ (t,x) \in Q.
\end{equation}

\noindent We recall that  a sequence of probability measures $\{ \mathcal{V}_n \}_{n=1}^\infty$ \emph{converges narrowly} to a probability measure $\mathcal{V}$ if
\[
\left \langle \mathcal{V}_n; g \right \rangle \to \langle \mathcal{V}; g \rangle \ \ \ \mbox{as} \  n\to \infty\  \mbox{for any }\ g \in BC(R^{d+3}).
\]
Here the symbol $BC(R^{d+3})$ denotes the set of bounded continuous functions.

The aim of the present paper is to apply these ideas to sequences of numerical solutions of the Euler system
\eqref{i1}--\eqref{i3}.  In addition to \eqref{nar}, we show that the available stability estimates yield convergence in the standard Wasserstein metric
\begin{equation} \label{i6}
W_1 \left( \frac{1}{N} \sum_{n=1}^N \mathcal{V}^{h_n}_{t,x}; \mathcal{V}_{t,x} \right)  \to 0
\ \mbox{as}\ N \to \infty \
\mbox{ for a.a. } (t,x) \in Q.
\end{equation}

\noindent  Recall that the  Wasserstein distance of $q$-th order of probability measures
$\mathcal{N}, \mathcal{V}$ is defined as
\begin{equation}
\label{Vas_def}
W_q(\mathcal{N}, \mathcal{V}):= \left\{ \inf_{\pi \in \Pi(\mathcal{N}, \mathcal{V})}  \int_{R^{d+3} \times R^{d+3}} | \zeta -\xi |^q d\pi(\zeta,\xi) \right\}^{1/q} \quad q \in [1, \infty),
\end{equation}
where $\Pi(\mathcal{N}, \mathcal{V})$ is the set of probability measures on $R^{d+3} \times R^{d+3}$ with
marginals $\mathcal{N}$ and $\mathcal{V}$.
Indeed, this is a natural metric on the set of the Young measures generated by approximate sequences of numerical
solutions in view of the available stability estimates. In comparison with the standard L\' evy--Prokhorov metric
induced by the weak-(*) topology (on the space of measures), the Wasserstein distance includes a piece of information on large distances
in the associated phase space $R^{d + 3}$.

For scalar conservation laws, the convergence of finite volume methods with respect to the Wasserstein distance has been investigated by Fjordholm and Solem \cite{SF}. They proved that monotone finite volume schemes being formally first order schemes show in special situations second order convergence rate in  $W_1$. On the other hand, the first order rate in $W_1$ has been proven optimal in general, see \cite{S19}. Such a result, however, seems to be out of reach for the Euler system as it would imply
strong (a.e. pointwise) convergence of the expected values -- the numerical solutions.

The paper is organized as follows. In Section~\ref{A} we introduce the concept of consistent approximate solutions.
Section~\ref{w} summarizes available results on the  strong (pointwise) convergence of the approximate solutions and the weak convergence in terms of Young measures. Section~\ref{K} is the heart of the paper. We use the theory of $\K$-convergence
adapted to the Young measures to show that the limiting Young measure can be effectively described by means of the Ces\`aro averages of the approximate solutions. As an added benefit with respect to available results, we therefore obtain strong convergence of the
Ces\`aro averages in $L^q((0,T) \times \Omega)$ and the Wasserstein distance $W_q$  for some $q \geq 1$. Finally, in Section~\ref{Brenner}, we present numerical simulations obtained by two finite volume methods: a recently proposed
finite volume method, see \cite{FLM_18}, that is (entropy)-stable and consistent and thus yields
the consistent approximate solutions required by the abstract theory and a more standard
finite volume method based on the generalized  Riemann problem, see, e.g.~\cite{ben1984, ben2003, ben2006, ben2007, Cheng2019} and the references therein.
The predicted strong convergence of the Ces\` aro averages when approximate solutions experience
oscillations is confirmed by both numerical methods. Our numerical results clearly demonstrate robustness
of the proposed $\mathcal{K}$-convergence technique.

\section{Approximate solutions}
\label{A}

For the sake of simplicity, we focus on the perfect gas with the standard equation of state
\[
p = (\gamma - 1) \vr e,\ \ e = \frac{1}{\gamma - 1} \vt,
\]
where $e$ is the specific internal energy and $\vt$ the absolute temperature. Accordingly,
\[
E = \frac{1}{2} \frac{|\vm|^2}{\vr} + \vr e,\
p = (\gamma - 1) \left( E - \frac{1}{2} \frac{|\vm|^2}{\vr} \right).
\]
Finally, we introduce the specific entropy
\[
s = \log\left(\vt^{\frac{1}{\gamma - 1}}  \right) - \log(\vr),\ \mbox{and the total entropy}\ S = \vr s.
\]

\begin{Definition}[Consistent approximations] \label{AD1}

We say that $[\vrh, \vmh, \Eh]$ is a family of approximate solutions \emph{consistent} with the Euler system \eqref{i1} in $Q = (0,T) \times \Omega$ if:
\begin{itemize}
\item $\vrh$, $\vmh$, $\Eh$ are measurable functions on $Q$,
\[
\vrh > 0, \ p_h \equiv (\gamma - 1) \left( E_h - \frac{1}{2} \frac{|\vm_h|^2}{\vr_h} \right) > 0 \ \mbox{a.e. in}\ Q;
\]

\item

\begin{equation} \label{D1}
\left[ \intO{ \vr_h \varphi } \right]_{t=0}^{t = \tau} =
\int_0^\tau \intO{ \left[ \vr_h \partial_t \varphi + \vc{m}_h \cdot \Grad \varphi \right]} \dt  + \int_0^\tau
e_{1,h} (t, \varphi) \dt
\end{equation}
for any $0 \leq \tau \leq T$, and any $\varphi \in C^1([0,T] \times \Ov{\Omega})$, where
\[
e_{1,h} (\cdot , \varphi) \to 0 \ \mbox{in}\ L^1(0,T) \ \mbox{as}\ h \searrow 0 \ \mbox{for any fixed}\ \varphi;
\]
\item
\begin{equation} \label{D2}
\left[ \intO{ \vc{m}_h \cdot \bfphi } \right]_{t=0}^{t = \tau} =
\int_0^\tau \intO{ \left[ \vc{m}_h \cdot \partial_t \bfphi + \frac{\vc{m}_h \otimes \vc{m}_h} {\vr_h} : \Grad \bfphi
+ p_h \Div \bfphi \right]} \dt  + \int_0^\tau
e_{2,h} (t, \bfphi) \dt
\end{equation}
for any $0 \leq \tau \leq T$, and any $\bfphi \in C^1([0,T] \times \Ov{\Omega}; R^d)$, $\bfphi \cdot \vc{n}|_{\Omega} = 0$,
where
\[
e_{2,h} (\cdot , \bfphi) \to 0 \ \mbox{in}\ L^1(0,T) \ \mbox{as}\ h \searrow 0 \ \mbox{for any fixed}\ \bfphi;
\]
\item
\begin{equation} \label{D3}
\intO{ E_h(\tau) } = \intO{ E_{0,h} }\ \mbox{for any}\ 0 \leq \tau \leq T.
\end{equation}
\end{itemize}

\noindent In addition, we say that $[\vrh, \vmh, \Eh]$ is \emph{admissible} if for the entropy $s_h$,
\[
s_h = \log\left(\vt^{\frac{1}{\gamma - 1}}_h  \right) - \log(\vrh),\ \vt_h = (\gamma - 1) \frac{1}{\vrh} \left( E_h - \frac{1}{2} \frac{|\vm_h|^2}{\vr_h} \right),
\]
the entropy inequality holds
\begin{equation} \label{D4}
\left[ \intO{ \vr_h \chi(s_h) \varphi } \right]_{t=0}^{t = \tau} \geq
\int_0^\tau \intO{ \left[ \vr_h \chi (s_h) \partial_t \varphi + \chi(s_h) \vc{m}_h \cdot \Grad \varphi \right]} \dt  + \int_0^\tau
e_{3,h} (t, \varphi) \dt
\end{equation}
for any $0 \leq \tau \leq T$, any $\varphi \in C^1([0,T] \times \Ov{\Omega})$, $\varphi \geq 0$, and any $\chi$,
\[
\chi : R \to R \ \mbox{a non--decreasing concave function}, \ \chi(s) \leq \Ov{\chi} \ \mbox{for all}\ s \in R,
\]
where
\[
e_{3,h} (\cdot , \varphi) \to 0 \ \mbox{in}\ L^1(0,T) \ \mbox{as}\ h \searrow 0 \ \mbox{for any fixed}\ \varphi.
\]

\end{Definition}

A family of approximate solutions provides seemingly less information than the corresponding \emph{weak formulation} of
the Euler system \eqref{i1}--\eqref{i3}, and \eqref{i4}. Nonetheless, as we shall see below, the approximate solutions in the sense
of Definition \ref{AD1} generate a measure--valued solution introduced in \cite{BF}.
In this paper, we focus on the consistent approximate solutions generated by suitable numerical schemes. A specific example of such a scheme -- the finite volume method \eqref{n2}-\eqref{n4} -- is given in Section~\ref{Brenner}. However, the concept presented here is quite general allowing to include the vanishing viscosity as well as other types of singular limit perturbations. In particular, $[\vrh,\vmh, \Eh]$ may be weak solutions of the Euler system itself (the error terms $e_{i,h}, i=1,2,3$ being zero).

Note that, in contrast with the Euler system \eqref{i1}, the approximate solutions satisfy merely the total energy balance \eqref{D3},
meaning the last equation in \eqref{i1} is integrated over the physical domain.
It can be shown, however, that \eqref{D3}, together with the entropy inequality \eqref{D4}, give rise to the energy equation as soon as
the all quantities are smooth and all error terms set to be zero. The proof is the same as for the Navier--Stokes--Fourier
system and we refer the reader to \cite{FLMS_19} for details.

\section{Strong vs. weak convergence}
\label{w}

We discuss sequential stability of the set of approximate solutions introduced in Definition \ref{AD1}. To this end, we suppose that
the corresponding initial data satisfy
\begin{equation} \label{w1}
\begin{split}
\vrh(0, \cdot) &= \vr_{0,h},\
\vmh(0, \cdot) = \vm_{0,h}, \ \Eh (0, h) = E_{0,h},\  \vr_{0,h} > 0, \ E_{0,h} - \frac{1}{2} \frac{|\vm_{0,h}|^2}{\vr_{0,h}} > 0,\\
\intO{ E_{0,h} } &\leq \mathcal{E}_0 \ \mbox{uniformly for}\ h \searrow 0;
\end{split}
\end{equation}
and
\begin{equation} \label{w2}
\begin{split}
s_h(0,\cdot) \equiv \log(\vt_{0,h}^{\frac{1}{\gamma - 1}}) - \log(\vr_{0,h}) \geq {\underline{s}} > - \infty
\ &\mbox{uniformly for}\ h \searrow 0,\\
&\mbox{where} \ \vt_{0,h} \equiv (\gamma - 1) \frac{1}{\vr_{0,h}} \left( E_{0,h} - \frac{1}{2} \frac{|\vm_{0,h}|^2}{\vr_{0,h}}
\right).
\end{split}
\end{equation}

The uniform bound \eqref{w1} imposed on the initial energy, together with the total energy balance \eqref{D3}, yield
\begin{equation} \label{w3}
\intO{ E_h(\tau, \cdot) } \leq \mathcal{E}_0 \ \mbox{uniformly for}\ h \searrow 0,\ 0 \leq \tau \leq T,
\end{equation}
in particular,
\begin{equation} \label{w4}
\intO{ \frac{1}{2} \frac{|\vm_h|^2}{\vr_h} (\tau, \cdot) } \leq \mathcal{E}_0.
\end{equation}

Next, we \emph{suppose}, in accordance with the fact that the entropy is transported along streamlines,
that
\begin{equation} \label{w5}
s_h (\tau, \cdot) \geq \underline{s} \ \mbox{for all}\ 0 \leq \tau \leq T.
\end{equation}

\begin{Remark} \label{wR1}

Note that \eqref{w5} actually follows from \eqref{w2} as soon as the error term in \eqref{D4} vanishes or, alternatively,
\[
e_{3,h}(t, \psi ) \equiv 0 \ \mbox{for any}\ \psi = \psi(t),
\]
see \cite[Section~4.2]{FLM_18} for details.

\end{Remark}

Anticipating \eqref{w5} we may rewrite the total energy in terms of $\vr_h, \vm_h, s_h$ as
\[
E_h = \frac{1}{2} \frac{|\vm_h|^2}{\vr_h} + \frac{1}{\gamma - 1} \vrh^\gamma \exp \Big[ (\gamma - 1) s_h \Big].
\]
In view of \eqref{w3}--\eqref{w5} we obtain
\begin{equation} \label{w6}
\intO{ \vrh^\gamma (\tau, \cdot) } \aleq \mathcal{E}_0 \ \mbox{uniformly in}\ \tau, \ h,
\end{equation}
and
\begin{equation} \label{w7}
\intO{ \vmh^{\frac{2 \gamma}{\gamma + 1}} (\tau, \cdot) } \aleq \mathcal{E}_0 \ \mbox{uniformly in}\ \tau, \ h.
\end{equation}
Finally, introducing the total entropy $S_h \equiv \vr_h (s_h - \underline{s}) \geq 0$, we deduce
\begin{equation} \label{w8}
\intO{ S_h^\gamma (\tau, \cdot) } \aleq \mathcal{E}_0 \ \mbox{uniformly in}\ \tau, \ h,
\end{equation}
and
\begin{equation} \label{w9}
\intO{ \left( \frac{S_h^2}{\vr_h} \right)^\gamma (\tau, \cdot) } \aleq \mathcal{E}_0 \ \mbox{uniformly in}\ \tau, \ h,
\end{equation}
see \cite[Section 3.2]{BFH_19} for details.

\subsection{Strong (pointwise) convergence}

Note that all the uniform estimates established so far yield only boundedness of the approximate solutions in various $L^p$-spaces.
The appropriate notion of convergence is therefore ``{weak}'', or, more precisely, ``{weak-(*)}'' convergence, meaning convergence in the sense of integral averages. As already pointed out, this is of little practical importance as the limits of oscillating quantities
are difficult to identify.

However,
the convergence is strong (pointwise a.e.) as soon as the limit Euler system \eqref{i1}--\eqref{i3} admits a (unique) strong solution.
Indeed any \emph{admissible} sequence of consistent approximate solutions $[\vrh, \vmh, \Eh]$ generates a dissipative measure--valued
(DMV) solution in the sense of \cite{BF}. The (DMV) solutions enjoy the weak--strong uniqueness property yielding the desired conclusion.
The relevant result can be stated as follows, cf.~\cite[Theorem 3.3]{BF}:

\begin{Proposition} \label{wP1}

Let $\Omega \subset R^d$, $d=1,2,3$ be a bounded Lipschitz domains. Let $\{ \vrh, \vmh, \Eh \}_{h \searrow 0}$ be a
family of admissible approximate solutions consistent with the Euler system in the sense of Definition \ref{AD1}. Let the approximate
initial data satisfy \eqref{w1}, \eqref{w2}, and suppose that \eqref{w5} holds. Let
\[
\begin{split}
\vr_{0,h} &\to \vr_0 > 0, \ \vm_{0,h} \to \vm_0, \ E_{0,h} \to E_0 > 0,\\
s_h(0, \cdot) &\equiv  \frac{1}{\gamma - 1} \log \left( \frac{\gamma - 1}{\vr_{0,h}} \left(
E_{0,h} - \frac{1}{2} \frac{|\vm_{0,h}|^2}{\vr_{0,h}} \right)     \right)- \log(\vr_{0,h})  \to s_0 \ \mbox{weakly in}\ L^1(\Omega)  \mbox{ as } h\searrow 0,
\end{split}
\]
where
\[
s_0 = \frac{1}{\gamma - 1} \log \left( \frac{\gamma - 1}{\vr_{0}} \left(
E_{0} - \frac{1}{2} \frac{|\vm_{0}|^2}{\vr_{0}} \right)     \right)- \log(\vr_{0}).
\]

Finally, suppose that the Euler system \eqref{i1}, \eqref{i2}, with the initial data \eqref{i3}, admits a strong solution
$[\vr, \vm, E]$ Lipschitz continuous in $[0,T] \times \Ov{\Omega}$.

Then
\[
\vrh \to \vr,\ \vmh \to \vm,\ \Eh \to E \ \mbox{in}\ L^1((0,T) \times \Omega) \ \mbox{as}\ h \searrow 0.
\]

\end{Proposition}

Note that the convergence stated in Proposition \ref{wP1} is unconditional, meaning it is not necessary to consider subsequences, as the limit solution is unique.

\subsection{Weak convergence}

If the limit system fails to possess a classical solution, the convergence might not be strong. In such a case, the nonlinear composition
operators do not commute with weak limits and the limit object may not be even a weak solution of the Euler system.

To analyze the weak convergence, it is convenient to work with the variables $[\vrh, \vmh, S_h]$. In view of the uniform bounds
\eqref{w6}--\eqref{w8} we obtain that
\begin{equation} \label{w10}
\begin{split}
\vrh &\to \vr \ \mbox{weakly-(*) in}\ L^\infty(0,T; L^\gamma(\Omega)),\\
\vmh &\to \vm \ \mbox{weakly-(*) in}\ L^\infty(0,T; L^{\frac{2 \gamma}{\gamma + 1}}(\Omega; R^d)),\\
S_h &\to S \ \mbox{weakly-(*) in}\ L^\infty(0,T; L^\gamma(\Omega))
\end{split}
\end{equation}
passing to suitable subsequences as the case may be.
The key quantity is the total energy
\[
\Eh = \frac{1}{2} \frac{|\vm_h|^2}{\vr_h} + \frac{a}{\gamma - 1}  \vr_h^\gamma \exp \left( (\gamma - 1) \frac{S_h}{\vr_h} \right),
\]
where $a = \exp((\gamma-1)\underline{s})$.
Note that
\[
[\vr, \vm, S] \mapsto E(\vr, \vm, S) \equiv \left\{
\begin{array}{l} \frac{1}{2} \frac{|\vm|^2}{\vr} + \frac{a}{\gamma - 1} \vr^\gamma \exp \left( (\gamma - 1) \frac{S}{\vr} \right),\
\ \mbox{if} \ \vr > 0 \\
0,\ \ \mbox{if}\ \vr = |\vm| =  0 \mbox { and } S\leq 0 \\
\infty,\ \ \mbox{otherwise}
\end{array} \right.
\]
is a convex lower semicontinuous function on $R^{d+2}$,
smooth and strictly convex for
$\vr > 0$ and $S$, see \cite[Lemma~3.1]{BFH_19}. Considering a subsequence if necessary, we conclude that
$$
\{ \vr_h, \vm_h, S_h,  E_h\}_{h > 0} \ \mbox{generates a Young measure}\ \mathcal{V}_{t,x} \in \mathcal{P} (R^{d+3}),\ (t,x) \in Q.
$$
In addition we introduce a \emph{marginal} $\nu_{t,x} \in \mathcal{P}(R^{d+2})$,
such that
\[
\langle \nu_{t,x}, g(\widetilde \vr_h, \widetilde \vm_h, \widetilde S_h )\rangle
\equiv \langle \mathcal{V}_{t,x}, g(\widetilde \vr_h, \widetilde \vm_h, \widetilde S_h) 1_{R} (\widetilde{E}) )\rangle
\
\mbox{for all} \ g \in BC(R^{d+2}).
\]
Moreover, we have
\[
\begin{split}
& E_h \to E \ \mbox{weakly-(*) in}\ L^\infty(0,T; \mathcal{M}^+(\Ov{\Omega})),\\
&\frac{1}{2} \frac{|\vm|^2}{\vr} + \frac{ a}{\gamma - 1} \vr^\gamma \exp \left( (\gamma - 1) \frac{S}{\vr} \right)
\leq \left< \nu; \frac{1}{2} \frac{|\widetilde{\vm}|^2}{\widetilde{\vr}} + \frac{{a}}{\gamma - 1} {\widetilde{\vr}}^\gamma \exp \left( (\gamma - 1) \frac{\widetilde{S}}{\widetilde{\vr}} \right)
\right> \leq E,\\
\end{split}
\]
where the first inequality is Jensen's inequality, while the second follows from
from the arguments of \cite[Lemma~2.1]{FGSGW_16} applied to $F=0$,  $G=E$ and $Z = (\widetilde \vr, \widetilde \vm, \widetilde S).$
Using the energy conservation \eqref{D3} we obtain
\[
\begin{split}
\int_{\Ov{\Omega}}  E(\tau) &= \lim_{h \to 0} \intO{ \left[ \frac{1}{2} \frac{|\vm_{0,h}|^2}{\vr_{0,h}} + \frac{{ a}}{\gamma - 1} \vr_{0,h}^\gamma \exp \left( (\gamma - 1) \frac{S_{0,h}}{\vr_{0,h}} \right) \right]} \ \mbox{for a.a.}\ \tau \in [0,T].
\end{split}
\]

\noindent We proceed by introducing the oscillation defect
\[
\begin{split}
{\bf osc}&[\vrh, \vmh, S_h] \\ &\equiv \left< \nu; \frac{1}{2} \frac{|\widetilde{\vm}|^2}{\widetilde{\vr}} + \frac{ { a}}{\gamma - 1} {\widetilde{\vr}}^\gamma \exp \left( (\gamma - 1) \frac{\widetilde{S}}{\widetilde{\vr}} \right)
\right> - \left[ \frac{1}{2} \frac{|\vm|^2}{\vr} + \frac{{ a}}{\gamma - 1} \vr^\gamma \exp \left( (\gamma - 1) \frac{S}{\vr} \right) \right] \geq 0;
\end{split}
\]
the concentration defect
\[
{\bf conc}[\vrh, \vmh, S_h] \equiv E - \left< \nu; \frac{1}{2} \frac{|\widetilde{\vm}|^2}{\widetilde{\vr}} + \frac{{ a}}{\gamma - 1} {\widetilde{\vr}}^\gamma \exp \left( (\gamma - 1) \frac{\widetilde{S}}{\widetilde{\vr}} \right)
\right> \geq 0;
\]
and the total energy defect
\[
{\bf edef}[\vrh, \vmh, S_h] = {\bf osc}[\vrh, \vmh, S_h] + {\bf conc}[\vrh, \vmh, S_h]
=
E - \left[ \frac{1}{2} \frac{|\vm|^2}{\vr} + \frac{ { a}}{\gamma - 1} \vr^\gamma \exp \left( (\gamma - 1) \frac{S}{\vr} \right) \right].
\]
Note that ${\bf osc}[\vrh, \vmh, S_h] \in L^\infty(0,T; L^1(\Omega))$, while
\[
{\bf conc}[\vrh, \vmh, S_h],\ {\bf edef}[\vrh, \vmh, S_h] \in L^\infty(0,T; \mathcal{M}^+ (\Ov{\Omega})).
\]

\noindent The following observations are standard:
\[
\begin{split}
{\bf osc}[\vrh, \vmh, S_h]|_B = 0 \ &\Leftrightarrow \ \vrh \to \vr,\ \vmh \to \vm, \ S_h \to S
\ \mbox{a.e. in}\ B \ \mbox{(up to a subsequence)};
\\
{\bf conc}[\vrh, \vmh, S_h]|_{\Ov{B}} = 0\ &\Leftrightarrow\
E_h \to E \ \mbox{weakly in}\ L^1(B);\\
{\bf edef}[\vrh, \vmh, S_h]|_{\Ov{B}} = 0 \ &\Leftrightarrow \ \vrh \to \vr,\ \vmh \to \vm, \ S_h \to S,\
E_h \equiv E(\vrh, \vm_h, S_h) \to E
\ \mbox{in}\ L^1(B)
\end{split}
\]
for any Borel set $B \subset [0,T] \times \Ov{\Omega}$.

Finally, we report the following result, see Chaudhuri \cite{NC_19}, which can be obtained in analogously 
way as in \cite{FH} due to the convexity of pressure.

\begin{Proposition} \label{wT1}
Let $\Omega \subset R^d$, $d=1,2,3$ be a bounded Lipschitz domain. Let $\{ \vrh, \vmh, \Eh \}_{h  \searrow0}$ be a
family of admissible approximate solutions consistent with the Euler system in the sense of Definition \ref{AD1}
such that
\[
\begin{split}
\vrh &\to \vr \ \mbox{weakly-(*) in}\ L^\infty(0,T; L^\gamma(\Omega)),\\
\vmh &\to \vm \ \mbox{weakly-(*) in}\ L^\infty(0,T; L^{\frac{2 \gamma}{\gamma + 1}}(\Omega; R^d)),\\
S_h &\to S \ \mbox{weakly-(*) in}\ L^\infty(0,T; L^\gamma(\Omega)),
\end{split}
\]
and
\[
E_h \equiv \frac{1}{2} \frac{|\vm_h|^2}{\vr_h} + \frac{{ a}}{\gamma - 1} \vr_h^\gamma \exp \left( (\gamma - 1) \frac{S_h}{\vr_h} \right)
\to E \ \mbox{in}\ L^\infty(0,T; \mathcal{M}^+(\Ov{\Omega})).
\]
Suppose that there is an open neighborhood ${U}$ of $\partial \Omega$ such that
\[
{\bf edef}[\vrh, \vmh, S_h]|_{[0,T] \times \Ov{U}} = 0.
\]
Finally, suppose that the limit $[\vr, \vm, S]$ is a weak solution of the Euler system in $\mathcal{D}'((0,T) \times \Omega)$, in particular,
\[
\partial_t \vm + \Div \left( \frac{\vm \otimes \vm}{\vr} \right) + \Grad p(\vr, S) = 0
\ \mbox{in}\ \mathcal{D}'((0,T) \times \Omega).
\]
Then
\[
{\bf edef}[\vrh, \vmh, S_h] \equiv 0,
\]
and 
\[
\vr_h \to \vr, \ \vm_h \to \vm,\ S_h \to S, \ E_h \to E \ \mbox{in}\ L^1((0,T) \times \Omega).
\]
\end{Proposition}

Proposition \ref{wT1} asserts that as long as the convergence of approximate solutions is strong in a neighborhood of the boundary and the limit is a weak solution of the Euler system, then the convergence must be strong everywhere. A very rough extrapolation
might be that either the convergence is strong or the limit object \emph{is not} a weak solution of the limit system.
Such a result might be seen as a sharp version of the celebrated Lax-Wendroff theorem \cite{LW_60}  which states that a bounded sequence of pointwisely convergent numerical solutions, that are
generated by a consistent, conservative and entropy stable numerical scheme, converges to a weak entropy solution.

\section{$\mathcal{K}$-convergence of Young measures}
\label{K}

Accepting the conclusion of Proposition \ref{wT1} as an evidence that oscillating (weakly converging) sequences of approximate solutions
may give rise to ``truly'' measure--valued solutions, we might want to compute the distribution of the associated Young measure.
Unfortunately, the method proposed in \cite{FjKaMiTa,FjMiTa1} asserts only the weak-(*) convergence of these objects
featuring the same difficulties to be captured numerically as oscillating solutions. Here, we propose  \emph{a new method} based on
the concept of \emph{$\mathcal{K}$-convergence} developed in the context of Young measures by Balder \cite{Bald, Bald1}.

Evoking the situation described in Section \ref{w}, we consider a Young measure $\mathcal{V}_{t,x}$ generated by a sequence of
admissible approximate solutions $[\vr_h, \vm_h, S_h, E_h]$ consistent with the Euler system.
Note that we have deliberately included
the conservative variables in order to capture correctly shock positions and to provide
a direct comparison with the results \cite{FjKaMiTa,FjMiTa1}. The Young measure adapted variant of the celebrated
Prokhorov theorem proved by Balder \cite{Bald1}, \cite[Theorem 3.15]{Bald} asserts the existence of a suitable subsequence such that
\begin{equation} \label{K1}
\begin{split}
\frac{1}{N} \sum_{n=1}^N \delta_{[\vr_{h_n}(t,x), \vm_{h_n}(t,x), S_{h_n}(t,x), E_{h_n}(t,x)]} \to
\mathcal{V}_{t,x} \  \mbox{as}\ N \to \infty\  &\mbox{narrowly in}\ \mathcal{P}( R^{d+3})\\
&\mbox{for a.a.}\ (t,x) \in Q.
\end{split}
\end{equation}
Here and hereafter, the symbol $\delta_Y$ denotes the Dirac mass supported at the point $Y$.
Moreover, in view of Castaign, de Fitte, and Valadier \cite[Lemma 6.5.17]{CFV} and the Koml\'os theorem \cite{Kom},
the above sequence can be chosen in such a way that the barycenters converge:
\begin{equation} \label{K2}
\begin{split}
\frac{1}{N} \sum_{n=1}^N \vr_{h_n}(t,x) &\to \vr(t,x) \equiv \left< \mathcal{V}_{t,x}; \widetilde{\vr} \right>, \qquad \ \
\\
\frac{1}{N} \sum_{n=1}^N \vm_{h_n}(t,x) &\to \vm(t,x) \equiv \left< \mathcal{V}_{t,x}; \widetilde{\vm} \right>, \qquad
\\
\frac{1}{N} \sum_{n=1}^N S_{h_n}(t,x) &\to S(t,x) \equiv \left< \mathcal{V}_{t,x}; \widetilde{S} \right>, \qquad
\\
\frac{1}{N} \sum_{n=1}^N E_{h_n}(t,x) &\to \Ov{E}(t,x) \equiv \left< \mathcal{V}_{t,x}; \widetilde{E} \right>
\end{split}
\end{equation}
for a.a.~$(t,x) \in Q$.
Finally, in view of \eqref{w10} and the
standard Banach--Sacks theorem, the convergence of the density, momentum and entropy can be strengthened to
\begin{equation} \label{K3}
\begin{split}
\frac{1}{N} \sum_{n=1}^N \vr_{h_n} &\to \vr \ \mbox{in}\ L^\gamma((0,T) \times \Omega)\\
\frac{1}{N} \sum_{n=1}^N \vm_{h_n} &\to \vm \ \mbox{in}\ L^{\frac{2 \gamma}{\gamma + 1}}((0,T) \times \Omega; R^d) \\
\frac{1}{N} \sum_{n=1}^N S_{h_n} &\to S \ \mbox{in}\ L^\gamma((0,T) \times \Omega).
\end{split}
\end{equation}

As observed in Section \ref{w}, we have $\Ov{E} \leq E$, where
$E$ is the weak-(*) limit of $E_{h_n}$ in $L^\infty(0,T; \mathcal{M}^+(\Ov{\Omega}))$.
Moreover, from the Fatou-Vitali theorem for the Young measures~\cite[Theorem~3.13]{Bald}
we also deduce
\begin{equation} \label{Fatou}
\int_0^T \int_\Omega  \langle  \nu_{t,x} , E(\widetilde \vr, \widetilde \vm, \widetilde S)(t,x)\rangle
\mbox{d} x \mbox{d} t
\leq \liminf_{N\to \infty}\int_0^T \int_\Omega \frac{1}{N} \sum_{n=1}^N E_{h_n}(t,x)  \mbox{d} x \mbox{d} t.
\end{equation}

\noindent As a consequence of \eqref{K2}--\eqref{Fatou}, we can extend \eqref{K1} as follows:

\begin{Lemma}~\label{Vas1}
Let
$$
\overline{\mathcal{V}}^N_{t,x} \equiv \frac{1}{N} \sum_{n=1}^N \delta_{[\vr_{h_n}, \vm_{h_n}, S_{h_n}, E_{h_n}](t,x)} \in \mathcal{P}(R^{d+3})
$$
be the family of Ces\` aro averages of the Young measures in \eqref{K1}. \\

\noindent Then
\begin{equation}
\langle  \overline{\mathcal{V}}^N_{t,x}, g  (\widetilde{\vr}, \widetilde{\vm}, \widetilde{S})\rangle \to
 \langle \mathcal{V}_{t,x}, g  (\widetilde{\vr}, \widetilde{\vm}, \widetilde{S})\rangle \ \mbox{as}\ N \to \infty
\ \mbox{for a.a.}\ (t,x) \in Q
\end{equation}
for any function $g \in C(R^{d+2}),$
$$
| g(z) | \leq 1 + |z|^q, \quad  1 \leq q  <  \frac{2 \gamma}{\gamma + 1}.
$$
\end{Lemma}
\begin{proof}
Decomposing $g$ into positive and negative part,  $ g = g^+ + g^-$ we may assume, without loss of generality, that $g \geq 0.$
Let $T_K$ be a family of cut-off functions,
$$
T_K (z) = \min (z,K),\ K \geq 0.
$$
We write
$$
\langle \overline{\mathcal{V}}^N_{t,x} , g(\widetilde \vr, \widetilde \vm, \widetilde S) \rangle
= \left \langle  \overline{\mathcal{V}}^N_{t,x}, T_k\left(g(\widetilde \vr, \widetilde \vm, \widetilde S) \right) \right \rangle
+ \left\langle  \overline{\mathcal{V}}^N_{t,x}, (Id - T_K) \left(g(\widetilde \vr, \widetilde \vm, \widetilde S) \right) \right \rangle.
$$
In virtue of \eqref{K1} we have
$$
\left \langle  \overline{\mathcal{V}}^N_{t,x}, T_K\left(g(\widetilde \vr, \widetilde \vm, \widetilde S) \right) \right \rangle \to \left \langle  \mathcal{V}_{t,x}, T_K\left(g(\widetilde \vr, \widetilde \vm, \widetilde S) \right) \right \rangle
\ \mbox{for a.a.}\ (t,x) \in Q.
$$
On the other hand,
\begin{eqnarray*}
&&\left |\left\langle  \overline{\mathcal{V}}^N_{t,x}, (Id - T_K) \left(g(\widetilde \vr, \widetilde \vm, \widetilde S) \right) \right \rangle \right | \leq
\left \langle \overline{\mathcal{V}}^N_{t,x}, \chi_{g(\widetilde \vr, \widetilde \vm, \widetilde S)\geq K}
g(\widetilde \vr, \widetilde \vm, \widetilde S) \right \rangle   \leq\\
&& \left \langle \overline{\mathcal{V}}^N_{t,x}, \chi_{g(\widetilde \vr, \widetilde \vm, \widetilde S) \geq K}
 \frac {1}{g(\widetilde \vr, \widetilde \vm, \widetilde S)^{\alpha }} g(\widetilde \vr, \widetilde \vm, \widetilde S)^{1 + \alpha}  \right   \rangle   \leq
\frac{1}{K^\alpha} \left \langle \overline{\mathcal{V}}^N_{t,x} ,  g(\widetilde \vr, \widetilde \vm, \widetilde S)^{1+\alpha}   \right  \rangle.
\end{eqnarray*}
Seeing that
$$
\widetilde{\vr}^\gamma + \widetilde{S}^\gamma + |\widetilde{\vm}|^{\frac{2 \gamma}{\gamma + 1}}
\aleq \frac 1 2 \frac{|\widetilde{\vm}|^2 }{\widetilde{\vr}} +  \frac{1}{\gamma - 1} \widetilde \vr^\gamma \exp \left( (\gamma - 1) \frac{\widetilde S}{\widetilde \vr} \right)
$$
we may infer there exists
$\alpha > 0$ small enough so that
$$
\left \langle  \overline{\mathcal{V}}^N_{t,x},  g(\widetilde \vr, \widetilde \vm, \widetilde S)^{1+\alpha}   \right  \rangle \aleq  \left \langle  \overline{\mathcal{V}}^N_{t,x},  E (\widetilde \vr, \widetilde \vm, \widetilde S)  \right  \rangle
= \frac{1}{N} \sum_{n = 1}^N E_{h_n} (t,x) \ \mbox{for a.a.}\ (t,x) \in Q.
$$
Consequently  using the convergence stated in \eqref{K2} and the
Levy monotone convergence theorem we may pass to the limit $K \to \infty$ obtaining the
desired conclusion.
\end{proof}

\subsection{Strong convergence in the Wasserstein distance}
\label{Vas_sec}

Relation \eqref{K1} asserts narrow convergence of the Ces\`aro averages $\overline{\mathcal{V}}^N$ to the Young measure
$\mathcal{V}$ that is (a.e.) pointwise with respect to the parameter $(t,x)$ in the physical space $Q$. It is legitimate to ask
whether this can be strengthened to the convergence in the Wasserstein metric. We first observe that
\begin{equation} \label{Vas2}
W_1\left( \overline{\mathcal{V}}^N_{t,x}, \mathcal{V}_{t,x}\right) \to 0 \
\ \mbox{as}\ N \to \infty \ \mbox{ for a.a. }  (t,x) \in Q.
\end{equation}
Indeed, as the narrow convergence has been established in \eqref{K1}, relation \eqref{Vas2} follows as soon as we observe
the convergence of first moments for  a.a. $(t,x) \in Q$, see Villani \cite[Definition~6.8, Theorem~6.9]{Villani}:
\begin{equation} \label{Vil1}
\langle \overline{\mathcal{V}}^N_{t,x}, |\widetilde{\vr}| + |\widetilde{\vm}| + |\widetilde S| + |\widetilde E| \rangle \equiv
\frac{1}{N} \sum_{h_n = 1}^N\left( \vr_{h_n} + |\vm_{h_n}| + S_{h_n} + E_{h_n}  \right)
\to \langle \mathcal{V}_{t,x}, \widetilde \vr + |\widetilde \vm| + \widetilde S + \widetilde E \rangle.
\end{equation}
However, relation \eqref{Vil1} follows directly from Lemma~\ref{Vas1} and the last statement in \eqref{K2}.

Our next goal is to investigate the convergence of the Ces\` aro sums of the marginals,
\[
\overline{\nu}^N_{t,x} \equiv \frac{1}{N} \sum_{n=1}^N \nu_{h_n} \equiv
\frac{1}{N} \sum_{n=1}^N \delta_{[\vr_{h_n}, \vm_{h_n}, S_{h_n} ](t,x)} \in \mathcal{P}(R^{d +2}).
\]
Specifically, we show that
\begin{equation} \label{Vas3}
\left \| W_q\left(\overline{\nu}^N, \nu\right) \right \|_{L^q(Q)} \to 0  \ \ \mbox{ as } N\to \infty \ \mbox{ for some }\ q  \geq 1.
\end{equation}
Using the definition of the Wasserstein distance we have the following inequality
\begin{eqnarray}
&&|W_q (\overline{\nu}_{t,x}^N,\nu_{t,x} )| \leq \langle  \overline{\nu}_{t,x}^N, |\widetilde{\vr}|^q + | \widetilde{\vm} |^q
+ |\widetilde S|^q \rangle + \langle  \nu_{t,x}, |\widetilde{\vr}|^q + | \widetilde{\vm} |^q
+ |\widetilde S|^q \rangle.
\end{eqnarray}
Now, the two terms on the right--hand side are bounded in $L^q(Q)$, $1 \leq q \leq \frac{2 \gamma}{\gamma + 1} $ due to
the convergences established in \eqref{K2} and \eqref{Fatou}.
Combining this with \eqref{Vas2} we conclude that
\begin{equation} \label{Vas3}
\left \| W_q\left(\overline{\nu}^N, \nu\right) \right \|_{L^q(Q)} \to 0   \ \ \mbox{ as } N\to \infty \
\mbox{for any}  \  1 \leq q < \frac{2 \gamma}{\gamma + 1}.
\end{equation}

Summing up the previous discussion, we can state the main result of the present paper.

\begin{Theorem} \label{Tmain}

Let $\Omega \subset R^d$, $d=1,2,3$ be a bounded Lipschitz domain. Let $\{ \vrh, \vmh, S_h, \Eh \}_{h \searrow 0}$ be a
family of admissible approximate solutions consistent with the Euler system in the sense of Definition \ref{AD1}. Let the approximate
initial data satisfy \eqref{w1}, \eqref{w2}, and let \eqref{w5} hold.

Then there exists a subsequence $h_n \to 0$ such that the following hold:
\begin{itemize}
\item
the sequence $\{ \vr_{h_n} , \vm_{h_n} , S_{h_n}, E_{h_n} \}_{n=1}^\infty$ generates a Young measure
$\{ \mathcal{V}_{t,x} \}_{(t,x) \in Q} \in \mathcal{P}(R^{d+3})$, with a marginal
$\{ \nu_{t,x} \}_{(t,x) \in Q} \in \mathcal{P}(R^{d+2})$,
\[
\left< \nu_{t,x}; g(\tvr, \tvm, \widetilde{S}) \right> \equiv \left< \mathcal{V}_{t,x}; g(\tvr, \tvm, \widetilde{S})
1_R (\widetilde{E}) \right>,\ g \in BC(R^{d+2});
\]
\item
\[
W_1 \left( \overline{\mathcal{V}}^N_{t,x} , \mathcal{V}_{t,x} \right) \to 0 \ \mbox{as}\ N \to \infty
\ \mbox{for a.a.}\ (t,x) \in Q,
\]
for the Ces\` aro averages
\[
\overline{\mathcal{V}}^N_{t,x} \equiv \frac{1}{N} \sum_{n=1}^N \delta_{[\vr_{h_n}, \vm_{h_n}, S_{h_n}, E_{h_n}](t,x)}
\in \mathcal{P}(R^{d + 3});
\]

\item

\[
\| W_q  \left( \overline{\nu}^N_{t,x} , \nu_{t,x} \right)  \|_{L^q(Q)} \to 0 \ \mbox{as}\ N \to \infty
\ \mbox{ for any }\ 1 \leq q < \frac{2 \gamma}{\gamma + 1}
\]
for the Ces\` aro averages
\[
\overline{\nu}^N_{t,x} \equiv \frac{1}{N} \sum_{n=1}^N \delta_{[\vr_{h_n}, \vm_{h_n}, S_{h_n}](t,x)}
\in \mathcal{P}(R^{d + 2}).
\]

\end{itemize}

\end{Theorem}

\subsection{Convergence in first variation}

Using \eqref{K3} we may show convergence of the first variation of the measures $\overline{\nu}^N_{t,x} $. Indeed,
\[
\begin{split}
\left<  \overline{\nu}^N_{t,x}; \left| \widetilde \vr - \left< \overline{\nu}^N_{t,x} ; \widetilde{\vr} \right> \right| \right>
&= \frac{1}{N} \sum_{n=1}^N \left| \vr_{h_n}(t,x) - \left( \frac{1}{N} \sum_{n=1}^N \vr_{h_n}(t,x) \right)  \right|
\\
&=
\frac{1}{N} \sum_{n=1}^N \left| \vr_{h_n}(t,x) - \vr(t,x)
 + \vr(t,x) - \left( \frac{1}{N} \sum_{n=1}^N \vr_{h_n}(t,x) \right)  \right|,
\end{split}
\]
where, by virtue of \eqref{K2},
\[
 \vr(t,x) - \left( \frac{1}{N} \sum_{n=1}^N \vr_{h_n}(t,x) \right)  \to 0
\ \mbox{ as }\ N \to \infty.
\]
Next,
\[
\frac{1}{N} \sum_{n=1}^N \left| \vr_{h_n}(t,x) - \vr(t,x)   \right| =
\left< \overline{\nu}^N_{t,x} ; \chi_R \left(\left| \widetilde{\vr} - \vr(t,x) \right| \right) \right>
+ \left< \overline{\nu}^N_{t,x} ; (1 - \chi_R) \left(\left| \widetilde{\vr} - \vr(t,x) \right| \right) \right>,
\]
where $\chi_R$ is a bounded cut--off function. In accordance with \eqref{K1}, we have
\[
\left< \overline{\nu}^N_{t,x} ; \chi_R \left(\left| \widetilde{\vr} - \vr(t,x) \right| \right) \right>
\to \left< \nu_{t,x}; \chi_R \left(\left| \widetilde{\vr} - \vr(t,x) \right| \right) \right>
\ \mbox{ for a.a. } \ (t,x) \in Q \ \mbox{ and any fixed }\ R,
\]
while, by virtue of \eqref{w1},
\[
\int_0^T \intO{ \left| \left< \overline{\nu}^N_{t,x} ; (1 - \chi_R) \left(\left| \widetilde{\vr} - \vr(t,x) \right| \right) \right>
\right| } \dt \to 0 \ \mbox{ as }\ R \to \infty \ \mbox{ uniformly in }\ N.
\]
As, finally,
\[
\left< \nu_{t,x}; \chi_R \left(\left| \widetilde{\vr} - \vr(t,x) \right| \right) \right>
\to \left< \nu_{t,x};  \left(\left| \widetilde{\vr} - \left< \mathcal{V}_{t,x}; \widetilde{\vr} \right> \right| \right) \right>
\ \mbox{as}\ R \to \infty \ \mbox{ for a.a.}\ (t,x) \in Q,
\]
we may infer that
\begin{equation} \label{K5}
\left\| \left< \overline{\nu}^N_{t,x} ; \left| \widetilde \vr - \left< \overline{\nu}^N_{t,x} ; \widetilde{\vr} \right> \right| \right>
- \left< \nu_{t,x}; \left|\widetilde{\vr} - \left< \nu_{t,x}; \widetilde{\vr} \right> \right| \right> \right\|_{L^1(Q)} \to 0
\ \mbox{ as }\ N \to \infty.
\end{equation}
Analogously, we have
\begin{equation} \label{K6}
\left\| \left< \overline{\nu}^N_{t,x} ; \left| \widetilde \vm - \left< \overline{\nu}^N_{t,x} ; \widetilde{\vm} \right> \right| \right>
- \left< \nu_{t,x}; \left|\widetilde{\vm} - \left< \nu_{t,x}; \widetilde{\vm} \right> \right| \right> \right\|_{L^1(Q)} \to 0
\ \mbox{ as }\ N \to \infty
\end{equation}
and
\begin{equation} \label{K51}
\left\| \left< \overline{\nu}^N_{t,x} ; \left| \widetilde S - \left< \overline{\nu}^N_{t,x} ; \widetilde{S} \right> \right| \right>
- \left< \nu_{t,x}; \left|\widetilde{S} - \left< \nu_{t,x}; \widetilde{S} \right> \right| \right> \right\|_{L^1(Q)} \to 0
\ \mbox{ as }\ N \to \infty.
\end{equation}

Unlike $\vrh$, $\vmh$, $S_h$ the energy $\Eh$ is not equi--integrable therefore its momentum cannot be estimated as in
\eqref{K5}, \eqref{K6}. Although it may be seen that \eqref{K5}--\eqref{K6} follow directly from the convergence in the Wasserstein distance as shown in Section~\ref{Vas_sec}, the
convergence of the variance
\begin{eqnarray}
&& \left< \overline{\nu}^N_{t,x} ; \left|  \widetilde \vr
-  \left< \overline{\nu}^N_{t,x} ;  \widetilde \vr \right> \right| \right>
+
\left< \overline{\nu}^N_{t,x} ; \left|  \widetilde \vm
-  \left< \overline{\nu}^N_{t,x} ;  \widetilde \vm \right> \right| \right>
+
\left< \overline{\nu}^N_{t,x} ; \left|  \widetilde S
-  \left< \overline{\nu}^N_{t,x} ;  \widetilde S \right> \right| \right> =
\nonumber \\
&& \frac{1}{N} \sum_{n=1}^N \left| \vr_{h_n}  -  \frac{1}{N} \sum_{j=1}^N  \vr_{h_j}\right|
+
\frac{1}{N} \sum_{n=1}^N \left| \vm_{h_n}  -  \frac{1}{N} \sum_{j=1}^N  \vm_{h_j}\right|
+
\frac{1}{N} \sum_{n=1}^N \left| S_{h_n}  -  \frac{1}{N} \sum_{j=1}^N  S_{h_j}\right|
\end{eqnarray}
is directly computable performing simulations on $N$ different grids with the size $h_n$, $n=1,\dots, N.$

\section{Finite volume methods} \label{Brenner}

In this section we demonstrate robustness of the new concept of {$\mathcal{K}$}-convergence and apply it to two different finite volume schemes.  We consider first a finite volume method that has been derived in our recent paper \cite{FLM_18}, based
on the work of Brenner \cite{BREN1} and Guermond and Popov \cite{GuePop}. In what follows we will refer to it as the \emph{FLM finite volume method}. The latter is an example of a consistent entropy stable finite volume numerical scheme whose solutions $[\vrh, \mh, E_h]$ satisfy (\ref{D1})--(\ref{D4}) and thus serves as an example to illustrate the abstract theory and rigorous convergence proofs.

The second finite volume method is a well-known second order finite volume method based on the generalized Riemann problem, see, e.g., \cite{ben1984, ben2003, ben2007, Li2016, Li2007}. In what follows we will refer to it as the \emph{GRP finite volume method}. Its theoretical convergence analysis  may be an interesting question for future study. We will consider the
classical benchmark problems, the Kelvin-Helmholtz and Richtmyer-Meshkov tests. Due to the oscillatory solution behaviour the standard mesh convergence study will
indicate that neither of the finite volume methods converges. However, we will demonstrate the strong convergence of the Ces\`aro averages as well as the first variance for both finite volume methods. For the empirical means of the corresponding Dirac distributions the convergence in the Wasserstein distance will be shown as well.

\subsection{FLM finite volume method} \label{FLM}

We assume that the physical space is a polyhedral domain
$\Omega \subset R^d$, $d=2,3$, decomposed into compact polygonal sets (tetrahedral or  parallelepipedal elements)
\[
 \Ov{\Omega} = \bigcup_{K \in \grid_h} K.
\]
We denote by $h$ the maximal size of the mesh,
$$
h=\max_{K \in \grid_h} h_K
$$
with $h_K$ being the diameter of an element $K$.
The elements are sharing either a common face, edge, or vortex. We assume that the mesh $\grid_h$ satisfies the standard regularity assumptions, cf.~\cite{ciarlet,EyGaHe}. More precisely, let a parameter $\theta_h$ be defined as
\begin{equation}\label{reg}
	\theta_h = \inf \left\{ \frac{\xi_K}{h_K}, K \in {\mathcal{T}}_h \right\},
\end{equation}
where $\xi_K$ stands for the diameter of the largest ball included in $K$.
Then the mesh $\grid_h$ is said to be regular and quasi-uniform,  if there exists positive real numbers
$c_0$ and $\theta_0$ independent of $h$ such that
\begin{equation}\label{reg1-}
 	 \theta_h \ge \theta_0 \text{ and } c_0 h \leq h_K   ,
\end{equation}
respectively.

The set of all faces is denoted by $\Sigma_h,$ while
$\Sigma^{int}_h=\Sigma_h\backslash\partial\Omega$ stands for the set of all interior faces. Each face
is associated with a normal vector $\vc{n}$.
In what follows, we shall suppose
\[
|K|_d \approx h^d, \ |\sigma|_{d-1} \approx h^{d-1} \ \mbox{for any}\ K \in \grid_h, \ \sigma \in \Sigma.
\]

Let us denote by $\mathcal{P}^0_h$ a space of piecewise constant function on  $\grid_h$.
For a function $f_h \in \mathcal{P}^0_h$ we define
\begin{align*}
f_h^{\rm out}(x) &= \lim_{\delta \to 0+} f_h(x + \delta \vc{n}),\ &
\av{f_h} &= \frac{f_h^{\rm in}(x) + f_h^{\rm out}(x) }{2},\ &
\\
f_h^{\rm in}(x) &= \lim_{\delta \to 0+} f_h(x - \delta \vc{n}), \ & [[ f_h ]] &= f_h^{\rm out}(x) - f_h^{\rm in}(x)
\end{align*}
whenever $x \in \sigma \in \Sigma^{int}_h$.

Finally, let
$D_t$ denote the first order backward finite difference approximating the time derivative, i.e.
$$
D_t f_h (t) = \frac{f_h(t_{k+1}) - f_h(t_k)}{\Delta t}, \qquad f_h(t_k), f_k(t_{k+1}) \in \mathcal{P}^0_h,
$$
where $\Delta t$ is a mesh step parameter on a time interval $[0,T]$, $t_k = k \Delta t$  and $t \in (t_k, t_{k+1}],\ k=1,2,\dots$. \\[0.2cm]

The \emph{numerical solutions} are piecewise constant functions on $(0,T) \times \Omega$, such that \,
$\vrh(t)\equiv \vrh(t_{k+1}) \in \mathcal{P}^0_h(\Omega_h)$, $\vc{m}_h(t) \equiv \vc{m}_h(t_{k+1}) \in \mathcal{P}^0_h(\Omega_h; R^d)$, and $E_h(t) \equiv E_h(t_{k+1}) \in \mathcal{P}^0_h(\Omega_h),$ \ $t \in (t_k, t_{k+1}] $ satisfying the following discrete equations:

\begin{itemize}
\item {\bf Continuity equation}
\begin{equation} \label{n2}
\intOh{ D_t \vrh \Phi } - \sum_{ \sigma \in \Sigma^{int}_h } \intSh{  F_h(\vrh,\vuh)
[[\Phi]]  } = 0 \ \mbox{for any}\ \Phi \in \mathcal{P}^0_h (\Omega_h),
\end{equation}
where
\[
\vc{u}_h = \frac{\vc{m}_h}{\vr_h}.
\]
\item {\bf Momentum equation}
\begin{equation} \label{n3}
\begin{split}
\intOh{ D_t \mh \cdot \bfPhi } &- \sum_{ \sigma \in \Sigma^{int}_h } \intSh{ {\bm{F}}_h(\mh,\vuh)
\cdot [[\bfPhi]]  }- \sum_{ \sigma \in \Sigma^{int}_h } \intSh{ \Ov{p_h} \vc{n} \cdot [[ \bfPhi ]] } \\
&= - h^{\alpha - 1} \sum_{ \sigma \in \Sigma^{int}_h } \intSh{ [[ \vu_h ]] \cdot [[ \bfPhi ]] } \ \mbox{for all}\
\bfPhi \in \mathcal{P}^0_{h} (\Omega_h, R^d),
\end{split}
\end{equation}
where
\[
p_h = (\gamma - 1) \left( E_h - \frac{1}{2} \frac{|\vc{m}_h|^2}{\vr_h} \right).
\]

\item {\bf Energy equation}
\begin{equation} \label{n4}
\begin{split}
\intOh{ D_t E_h \Phi } &- \sum_{ \sigma \in \Sigma^{int}_h } \intSh{  F_h(E_h,\vu_h)
[[\Phi]]  }\\
&- \sum_{ \sigma \in \Sigma^{int}_h } \intSh{ \Ov{ p_h } [[ \Phi \vu_h ]] \cdot \vc{n} }
+ \sum_{ \sigma \in \Sigma^{int}_h } \intSh{ \Ov{ p_h \Phi} [[\vu_h]] \cdot \vc{n} } \\
&= - h^{\alpha - 1} \sum_{ \sigma \in \Sigma^{int}_h } \intSh{ [[ \vu_h ]] \cdot \Ov{\vu_h}   [[ \Phi ]] } \ \mbox{for all}\
\Phi \in \mathcal{P}^0_{h} (\Omega_h).
\end{split}
\end{equation}
\end{itemize}

The function $F_h$ stays for a \emph{finite volume numerical flux} approximating the physical flux. In the literature one can find already a large variety of different numerical fluxes proposed for the hyperbolic conservation laws and the Euler equations in particular, see, e.g.,
\cite{GR96, feist1, feist2, Kroner, LeVeque1, LeVeque2}
and the references therein. In \cite{FLM_18} we have used the following numerical flux function based on the upwinding strategy
\begin{align}\label{num_flux}
F_h(r_h,\vu_h)={Up}[r_h, \vuh] - \mu_h [[ r_h ]], \qquad r_h \in \mathcal{P}^0_h,
\end{align}
where $\mu_h\geq 0$ and ${Up}[r_h, \vuh]$  is a classical upwinding of $r_h$.
This means that depending on the sign of the normal component of the velocity $\av{\vuh} \cdot \vc{n}$,
$r_h$ is taken to be either the upward or downward value of the neighbouring cells. More precisely,  we have
\begin{align}\label{Up}
Up [r_h, \vu_h] = \av{r_h} \ \av{\vu_h} \cdot \vc{n} - \frac{1}{2} |\av{\vu_h} \cdot \vc{n}| [[r_h]] =
r_h^{\rm in} [\av{\vu_h} \cdot \vc{n}]^+ + r_h^{\rm out} [\av{\vu_h} \cdot \vc{n}]^-.
\end{align}
The numerical flux $\bm{F}_h(\bm{m}_h,\bm{u}_h)$ denotes the corresponding vector--valued flux function for the approximation of a physical flux in the momentum equation.  Note, that our upwinding  ${Up}[r_h, \vu_h], $  $r_h = \vr_h, \vc{m}_h, E_h$,  is
based only on the sign of the normal component of velocity, that is  typically used in incompressible flow approximations, instead of the sign of the eigenvalues as it is standard in the flux--vector splitting schemes for compressible flows.
Instead of the corresponding eigenvalues of the underlying hyperbolic system we add a numerical diffusion term $-\mu_h [[ r_h ]]  $  for the conservative variables. In numerical simulations one can take, for instance, $h^{\beta} \lesssim \mu_h \lesssim 1$ with $0 \leq \beta <1$.

\subsubsection{Stability and consistency}
The finite volume method \eqref{n2}-\eqref{n4} has the following remarkable properties that have been proven in \cite{FLM_18}.
\begin{itemize}
\item \emph{Positivity of the discrete density and temperature:} numerical density and temperature remain strictly positive on any finite time interval.
\item \emph{Entropy stability:} discrete entropy inequality holds, cf.~the concept of entropy stable numerical schemes~\cite{Tad86,Tad03}.
\item \emph{Minimum entropy principle:} the discrete entropy attains its minimum at the initial time.
\item \emph{Weak BV estimates:} discrete entropy inequality allows us to control suitable weak BV norms of the discrete density, temperature and velocity.
\end{itemize}

It is to be pointed out that the numerical diffusions $\mu_h [[\vr_h]], \mu_h [[\bm{m}_h]]$ and $h^{\alpha-1} [[\bm{u}_h]], \ h^{\alpha-1} [[\bm{u}_h]] \av{\bm{u}_h}$ arising from the numerical flux and additional $h^\alpha$-diffusive terms, respectively,  allow us to prove \emph{positivity of the discrete density} on each $\grid_h$, i.e.
\begin{itemize}
\item[]
for any $T>0,$ there exists $\underline{\vr}=\underline{\vr}(h,T)>0,$ such that $\vrh(t)\geq \underline{\vr}>0$ for all $t\in [0,T]$,  if the initial data satisfy the corresponding condition.
\end{itemize}

\noindent Further, the minimum entropy principle yields the positivity of the discrete pressure and
temperature, see \cite{FLM_18}. Taking into account a priori estimates that yield uniform bounds on $\vr_h \in L^\infty((0,T), L^\gamma(\Omega_h)),$ $ \bm{m}_h \in  L^\infty((0,T), L^{\frac{2\gamma}{\gamma +1}}(\Omega_h)),$ and  $E_h \in L^\infty((0,T), L^1(\Omega_h))$
and the weak BV estimates resulting from the discrete entropy inequality, we obtain the following consistency result of the finite volume method
\eqref{n2}--\eqref{n4}, see Section~\ref{w} and \cite{FLM_18}.

\begin{Proposition}[Consistency of the FLM finite volume method] \label{Tm2}

Assume that the initial data $\vr_{0,h}$, $\vc{m}_{0,h}$, $E_{0,h}$ satisfy
\[
\vr_{0,h} \geq \underline{\vr} > 0,\ E_{0,h} - \frac{1}{2} \frac{ |\vc{m}_{0,h}|^2 }{\vr_{0,h}} > 0.
\]
Let $[\vr_h, \vc{m}_h, E_h]$ be the unique solution of the finite volume method \eqref{n2}--\eqref{n4}
on the time interval $[0,T]$.  Then
\begin{equation} \label{cP1}
\left[ \intOh{ \vr_h \varphi } \right]_{t=0}^{t = \tau} =
\int_0^\tau \intOh{ \left[ \vr_h \partial_t \varphi + \vc{m}_h \cdot \Grad \varphi \right]} \dt  + \int_0^T
e_{1,h} (t, \varphi) \dt
\end{equation}
for any $\varphi \in C^1([0,T] \times \Ov{\Omega}_h)$;
\begin{equation} \label{cP2}
\left[ \intOh{ \vc{m}_h \bfphi } \right]_{t=0}^{t = \tau} =
\int_0^\tau \intOh{ \left[ \vc{m}_h \cdot \partial_t \bfphi + \frac{\vc{m}_h \otimes \vc{m}_h} {\vr_h} : \Grad \bfphi
+ p_h \Div \bfphi \right]} \dt  + \int_0^T
e_{2,h} (t, \bfphi) \dt
\end{equation}
for any $\bfphi \in C^1([0,T] \times \Ov{\Omega}_h; R^d)$, $\bfphi \cdot \vc{n}|_{\Omega_h} = 0$;
\begin{equation} \label{cP3}
\intOh{ E_h(t) } = \intOh{ E_{0,h} };
\end{equation}
\begin{equation} \label{cP4}
\left[ \intOh{ \vr_h \chi(s_h) \varphi } \right]_{t=0}^{t = \tau} \geq
\int_0^\tau \intOh{ \left[ \vr_h \chi (s_h) \partial_t \varphi + \chi(s_h) \vc{m}_h \cdot \Grad \varphi \right]} \dt  + \int_0^T
e_{3,h} (t, \varphi) \dt
\end{equation}
for any $\varphi \in C^1([0,T] \times \Ov{\Omega}_h)$, $\varphi \geq 0$, and any $\chi$,
\[
\chi :  R \to  R \ \mbox{a non--decreasing concave function}, \ \chi(s) \leq \av{\chi} \ \mbox{for all}\ s \in \ R.
\]

If, in addition,
\begin{equation} \label{cP5}
 h^\beta \aleq \mu_h \aleq 1, \ 0 \leq \beta < 1,\ 0< \alpha < \frac{4}{3},
\end{equation}
and
\begin{equation} \label{cP6}
0 < \underline{\vr} \leq \vr_h(t), \ \vt_h(t) \leq \Ov{\vt} \ \mbox{for all}\ t \in [0,T] \ \mbox{uniformly for}\ h \to 0,
\end{equation}
then
\[
\| e_{j,h} (\cdot, \varphi ) \|_{L^1(0,T)} \aleq h^\delta \| \varphi \|_{C^1} \ \mbox{for some}\ \delta > 0, \ j=1,2,3.
\]
\end{Proposition}

\begin{Remark}
In the case of  uniform rectangular/cubic elements the conclusion of Proposition~\ref{Tm2} remains valid  for $0\leq \mu_h\aleq 1,$ and $\bfphi \in C^1([0,T]; C^2(\Ov{\Omega}_h; R^d))$, $\bfphi \cdot \vc{n}|_{\Omega_h} = 0.$ Moreover, if $d=2$ then the condition \eqref{cP5} can be
relaxed  to $ 0 \leq \beta < 1, \ 0 <  \alpha < 2. $ Let us point out that omitting the artificial diffusion terms, i.e.~$\mu_h = 0$ and removing the terms with the factor $h^{\alpha-1},$ yields the upwind finite volume scheme.
\end{Remark}

Under the extra hypothesis \eqref{cP6}, the numerical solutions
$\{ \vr_h, \vm_h, S_h, E_h \}_{h \searrow 0}$ resulting from the scheme \eqref{n2}--\eqref{n4}  represent a
family of {\em admissible consistent approximate solutions} in the sense of Definition \ref{AD1}. In particular,
the conclusion of Theorem \ref{Tmain} as well as other
results obtained in Section \ref{K} can be applied.

\subsection{GRP finite volume method} \label{GRP}

The generalized Riemann problem finite volume method is one of successful standard  numerical methods to simulate the Euler equations.
It was developed as an analytical second order accurate extension of the classical Godunov finite volume method.
Originally the method was based on the Lagrangian formulation \cite{ben1984,ben2003}. A direct Eulerian GRP scheme was presented in \cite{ben2007,ben2006,Li2007} by employing the regularity property of the Riemann invariants. Theoretically, a close coupling between the spatial and temporal evolution is recovered through the analysis of detailed wave interactions in the GRP scheme. The GRP method has been applied successfully to develop high resolution schemes and used for many engineering problems, see, e.g., \cite{ben2003, Cheng2019, Li2016}
and the references therein.

To describe the main ingredients of the GRP finite volume method let us
consider a two-dimensional regular rectangular grid with mesh cells $K_{ij} \equiv [x_{i-1/2},x_{i+ 1/2}] \times [y_{j-1/2},y_{j+1/2}]$, $x_{i\pm 1/2} = x_i \pm \Delta x/2$, $ y_{j\pm 1/2} = y_j \pm \Delta y/2$, $\Delta x = h = \Delta y$, $i,j \in \Bbb N.$
The basic idea of the GRP scheme consists of replacing the exact solution $\bU \equiv [\vr, \vm, E] $  on a mesh cell
$K_{ij}$ by a piecewise linear function
\begin{align}
\bU(x,y,t)=\bU(x_i, y_j,t)+\bU_x(x_i,y_j,t) (x-x_i)+\bU_y(x_i,y_j,t)(y-y_j),\label{rec}
\end{align}
and analytically solve locally at each cell interface $\sigma \in \partial K_{ij}$
an one-dimensional generalized Riemann problem with the initial data
\begin{align}
\bU(\xi,t=0) = \left\{\begin{array}{ll}
\bU_L+\bU'_L \xi, & \xi <0 \\
\bU_R+\bU'_R \xi, & \xi >0.
\end{array}
\right.
\end{align}

Let us now denote by $\bm{F} = (\bm{F}_1, \bm{F}_2)$  the physical flux of the Euler equations (\ref{i1}), i.e.
$$
\bm{F} = \left[\bm{m},  \frac{\vm \otimes \vm}{\vr} + p \Bbb I, \left( E + p \right) \frac{\vm}{\vr}\right].
$$
Then the GRP numerical flux  on  the cell interface $\sigma = (x_{i+1/2},y)$, $y \in [y_{j-1/2}, y_{j+1/2}],$  is given as
\begin{align}
&  \bm{F}_h \equiv \bm{F}_{i+1/2,j}^{k+1/2}=\bm{F}\left(\bU_{i+1/2,j}^{k+1/2}\right), \qquad   \bU_{i+1/2,j}^{k+1/2}=\bU_{i+1/2,j}^{k}+\frac{\Delta t}{2}(\bU_t)_{i+1/2,j}^{k}.\notag
\end{align}
Here
$\bU_{ij}^{k}$ is the cell average of $\bU(x,y,t_k)$ over the control volume $K_{ij}$ evaluated at time $t_k$.

\noindent The resulting scheme enjoys the second order accuracy both in space and time
 \begin{eqnarray}
&& \bU_{ij}^{k+1}=\bU_{ij}^{k}-\frac{\Delta t}{\Delta x}\left[\bm{F}_{h,1}\left(\bU_{i+1/2,j}^{k+1/2}\right)-\bm{F}_{h,1}\left(\bU_{i-1/2 ,j}^{k+1/2}\right)\right]-\frac{\Delta t}{\Delta y}\left[\bm{F}_{h,2}\left(\bU_{i,j+1/2}^{k+1/2}\right)-\bm{F}_{h,2}\left(\bU_{i,j-1/2}^{k+1/2}\right)\right].\nonumber \\
&&  \label{GRP_FV}
 \end{eqnarray}
Gradients arising in (\ref{rec}) are standardly computed by the so-called minmod limiter to avoid spurious oscillations on discontinuities, see, e.g., \cite{ben1984}.

\subsection{Numerical experiments}  \label{numerics}
In this section we consider two classical benchmarks, the Kelvin-Helmholtz and the Richtmyer-Meshkov problems, and illustrate robustness of the concept of ${\mathcal K}$-convergence using, as an example, the FLM \eqref{n2}--\eqref{n4} and GRP \eqref{GRP_FV} finite volume schemes.
These benchmarks have been also used by Fjordholm et al.~\cite{FjMiTa1, FjMiTa3}, where the {\em weak$(^*)$-convergence} of the statistical solutions
has been investigated.  It is to be pointed out that our theoretical results yield the {\em strong convergence}
to a dissipative solution, which is the barycenter of the DMV solution. This fact will be confirmed by the numerical experiments in what follows.

For this purpose, we denote by $E_1$, $E_2$, $E_3$, and $E_4$ the $L^1$-error of the difference between the numerical solution and the reference solution
computed on a finest grid, the $L^1$-error of the Ces\`aro averages, the $L^1$-error of the first variance, and the $L^1$-error of the Wasserstein distance\footnote{An open source library SciPy \cite{SciPy} was used to compute $W_1$-distance.
We thank S.~Mishra and K.~Lye for pointing us this package.}
 for the Ces\`aro average of the Dirac measures concentrated on numerical solutions, respectively. More precisely, we have
\[
E_1 = \left\| U_{h_n}  -   U_{h_N}\right\|, \;
E_2 = \left\| \widetilde{U}_{h_n}  -  \widetilde{U}_{h_N} \right\|, \;
E_3 = \left\| {U}^{\dagger}_{h_n}  -  {U}^{\dagger}_{h_N} \right\|, \;
E_4= \left\| W_1( \overline{\mathcal{V}}^n_{t,x}, \overline{\mathcal{V}}^N_{t,x}) \right\|,
\]
where $h_n= \frac{1}{n}$, $n$ is the number of rectangular mesh cells in each direction, $N= 2048$, $\widetilde{U}_{h_n} =\frac1n\sum_{j=1}^n U_{h_j}$, ${U}^{\dagger}_{h_n}= \frac1n \sum_{j=1}^n \left| U_{h_{j}} - \widetilde{U}_{h_n} \right|$, and $U\in \{\vr, \vm, S, E \}$.
In all of the following tests we set  $\Omega=[0,1]^2$ and apply periodic boundary conditions.
\paragraph{Experiment 1.}
The first experiment is the Kelvin-Helmholtz  problem \cite{Helmhotz, Kelvin} with the initial data
\[
(\vr, u_1, u_2,p)(x)=\left\{
\begin{array}{ll}
(2, -0.5, 0, 2.5), & \text{if}\ I_1<x_2<I_2\\
(1, 0.5, 0, 2.5), & \text{otherwise.}
\end{array}
\right.
\]
Here the interface profiles
\[
I_j = I_j(x,\omega):=J_j+\epsilon Y_j(x,\omega), \quad j=1,2,
\]
are chosen to be small perturbations around the lower $J_1=0.25$ and the upper $J_2=0.75$ interface, respectively. Further,
\[
Y_j(x,\omega)=\sum_{n=1}^{m}a_j^n(\omega)\cos(b_j^n(\omega)+2n\pi x_1),\quad j=1,2,
\]
where $a_j^n=a_j^n(\omega)\in[0,1]$ and $b_j^n=b_j^n(\omega)\in[-\pi,\pi]$, $i=1,2$, $n=1,\ldots,m$ are (fixed) random numbers. The coefficients $a_j^n$ have been normalized such that $\sum_{n=1}^{m}a_j^n=1$ to guarantee that $|I_j(x,\omega)-J_j|\leq \epsilon$ for $j=1,2$. We have set $m=10$ and $\epsilon=0.01.$

In what follows we present the numerical simulations obtained by the FLM finite volume method with $\alpha = 1.8,$ $\beta = 0.8$, upwind finite volume method (see Remark~5.2) and the GRP finite volume method at the final
time $T=2.$ It is the time when small-scales vortex sheets have already been formed at the interfaces.

In Table~\ref{tab1} we show the results of the convergence study for the errors $E_1, \dots, E_4$ in density. We can clearly observe that none of the methods converges in the classical sense, i.e.~single numerical solutions do not converge, see the first column. This behaviour is also demonstrated in the Figures~\ref{fig1}, \ref{fig3_GRP}. Similar non-convergence results have been presented for the second order TeCNO scheme in \cite{FjKaMiTa}.

The second, third and fourth columns in Table~\ref{tab1} show the convergence results for the Ces\`aro averages of 
numerical solutions and their first variance, as well as 
$\mathcal{K}$-convergence of the Wasserstein distance of the corresponding Dirac measures. We should point out that the convergence is strong in the $L^1$-norm as proved above, cf.~\eqref{K3},  \eqref{K5}--\eqref{K51} and Theorem~\ref{Tmain}. The graphs in Figure~\ref{fig1} show the results of the convergence study for all variables
$\vr, \vm, S, E.$ As expected, these variables behave in a similar manner.

Approximate solutions for the density computed by the second order GRP scheme are presented in Figure~\ref{fig3_GRP}, which also clearly indicates that by refining the mesh we recover finer and finer vortex structures and the method does not
converge in the classical sense. However, as already pointed out we have $\mathcal{K}$-convergence of the numerical solution and its first variance, see~Figures~\ref{fig4_GRP}, \ref{fig5_GRP}. The results obtained by the FLM scheme are similar, but more diffusive due to the first order accuracy and we do not present them here.

\begin{figure}[h!]
\begin{subfigure}{\linewidth} \centering
\includegraphics[width=0.24\textwidth]{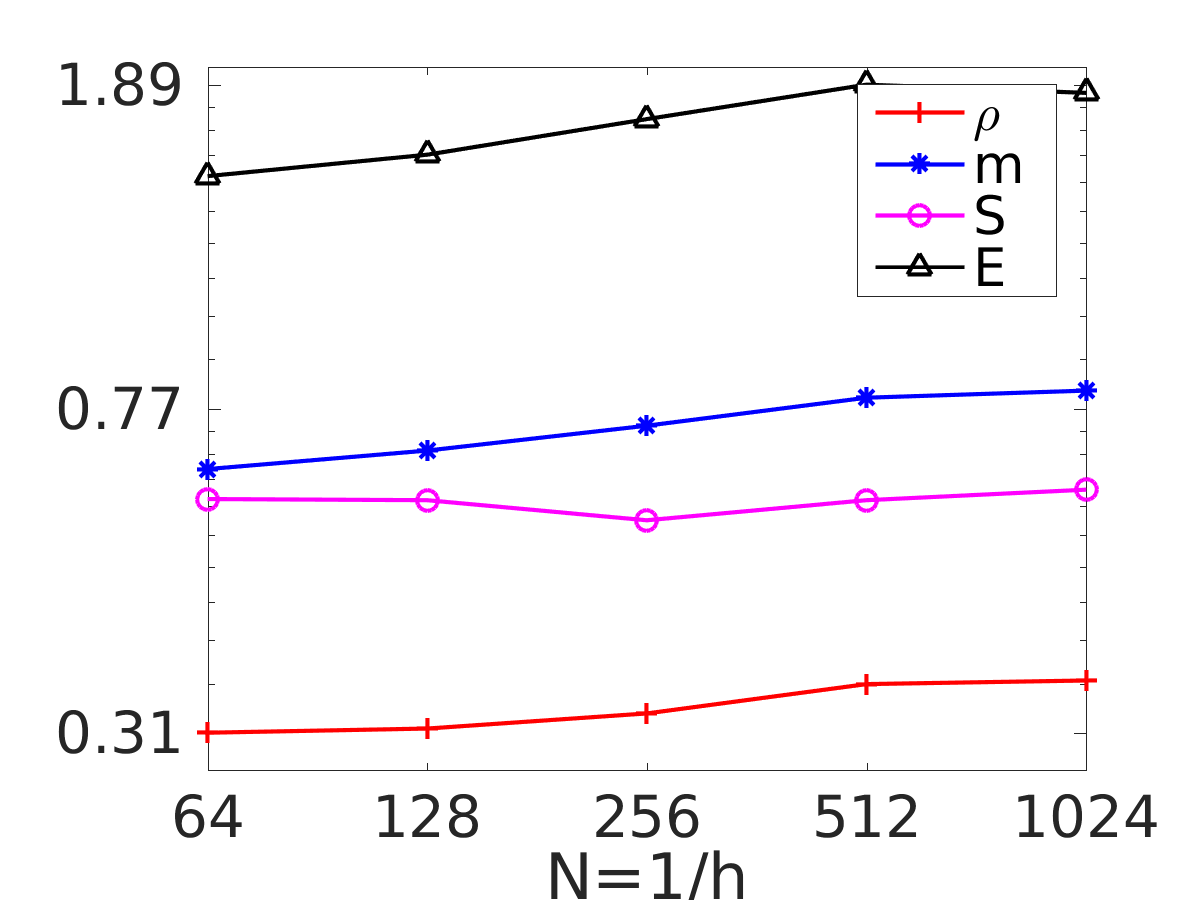}
\includegraphics[width=0.24\textwidth]{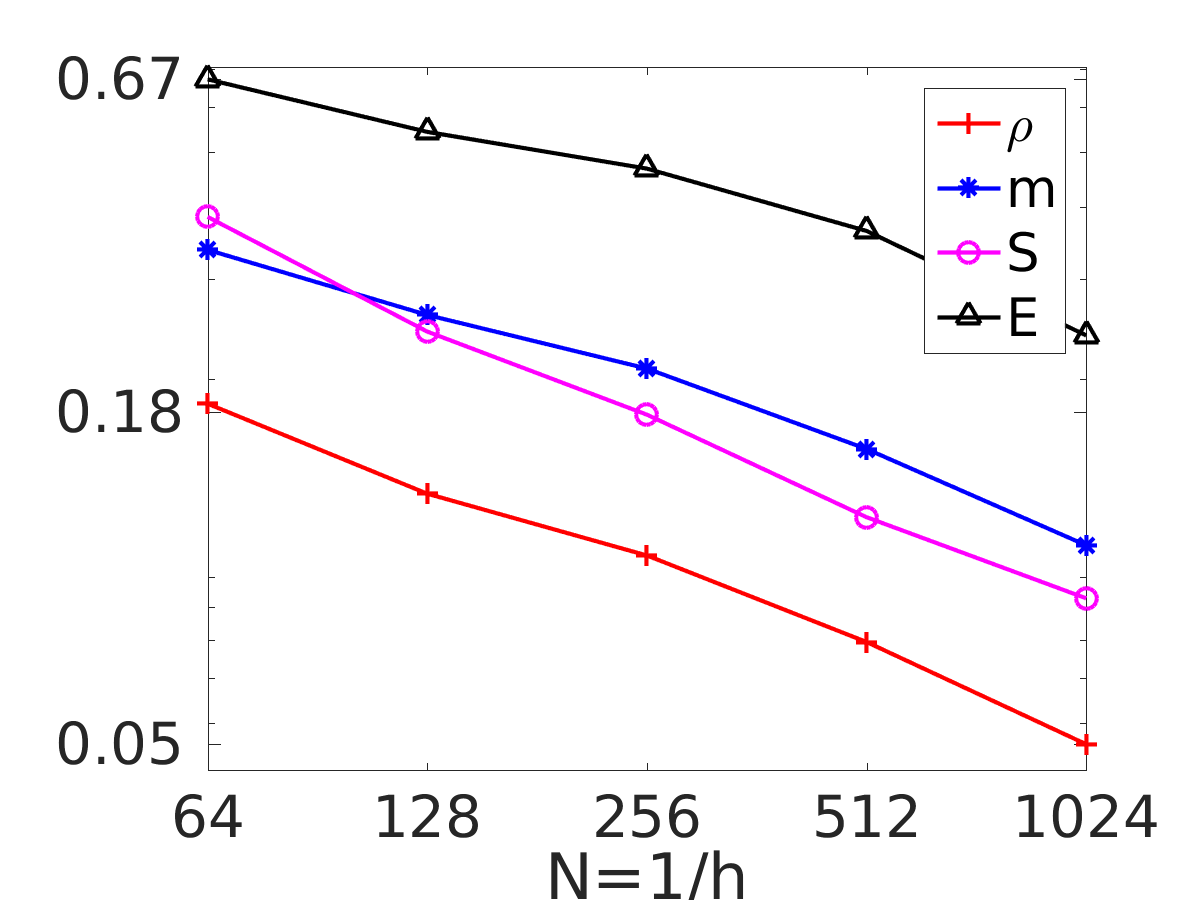}
\includegraphics[width=0.24\textwidth]{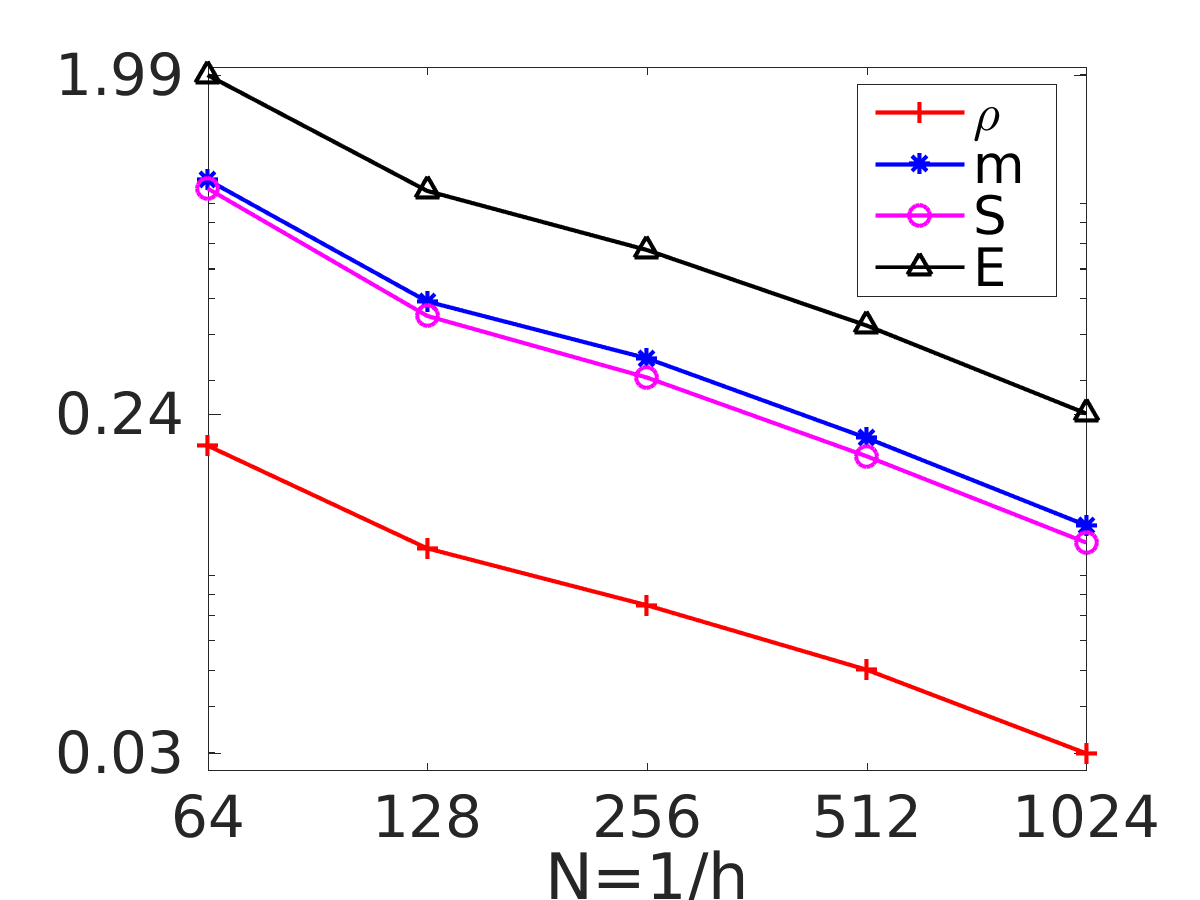}
\includegraphics[width=0.24\textwidth]{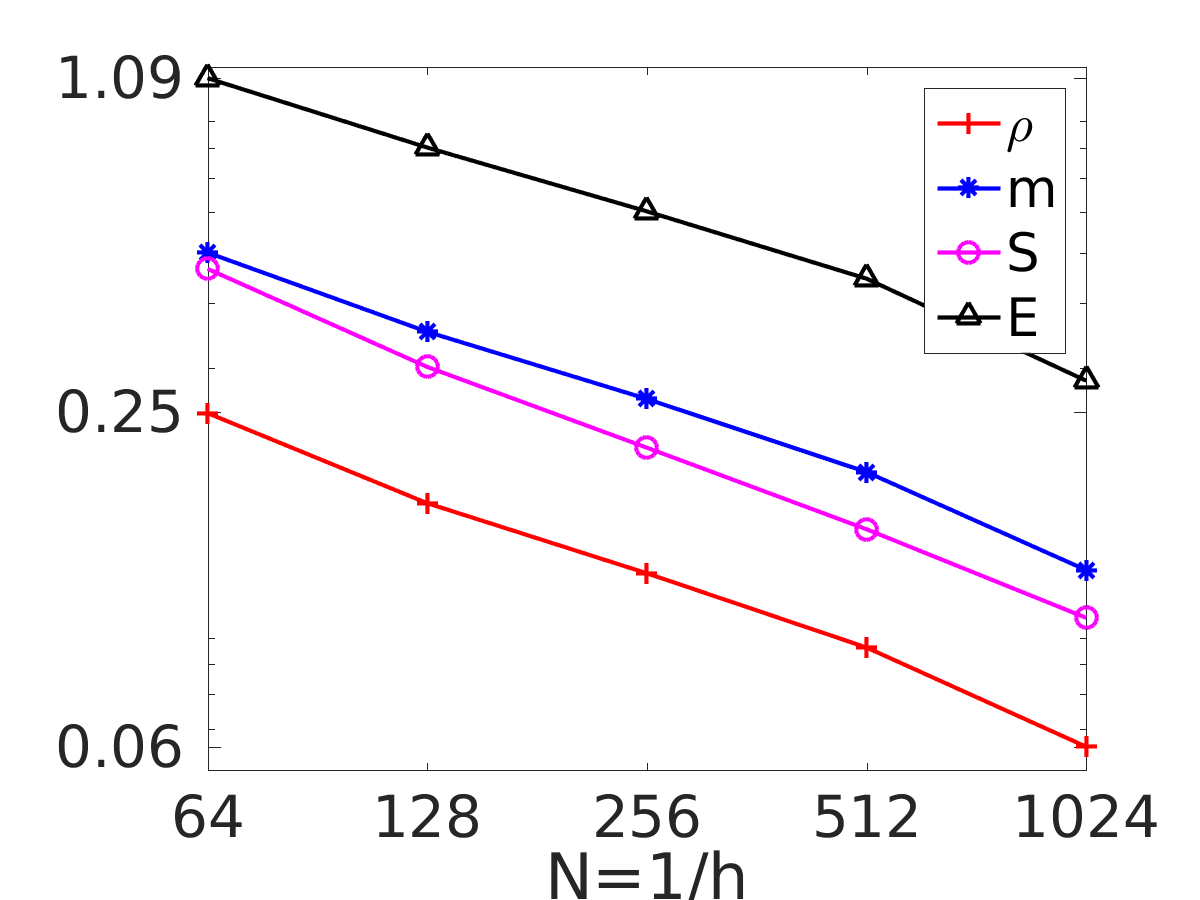}
\caption{FLM scheme}\label{fig1b}
\end{subfigure}
\begin{subfigure}{\linewidth} \centering
\includegraphics[width=0.24\textwidth]{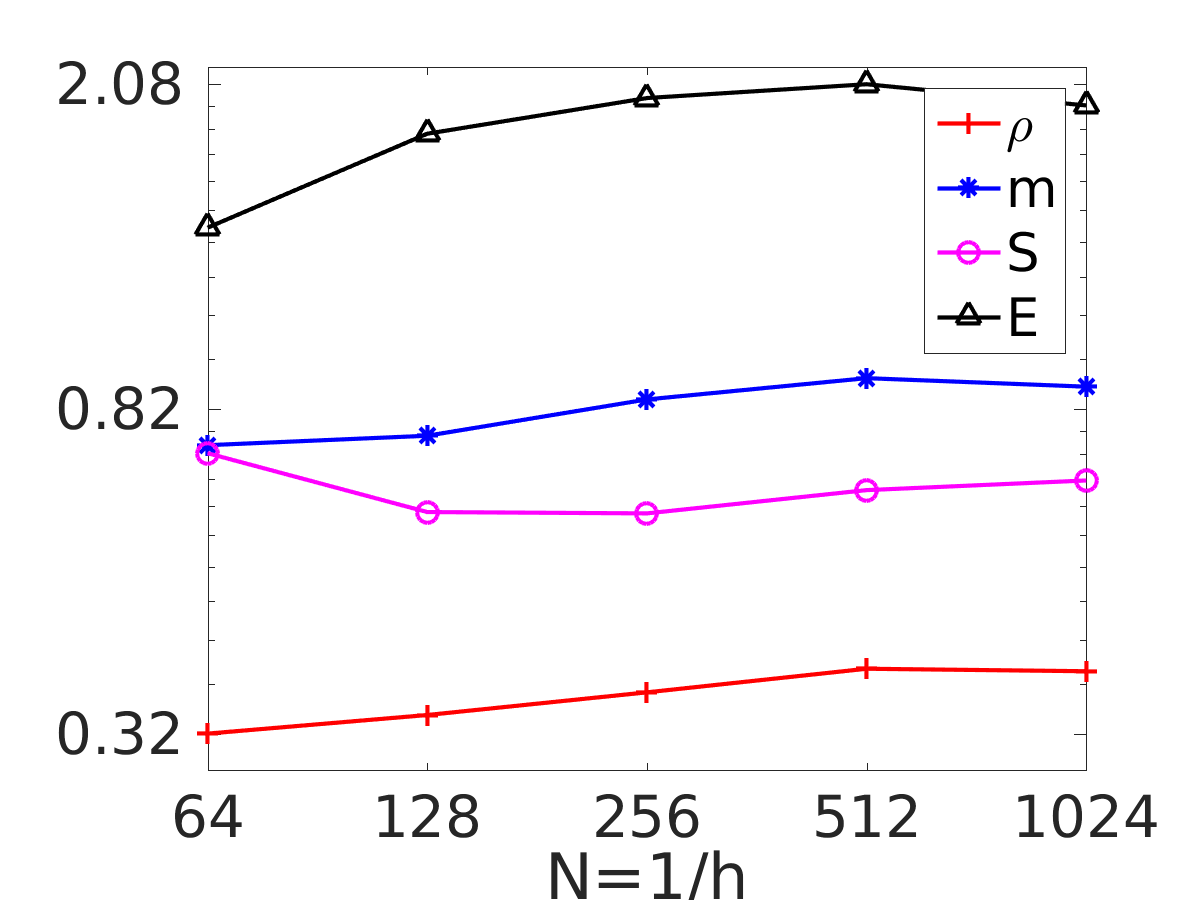}
\includegraphics[width=0.24\textwidth]{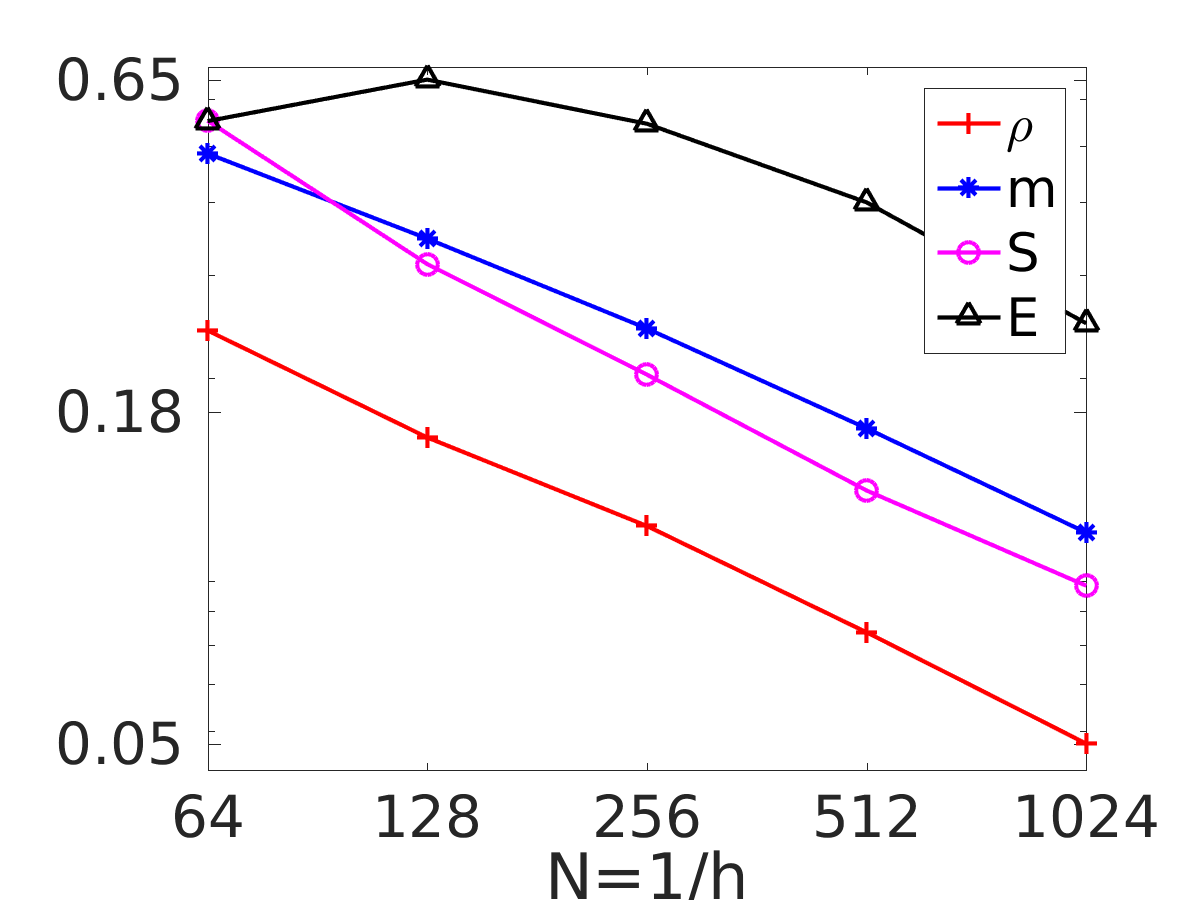}
\includegraphics[width=0.24\textwidth]{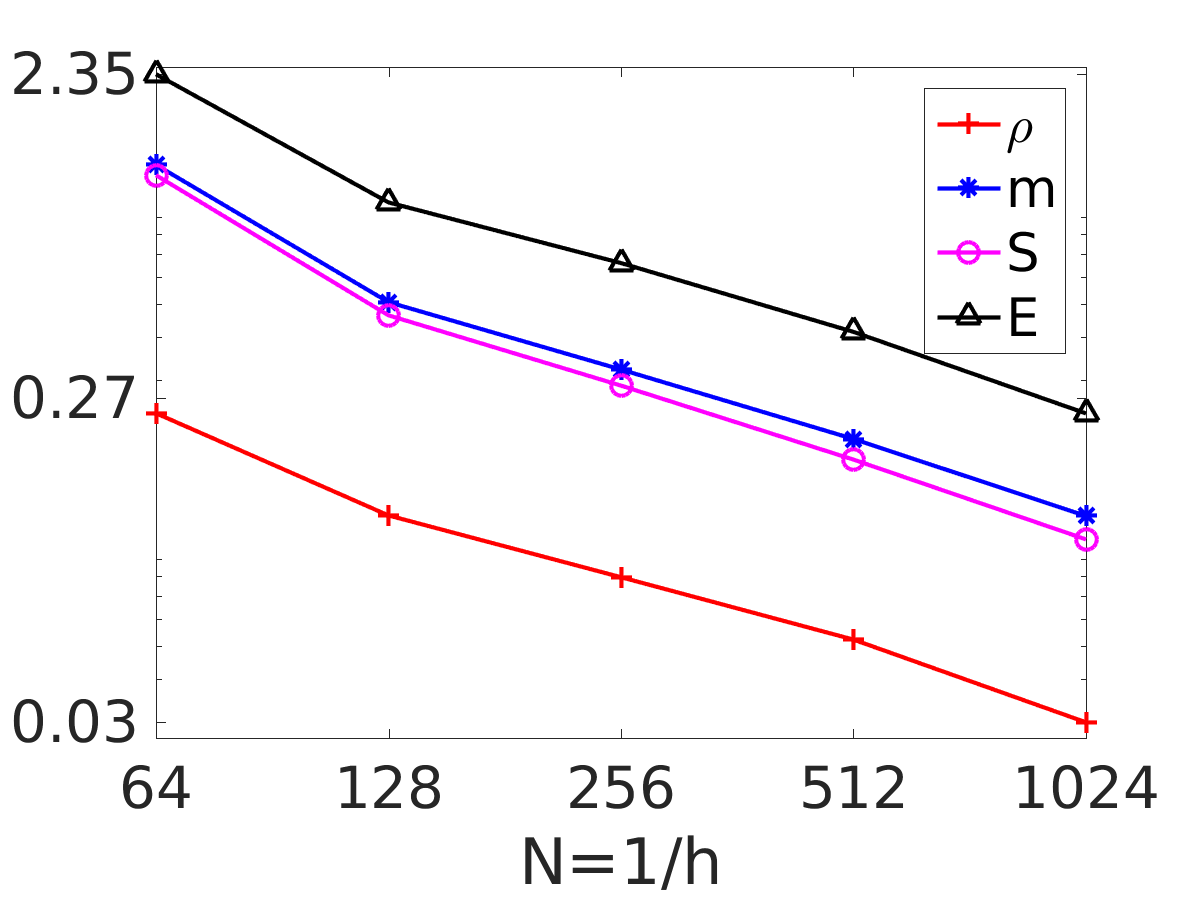}
\includegraphics[width=0.24\textwidth]{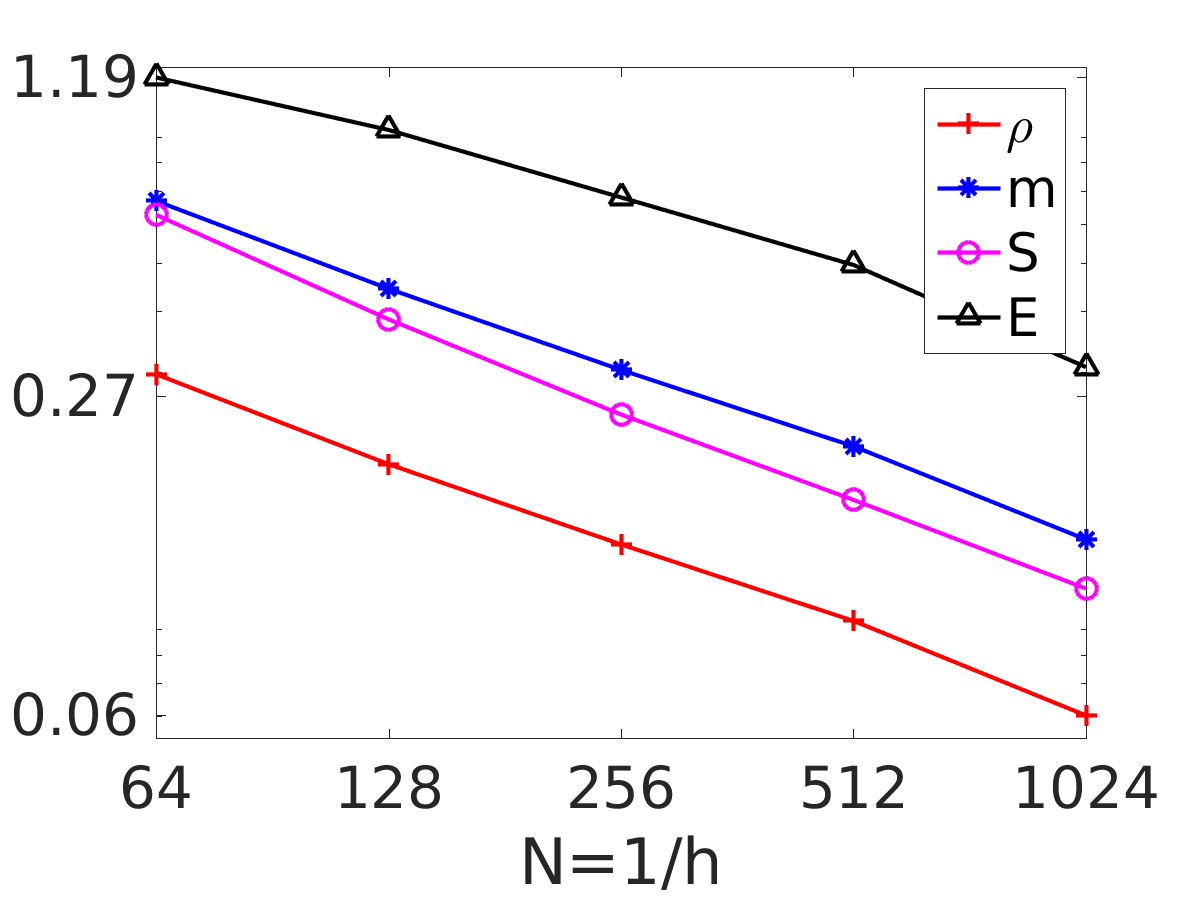}
\caption{upwind FV scheme}\label{fig1c}
\end{subfigure}
\begin{subfigure}{\linewidth} \centering
\includegraphics[width=0.24\textwidth]{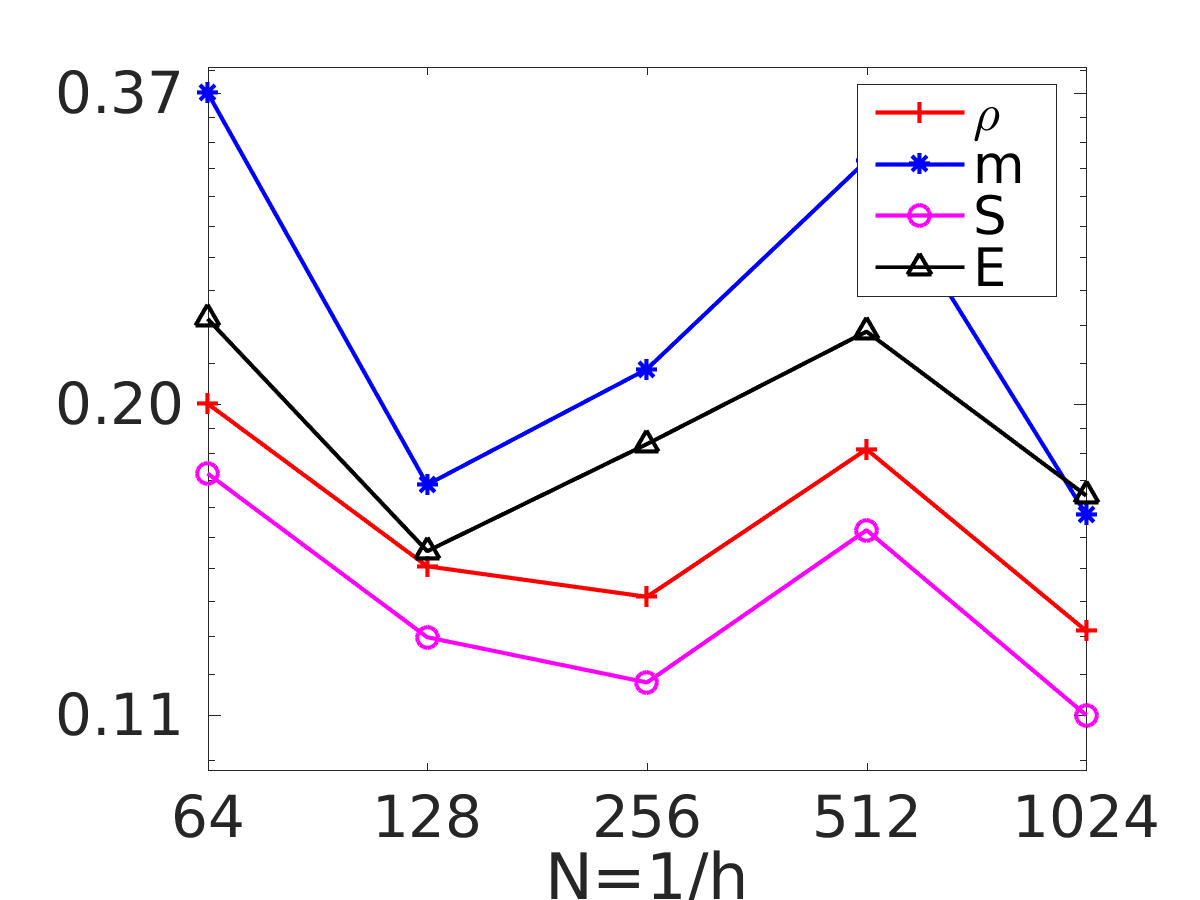}
\includegraphics[width=0.24\textwidth]{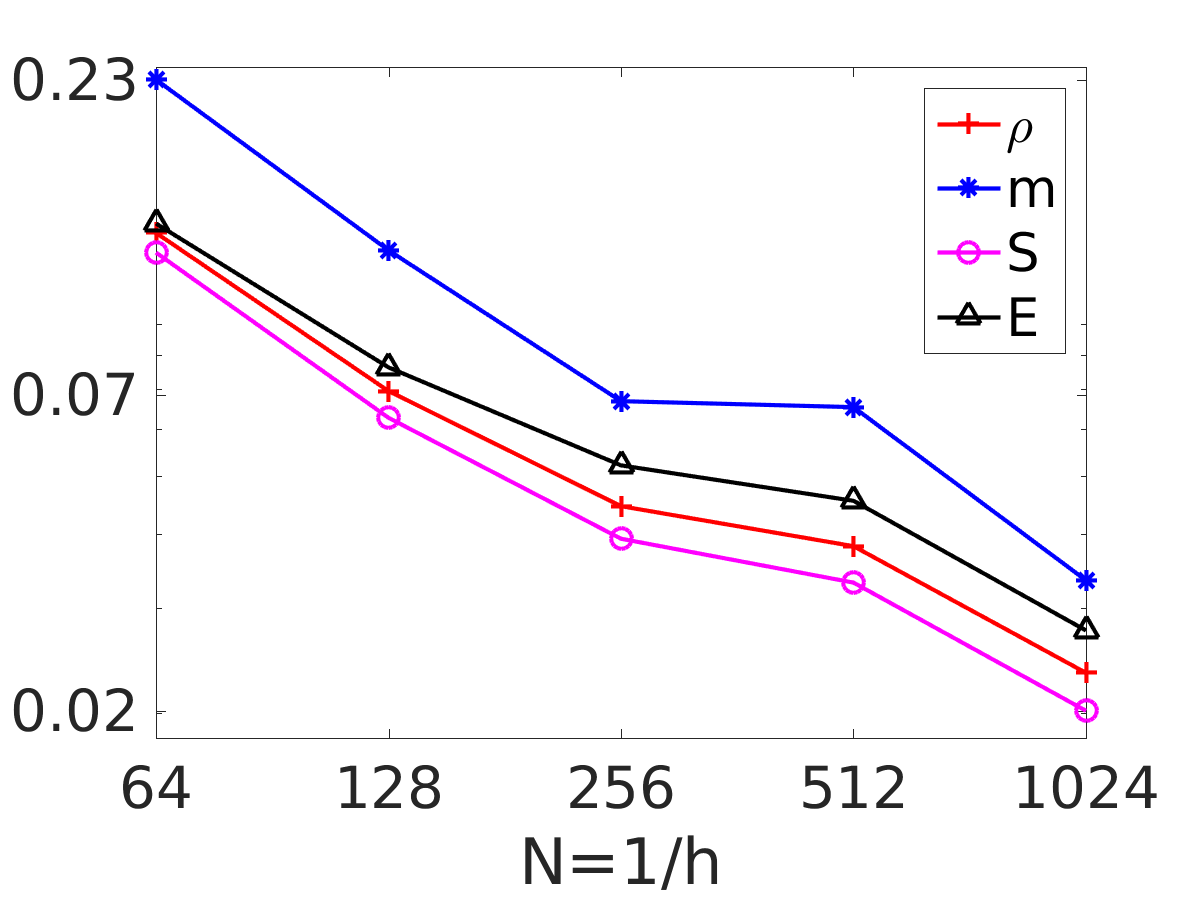}
\includegraphics[width=0.24\textwidth]{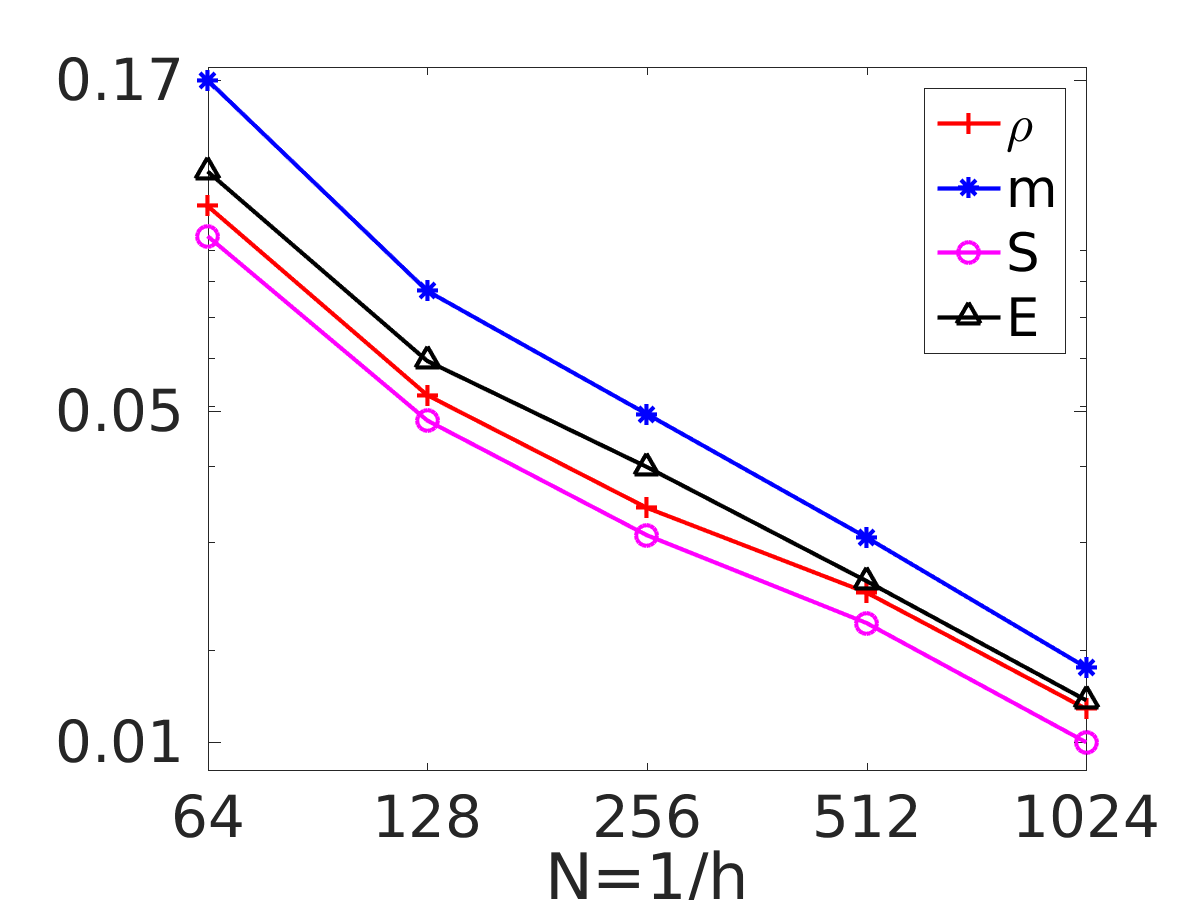}
\includegraphics[width=0.24\textwidth]{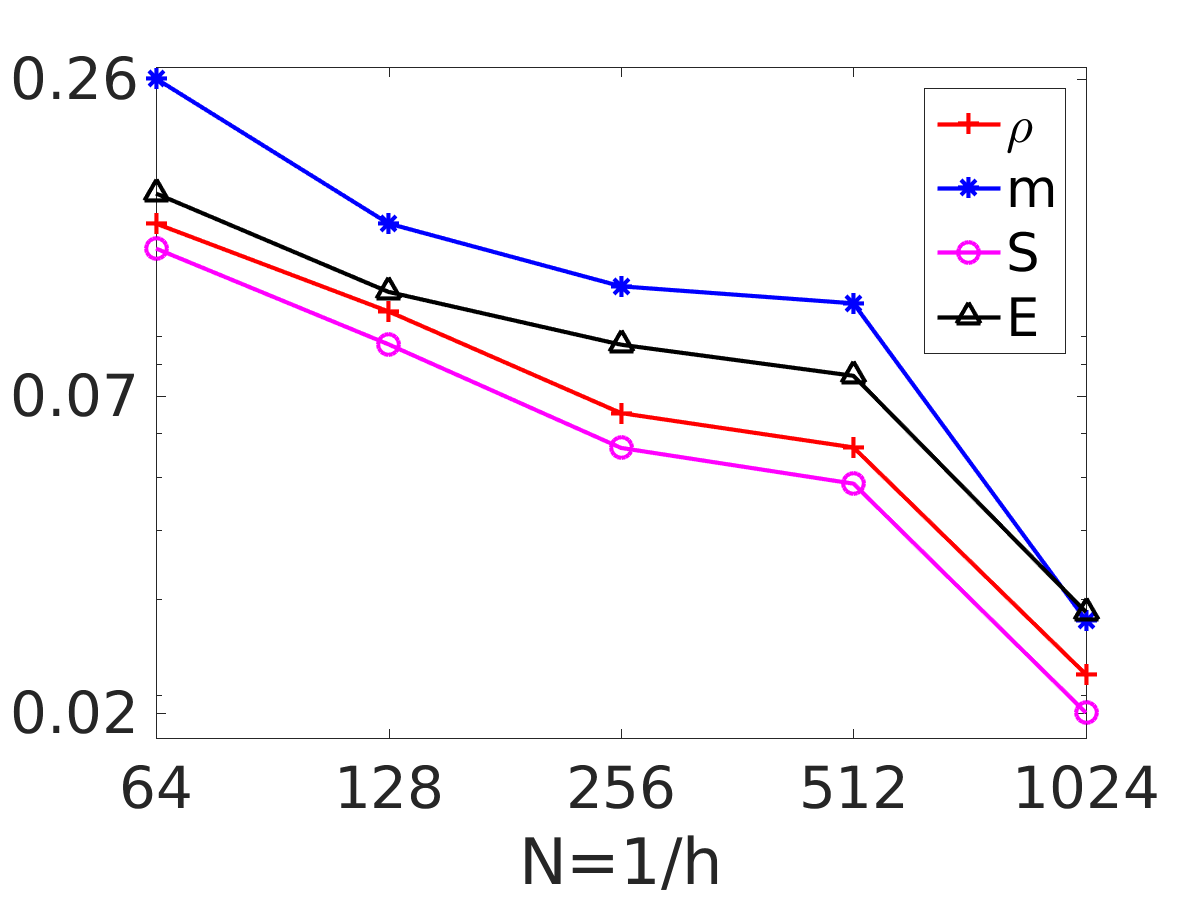}
\caption{GRP scheme}\label{fig1_GRP}
\end{subfigure}
\caption{Experiment~1, convergence study for the Kelvin-Helmholtz problem: $E_1$, $E_2$, $E_3$, and $E_4$ errors (left to right). }
\label{fig1}
\end{figure}

\begin{table}[htbp]
\caption{Experiment~1, convergence study for the Kelvin-Helmholtz problem: $E_1$, $E_2$, $E_3$, and $E_4$ errors for density (left to right).}
\label{tab1}
\begin{tabular}{|c|c|c|c|c|c|c|c|c|}
  \hline
  \multirow{2}{*}{mesh density} & \multicolumn{2}{c|}{$E_1$}  & \multicolumn{2}{c|}{$E_2$} & \multicolumn{2}{c|}{$E_3$} & \multicolumn{2}{c|}{$E_4$} \\
  \cline{2-9}
      & error   & order & error   & order & error   & order  & error   & order \\
   \hline
\multicolumn{9}{c}{(a) \quad FLM scheme} \\\hline
64	 & 3.09e-01 	 & - 	 & 1.81e-01 	 & - 	 & 2.00e-01 	 & - 	 & 2.45e-01 	 & - 	 \\
128	 & 3.13e-01 	 & -0.02 	 & 1.26e-01 	 & 0.52 	 & 1.06e-01 	 & 0.92 	 & 1.64e-01 	 & 0.58 	 \\
256	 & 3.27e-01 	 & -0.06 	 & 9.83e-02 	 & 0.36 	 & 7.44e-02 	 & 0.51 	 & 1.20e-01 	 & 0.45 	 \\
512	 & 3.54e-01 	 & -0.12 	 & 6.94e-02 	 & 0.50 	 & 4.99e-02 	 & 0.58 	 & 8.62e-02 	 & 0.48 	 \\
1024	 & 3.58e-01 	 & -0.01 	 & 4.59e-02 	 & 0.59 	 & 2.97e-02 	 & 0.75 	 & 5.54e-02 	 & 0.64 	 \\
\hline
\multicolumn{9}{c}{(b) \quad upwind FV scheme }\\\hline
64	 & 3.24e-01 	 & - 	 & 2.41e-01 	 & - 	 & 2.40e-01 	 & - 	 & 2.97e-01 	 & - 	 \\
128	 & 3.41e-01 	 & -0.08 	 & 1.59e-01 	 & 0.61 	 & 1.21e-01 	 & 0.99 	 & 1.95e-01 	 & 0.61 	 \\
256	 & 3.64e-01 	 & -0.09 	 & 1.12e-01 	 & 0.50 	 & 7.96e-02 	 & 0.60 	 & 1.34e-01 	 & 0.54 	 \\
512	 & 3.90e-01 	 & -0.10 	 & 7.36e-02 	 & 0.60 	 & 5.23e-02 	 & 0.61 	 & 9.37e-02 	 & 0.52 	 \\
1024	 & 3.87e-01 	 & 0.01 	 & 4.75e-02 	 & 0.63 	 & 3.00e-02 	 & 0.80 	 & 6.02e-02 	 & 0.64 	 \\
\hline
\multicolumn{9}{c}{(c)\quad GRP scheme} \\\hline
64	 & 2.03e-01 	 & - 	 & 1.28e-01 	 & - 	 & 1.06e-01 	 & - 	 & 1.44e-01 	 & - 	 \\
128	 & 1.49e-01 	 & 0.45 	 & 6.95e-02 	 & 0.88 	 & 5.21e-02 	 & 1.03 	 & 9.98e-02 	 & 0.53 	 \\
256	 & 1.40e-01 	 & 0.08 	 & 4.45e-02 	 & 0.64 	 & 3.41e-02 	 & 0.61 	 & 6.52e-02 	 & 0.61 	 \\
512	 & 1.86e-01 	 & -0.41 	 & 3.81e-02 	 & 0.22 	 & 2.48e-02 	 & 0.46 	 & 5.65e-02 	 & 0.21 	 \\
1024	 & 1.31e-01 	 & 0.50 	 & 2.33e-02 	 & 0.71 	 & 1.60e-02 	 & 0.63 	 & 2.18e-02 	 & 1.37 	 \\
\hline
\end{tabular}
\end{table}

\begin{figure}[h]
\begin{subfigure}{0.245\linewidth} \centering
\includegraphics[width=1\textwidth]{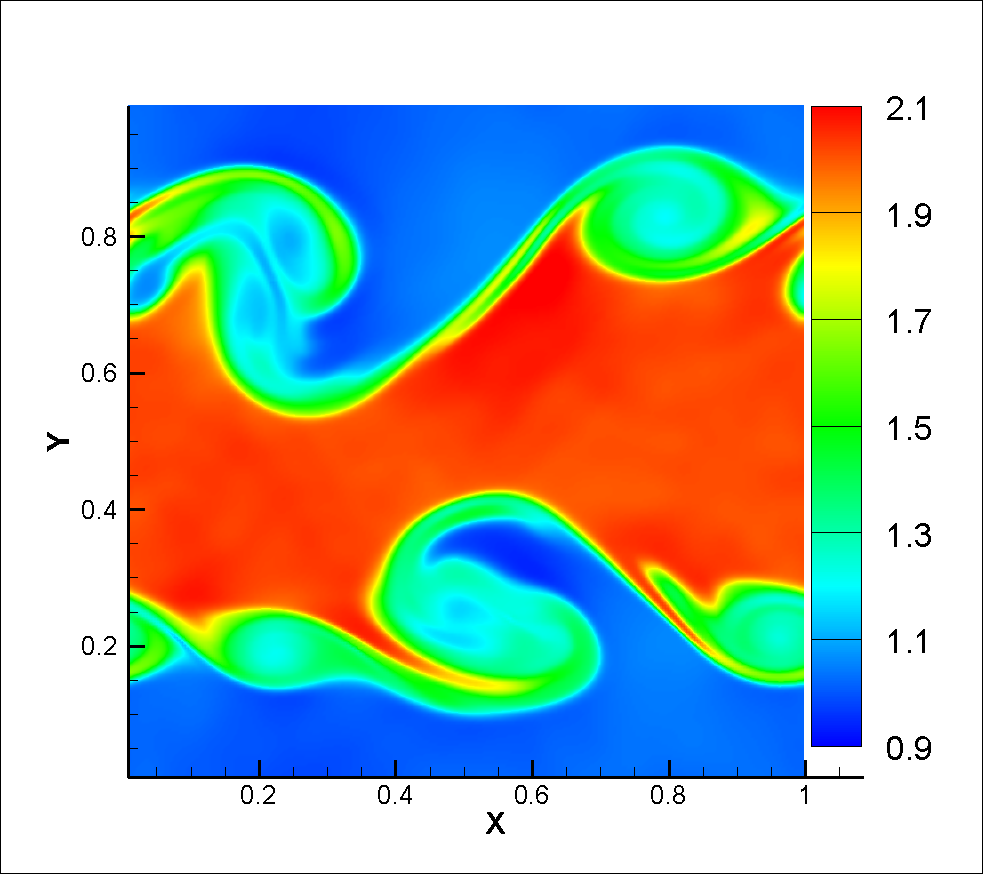}
\caption{$n=256$}
\end{subfigure}
\begin{subfigure}{0.245\linewidth} \centering
\includegraphics[width=1\textwidth]{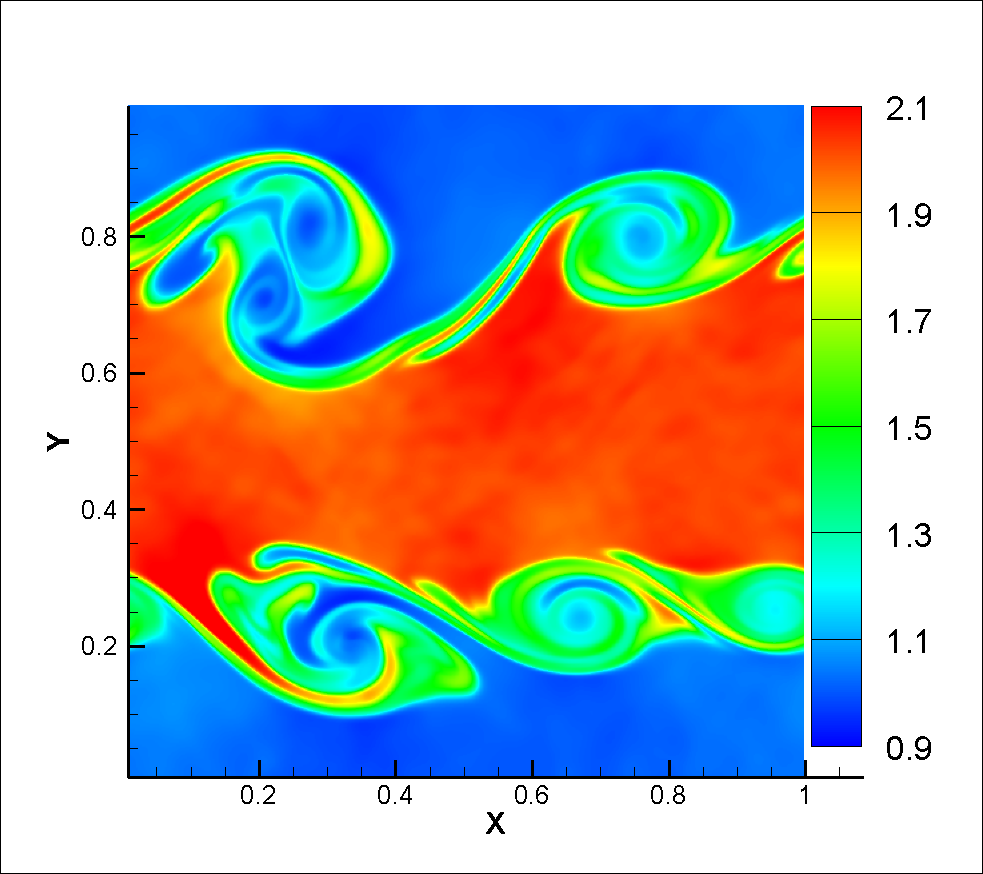}
\caption{$n=512$}
\end{subfigure}
\begin{subfigure}{0.245\linewidth} \centering
\includegraphics[width=1\textwidth]{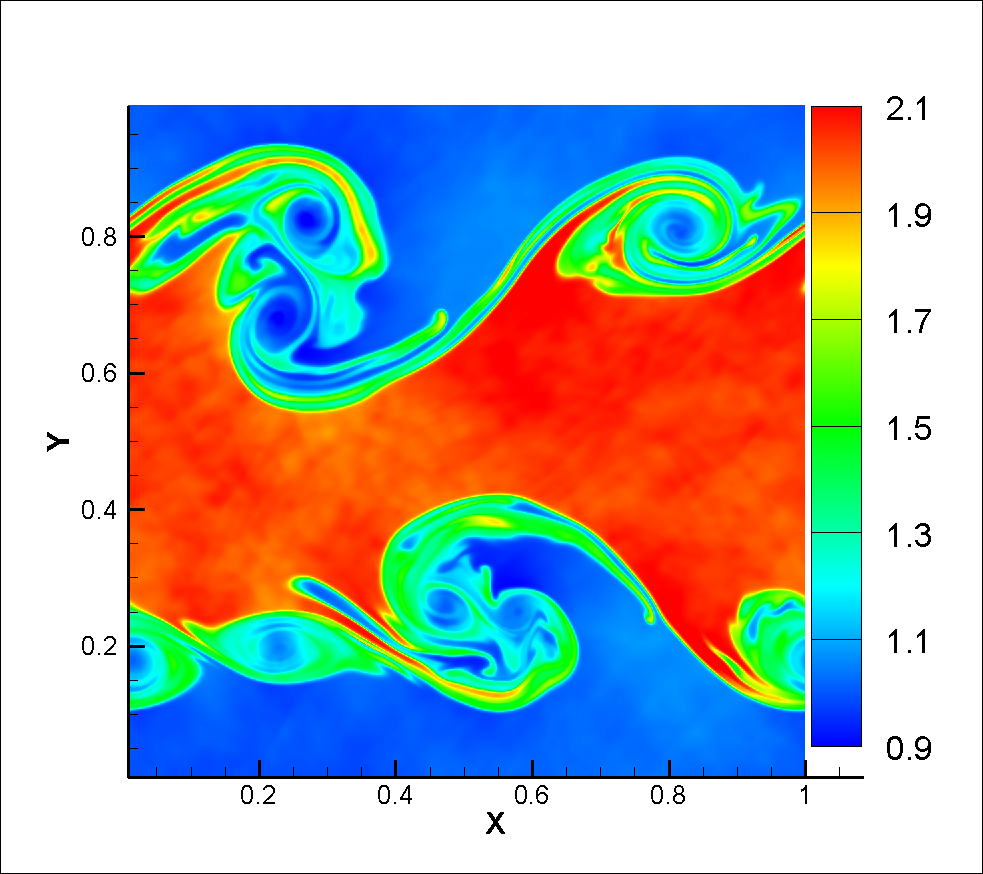}
\caption{$n=1024$}
\end{subfigure}
\begin{subfigure}{0.245\linewidth} \centering
\includegraphics[width=1\textwidth]{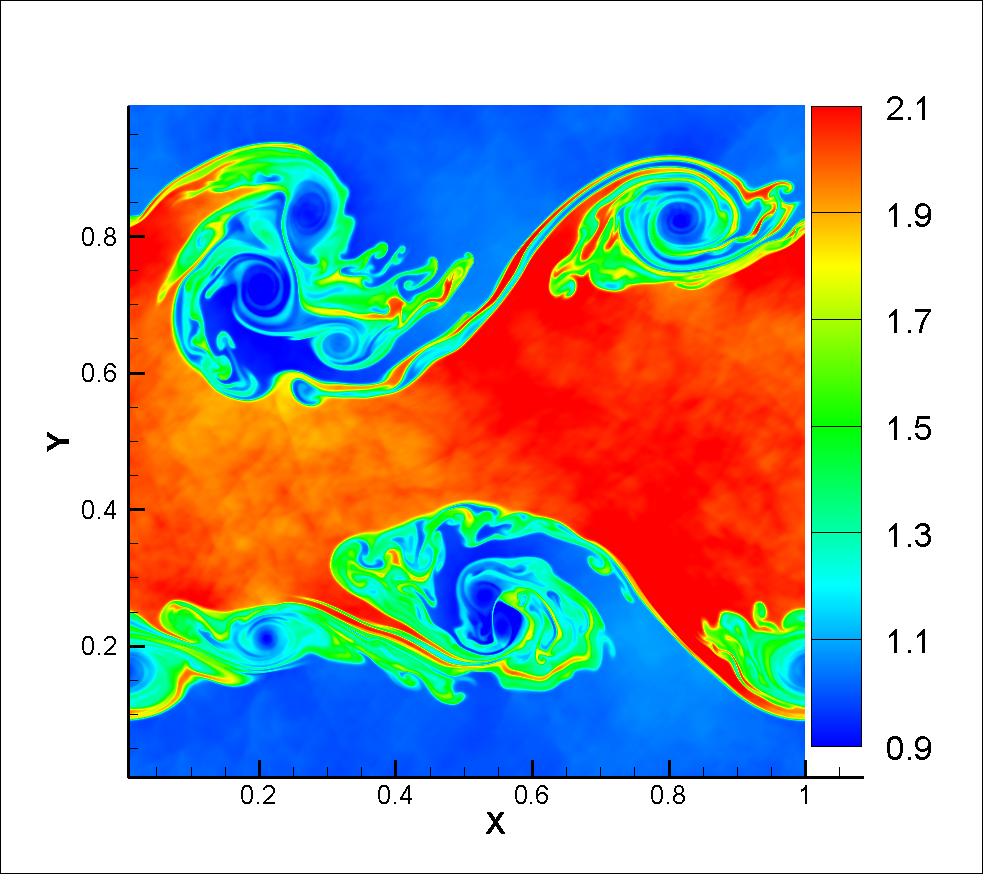}\caption{$n=2048$}
\end{subfigure}
\caption{Experiment~1,  density computed by the GRP scheme at $T=2$ for the Kelvin-Helmholtz problem on a mesh with $n \times n$ cells.}
\label{fig3_GRP}
\end{figure}

\begin{figure}[htbp]
\begin{subfigure}{0.245\linewidth} \centering
\includegraphics[width=1\textwidth]{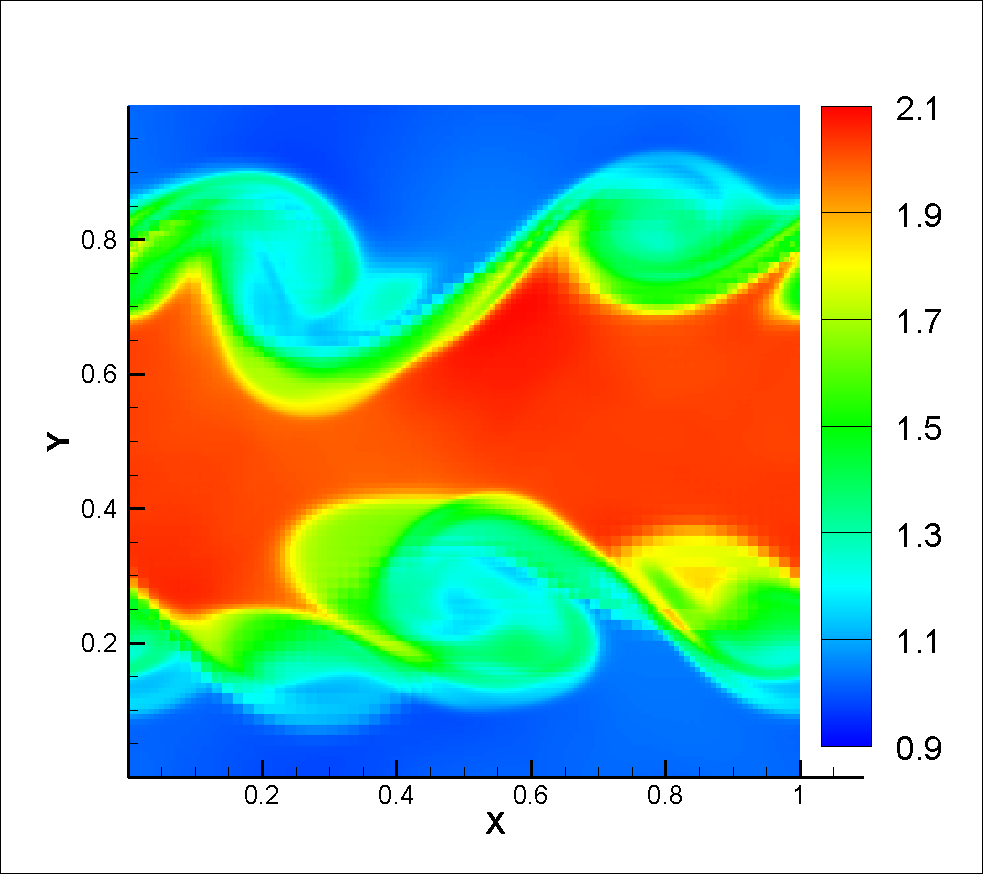}
\caption{$up\ to\ n=256$}
\end{subfigure}
\begin{subfigure}{0.245\linewidth} \centering
\includegraphics[width=1\textwidth]{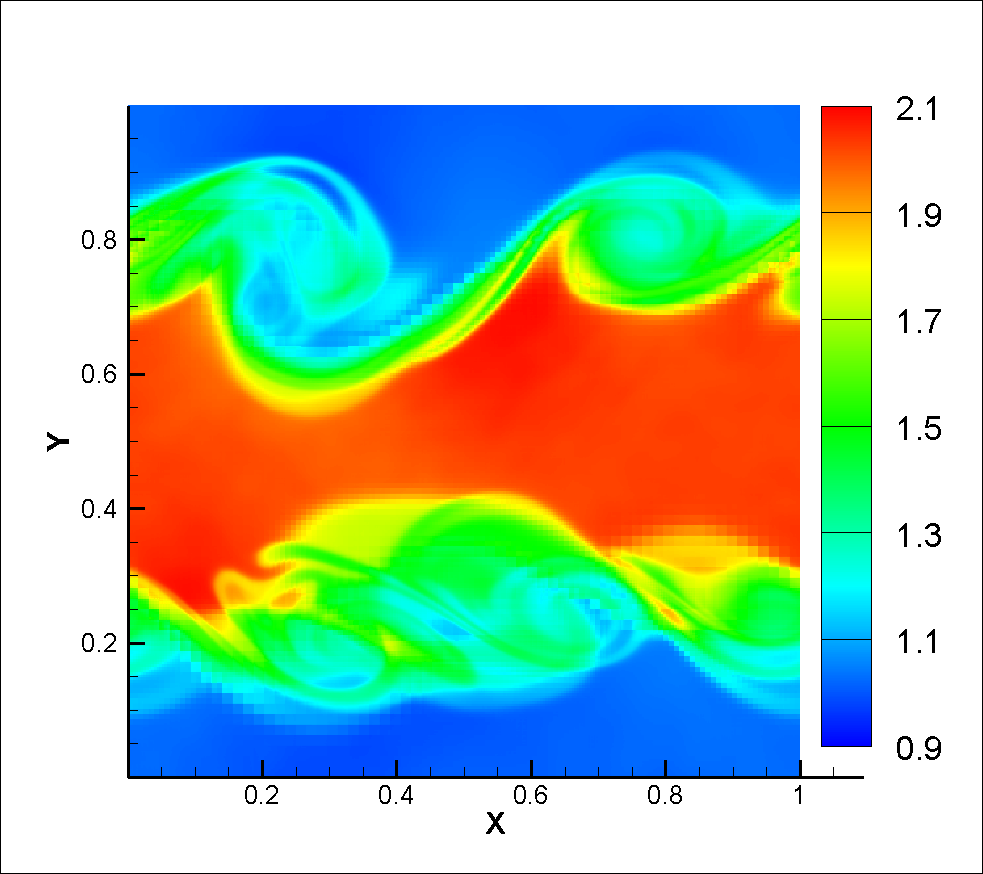}
\caption{$up\ to\ n=512$}
\end{subfigure}
\begin{subfigure}{0.245\linewidth} \centering
\includegraphics[width=1\textwidth]{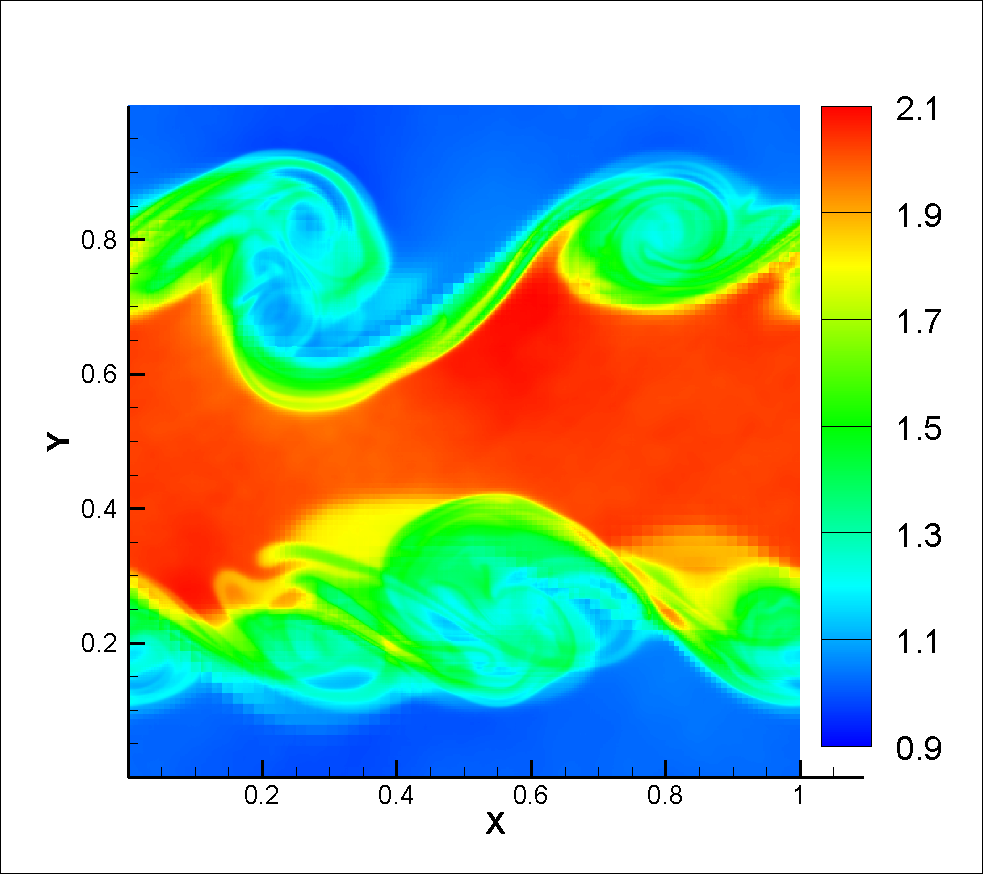}
\caption{$up\ to\ n=1024$}
\end{subfigure}
\begin{subfigure}{0.245\linewidth} \centering
\includegraphics[width=1\textwidth]{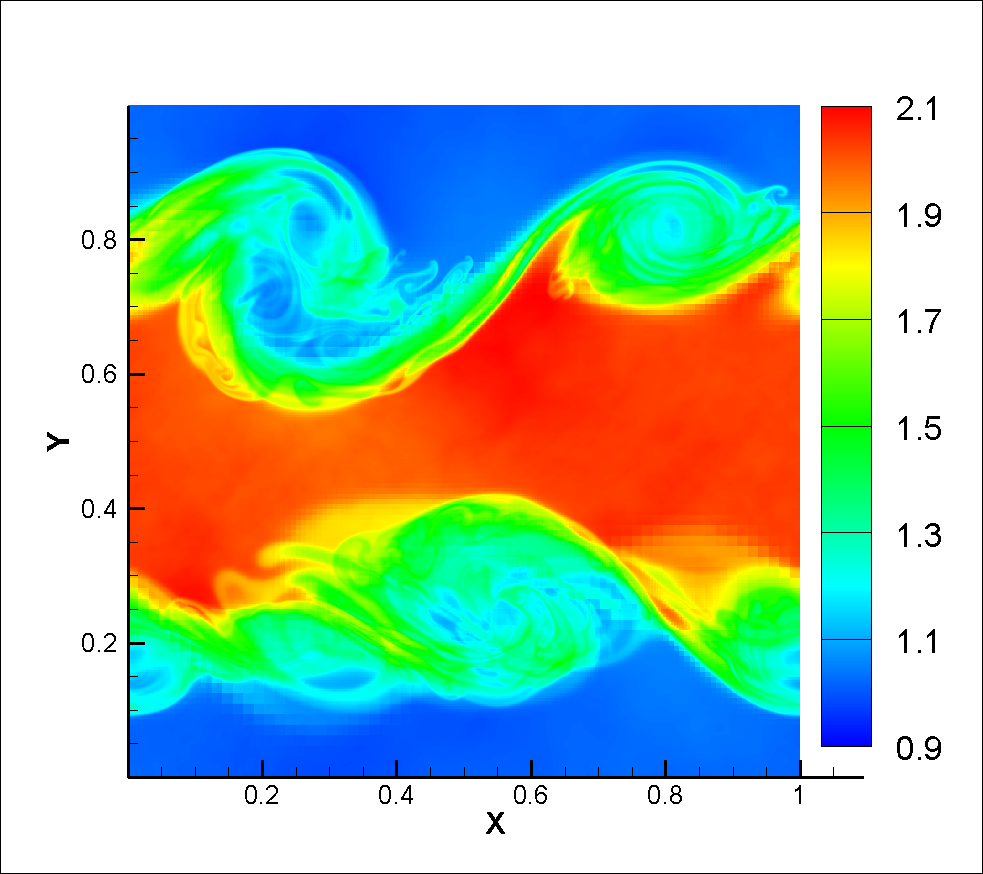}\caption{$up\ to\ n=2048$}
\end{subfigure}
\caption{Experiment~1, Ces\`aro averages of the density computed by the GRP method on  meshes with $j\times j$ cells, $j = 64, 128, \dots, n,$  for the Kelvin-Helmholtz problem.}
\label{fig4_GRP}
\end{figure}

\begin{figure}[h]
\begin{subfigure}{0.245\linewidth} \centering
\includegraphics[width=1\textwidth]{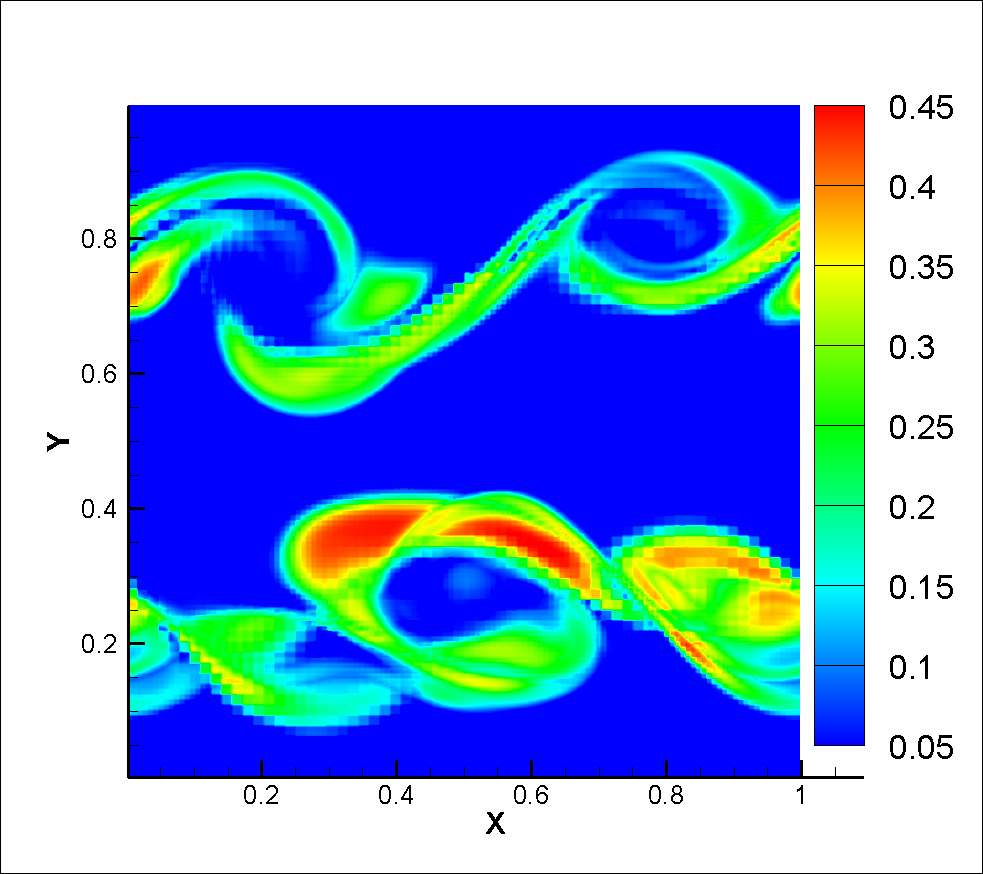}
\caption{$up\ to\ n=256$}
\end{subfigure}
\begin{subfigure}{0.245\linewidth} \centering
\includegraphics[width=1\textwidth]{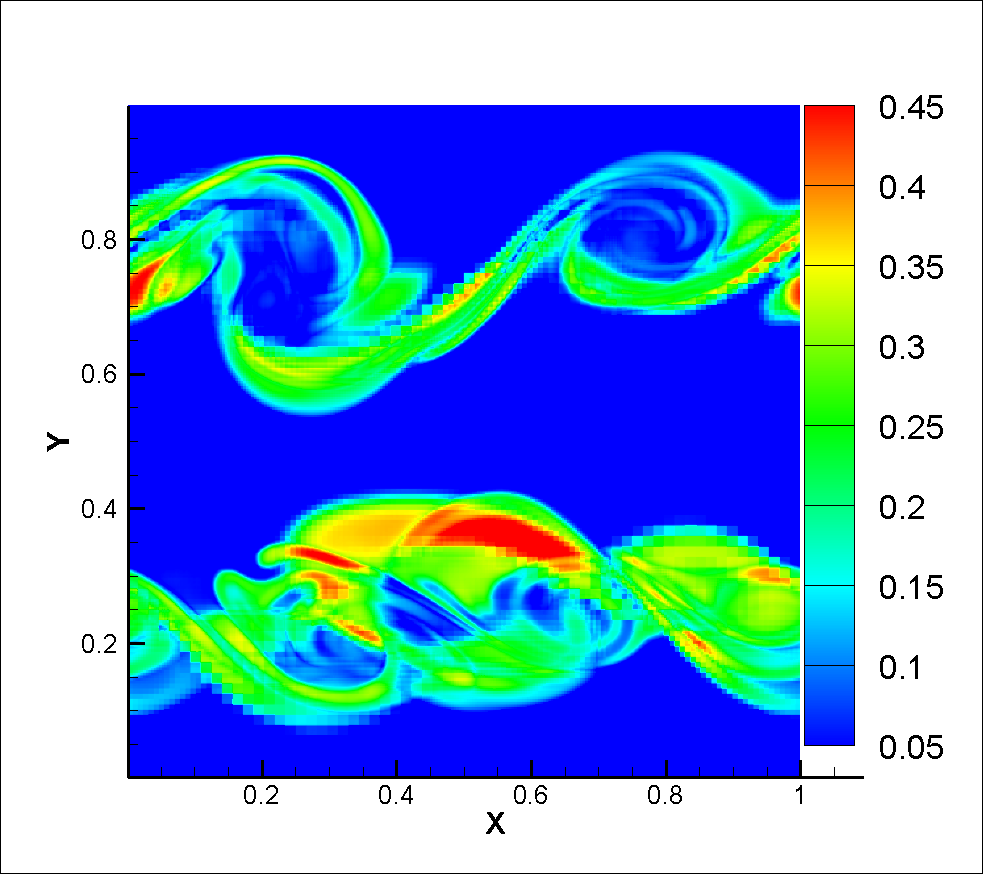}
\caption{$up\ to\ n=512$}
\end{subfigure}
\begin{subfigure}{0.245\linewidth} \centering
\includegraphics[width=1\textwidth]{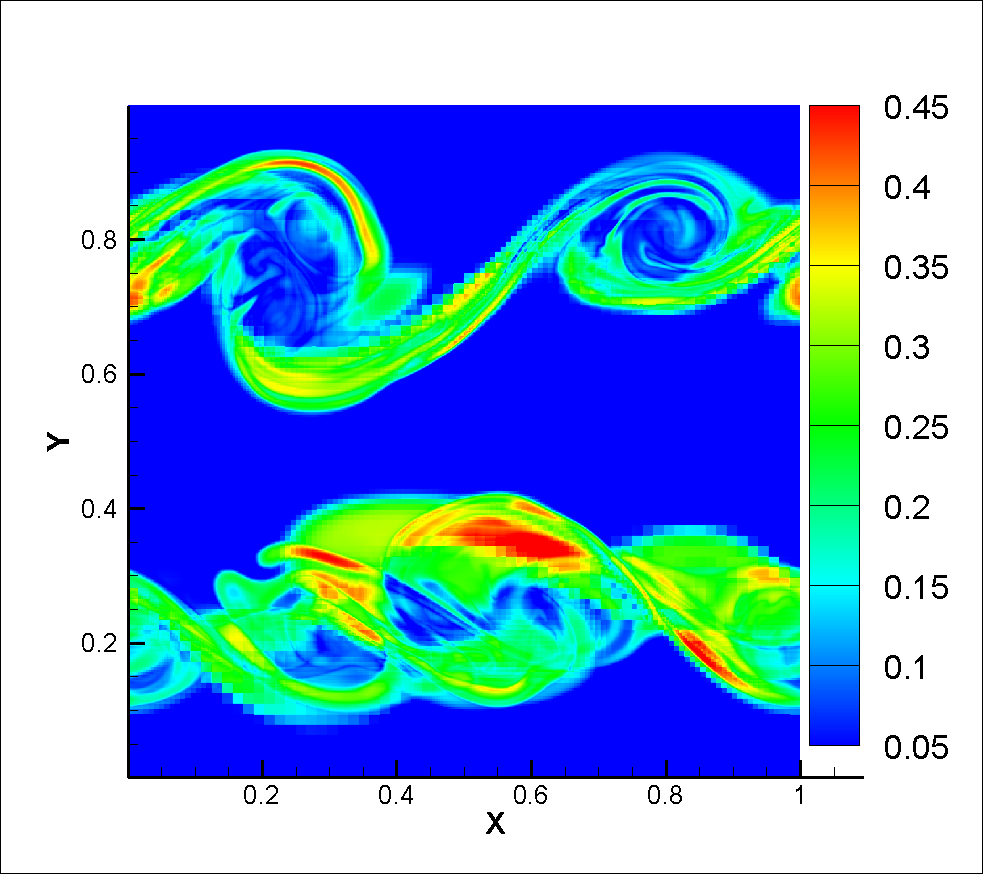}
\caption{$up\ to\ n=1024$}
\end{subfigure}
\begin{subfigure}{0.245\linewidth} \centering
\includegraphics[width=1\textwidth]{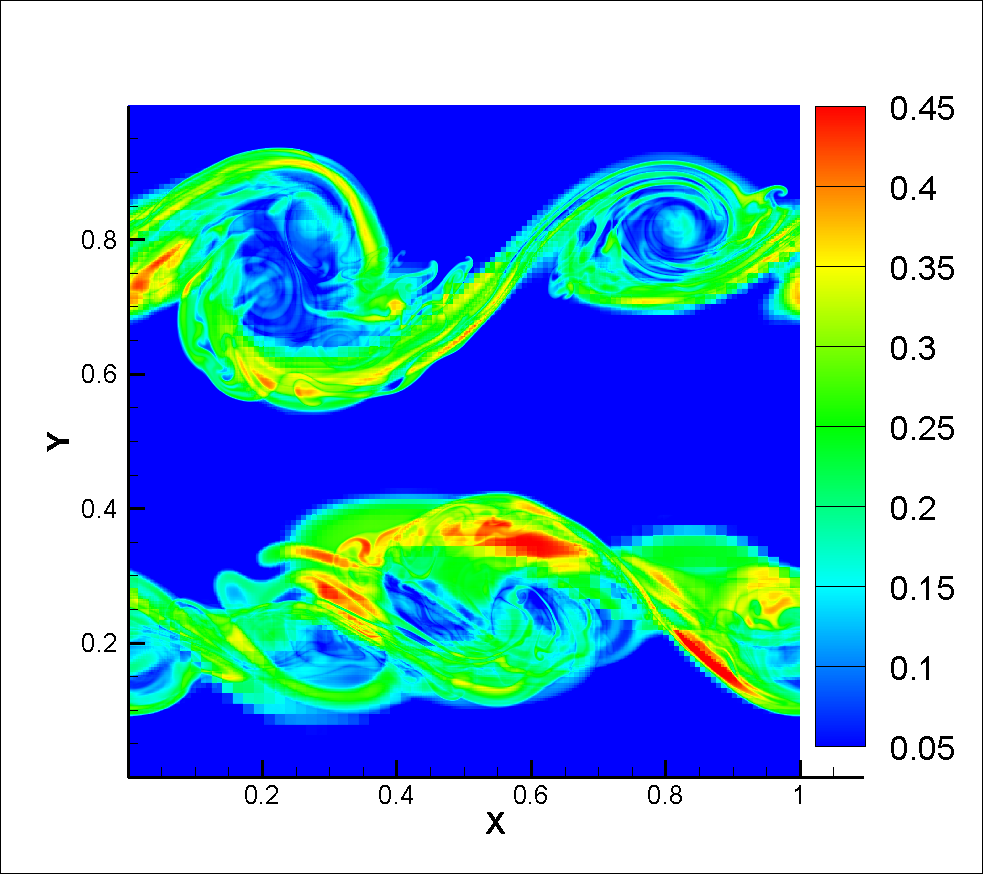}\caption{$up\  to\ n=2048$}
\end{subfigure}
\caption{Experiment~1, first variance of the density computed by the GRP method on  meshes with $j\times j$ cells, $j = 64, 128, \dots, n,$  for the Kelvin-Helmholtz problem.}
\label{fig5_GRP}
\end{figure}

\paragraph{Experiment 2.} The second experiment is the Richtmyer-Meshkov problem \cite{Meshkov, Richtmyer} of
complex interaction of strong shocks with unstable interfaces. The initial data are given as
\[
u_1=u_2=0, \quad
p(x)=\left\{
\begin{array}{ll}
20, & \text{if}\ r<0.1,\\
1, & \text{otherwise,}
\end{array}
\right. \quad
\rho(\bf x)=\left\{
\begin{array}{ll}
2, & \text{if}\ r <I(x,\omega),\\
1, & \text{otherwise,}
\end{array}
\right.
\]
where $r=\sqrt{(x_1-0.5)^2 + (x_2-0.5)^2}$ and
the radial density interface $I(x,\omega)=0.25 + 0.01 Y(x,\omega)$ is perturbed by
\[
Y(x,\omega)=\sum_{n=1}^{m}a^n(\omega)\cos(\phi+b^n(\omega)).
\]
Here $\phi =\arccos(x_2/r)$, and the parameters $a^n, b^n$ are random data chosen in the same way as in Experiment 1.
We have computed numerical solutions by the FLM, upwind FV and GRP schemes until the final time $T=4.$
At this time  the leading shock waves re-entered from the corners due to the periodic boundary conditions and interact with each others
leading to complex small-scales vortices.

Table~\ref{tab2} and Figure~\ref{fig3} demonstrate that there is no convergence of single numerical solutions in the classical sense, see the first
column, but we have $L^1$-convergence of the Ces\`aro averages of  numerical solutions and their first variance  as well as $\K$-convergence of the Wasserstein distance of the corresponding Dirac distributions. Numerical approximation of the density, its Ces\`aro averages and of the first variance at time $T=4$ is illustrated by Figures~\ref{fig_w4}--\ref{fig_w6}. We can again confirm that single numerical solutions do not converge in the classical sense,
but we have $\mathcal{K}$-convergence that is strong in $L^1(\Omega).$

Let us point out  that in both experiments convergence of the Ces\`aro averages of numerical solutions and their first variance 
as well as ${\mathcal K}$-convergence of the
 corresponding Dirac distributions is approximately of order 0.5.

\begin{figure}[hb!]
\begin{subfigure}{\linewidth} \centering
\includegraphics[width=0.24\textwidth]{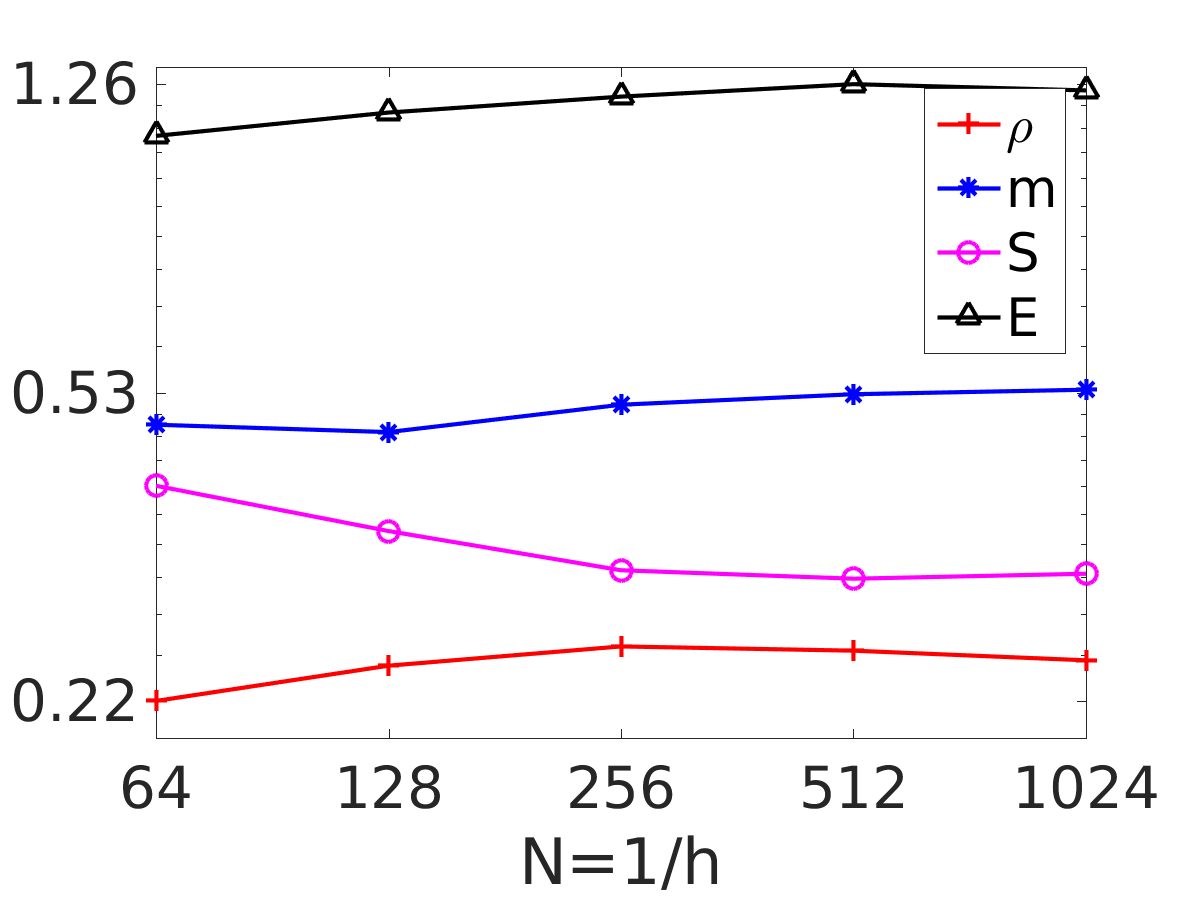}
\includegraphics[width=0.24\textwidth]{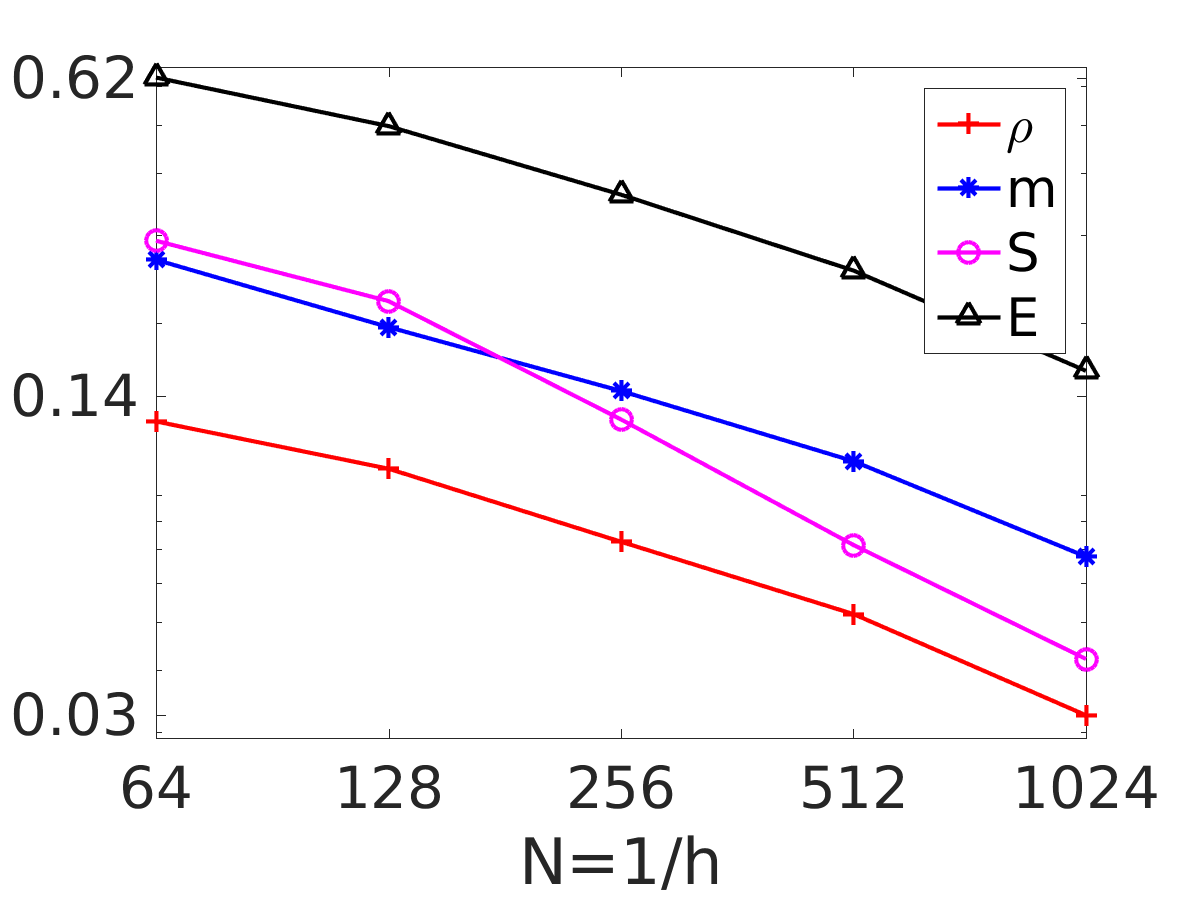}
\includegraphics[width=0.24\textwidth]{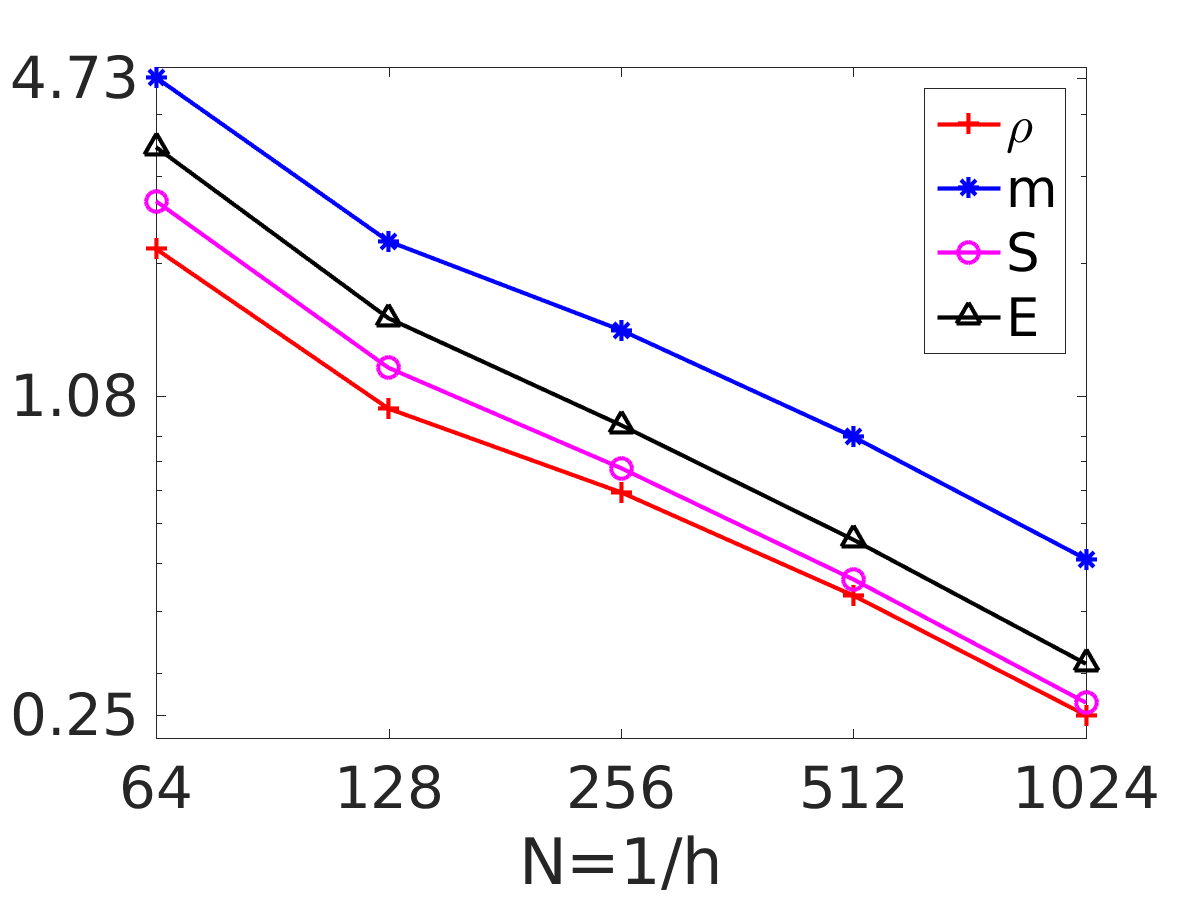}
\includegraphics[width=0.24\textwidth]{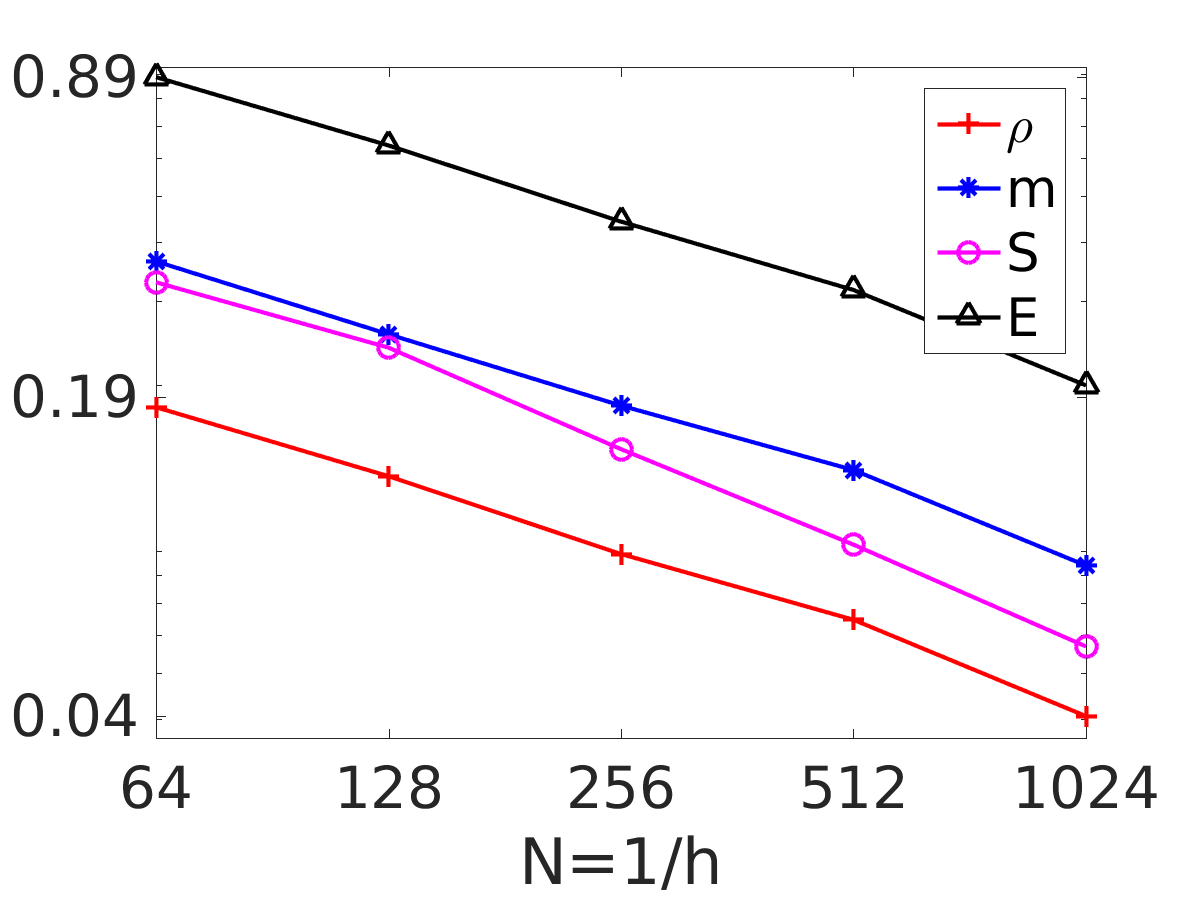}
\caption{FLM scheme }\label{fig8b}
\end{subfigure}
\begin{subfigure}{\linewidth} \centering
\includegraphics[width=0.24\textwidth]{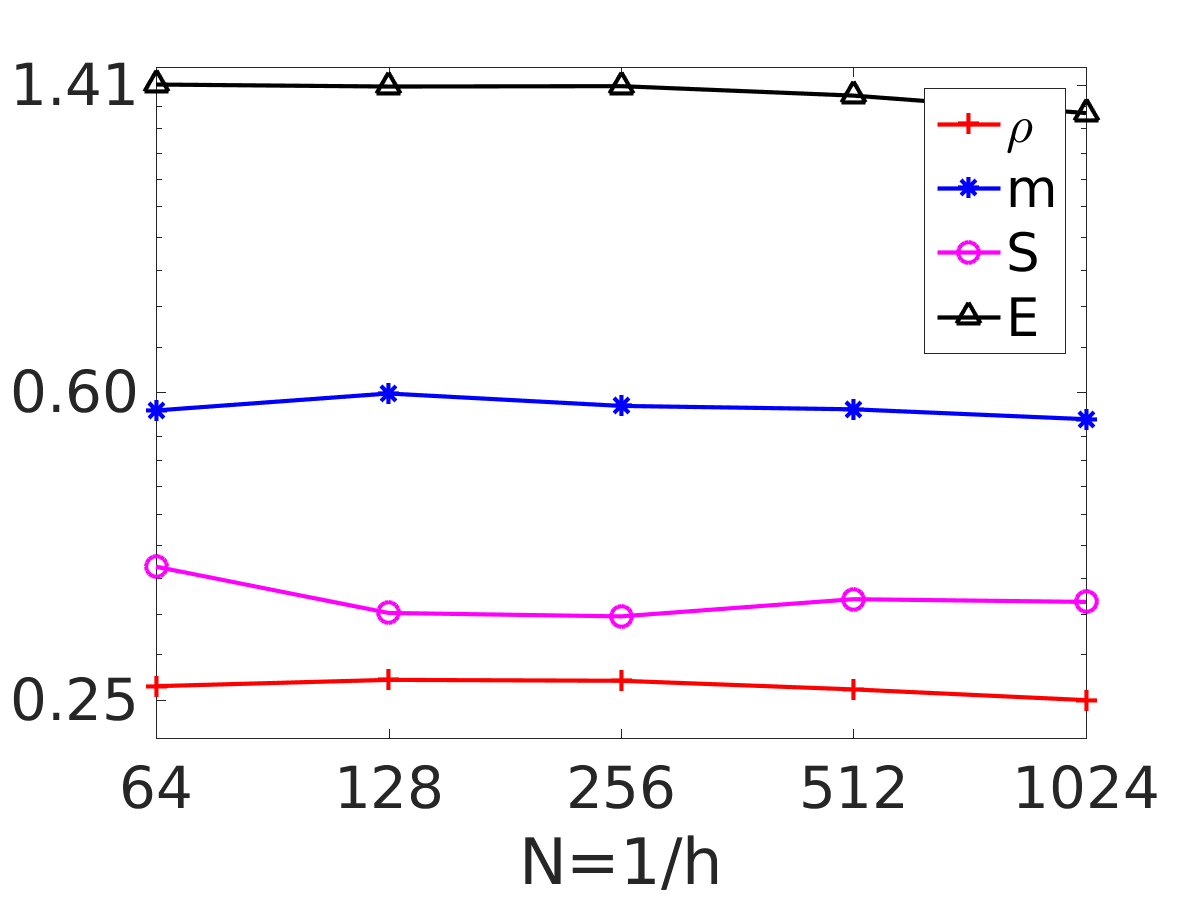}
\includegraphics[width=0.24\textwidth]{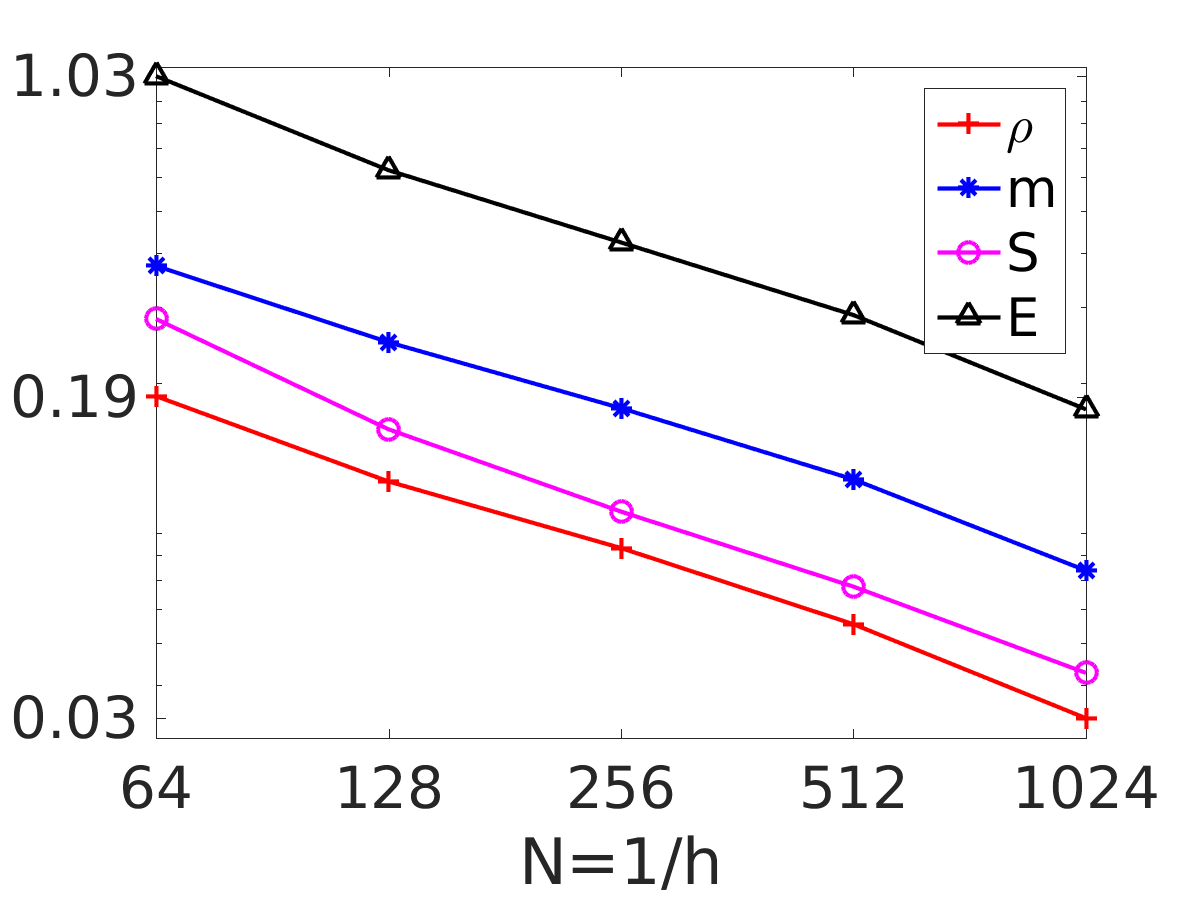}
\includegraphics[width=0.24\textwidth]{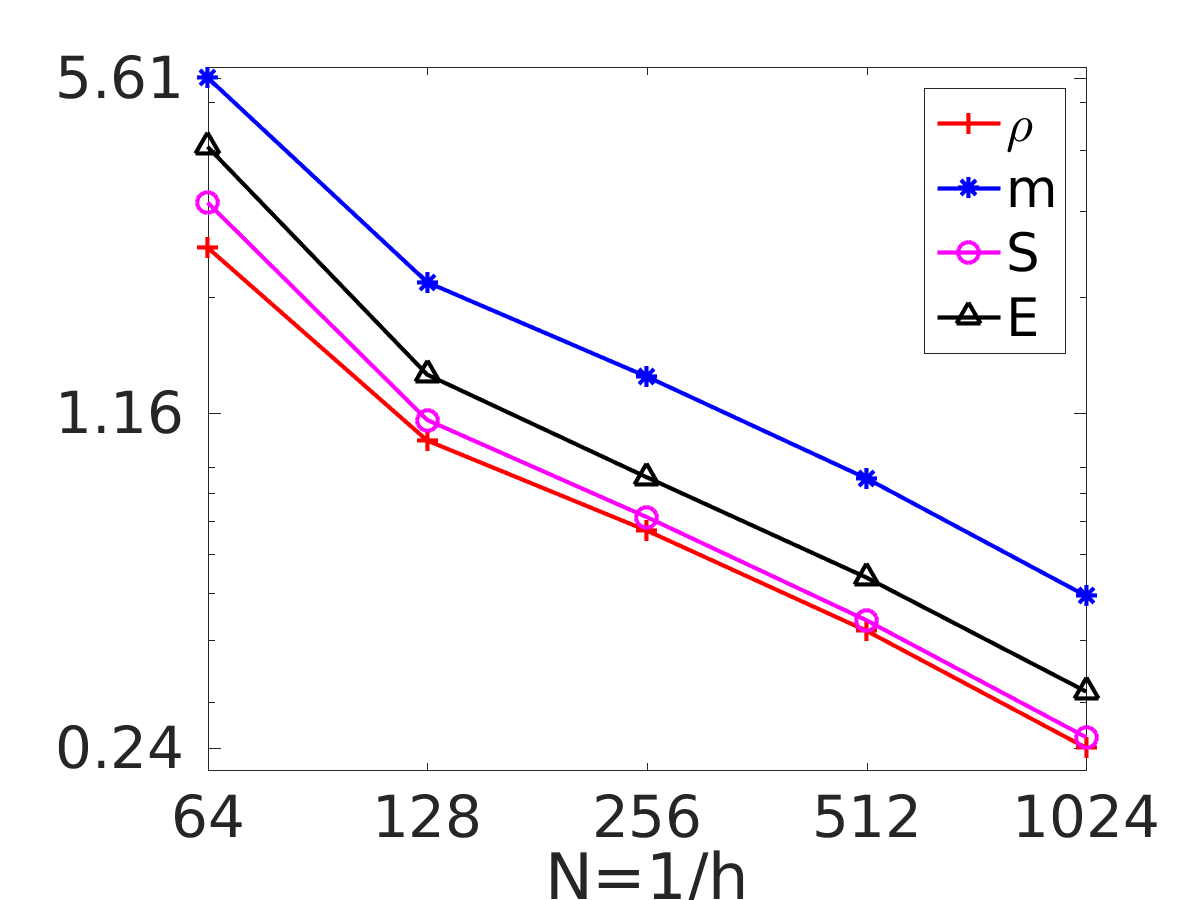}
\includegraphics[width=0.24\textwidth]{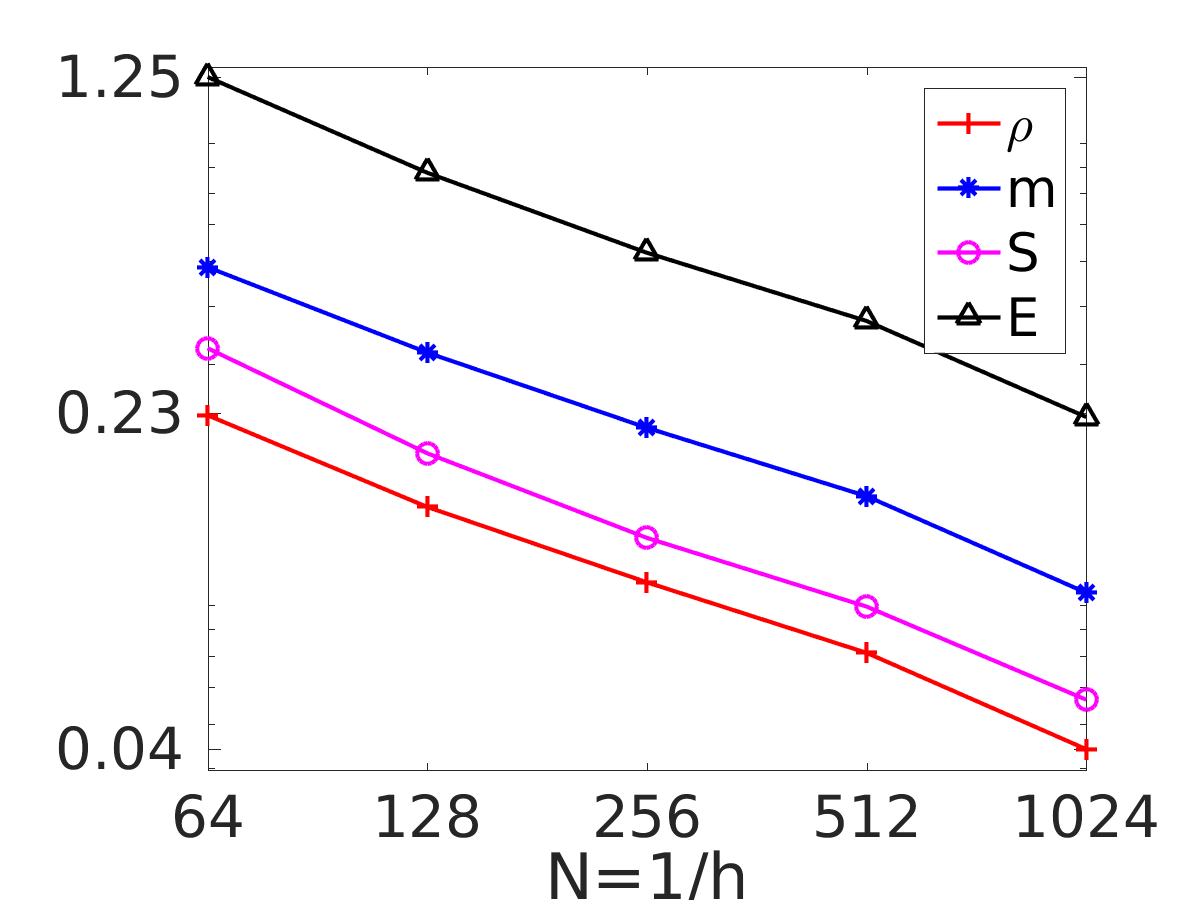}
\caption{upwind FV scheme}\label{fig8c}
\end{subfigure}
\begin{subfigure}{\linewidth} \centering
\includegraphics[width=0.24\textwidth]{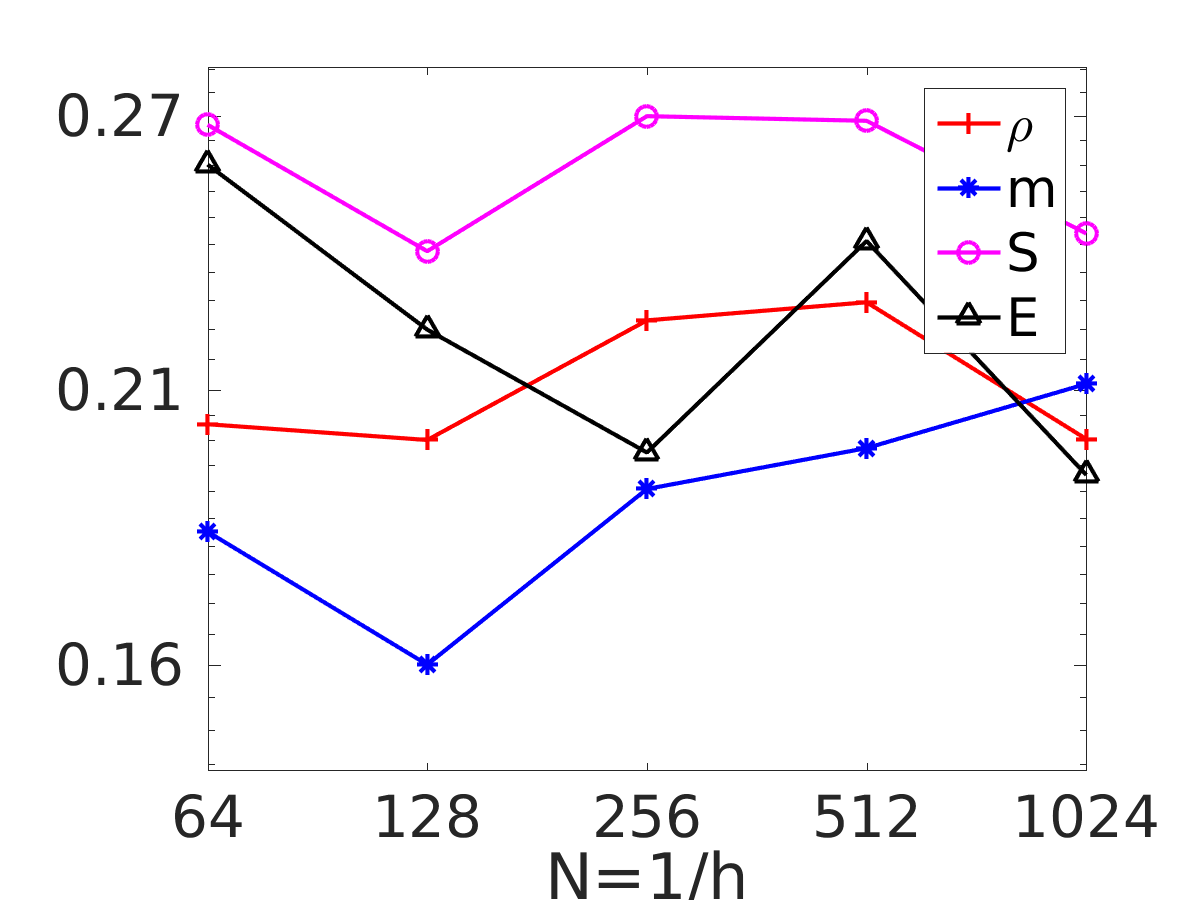}
\includegraphics[width=0.24\textwidth]{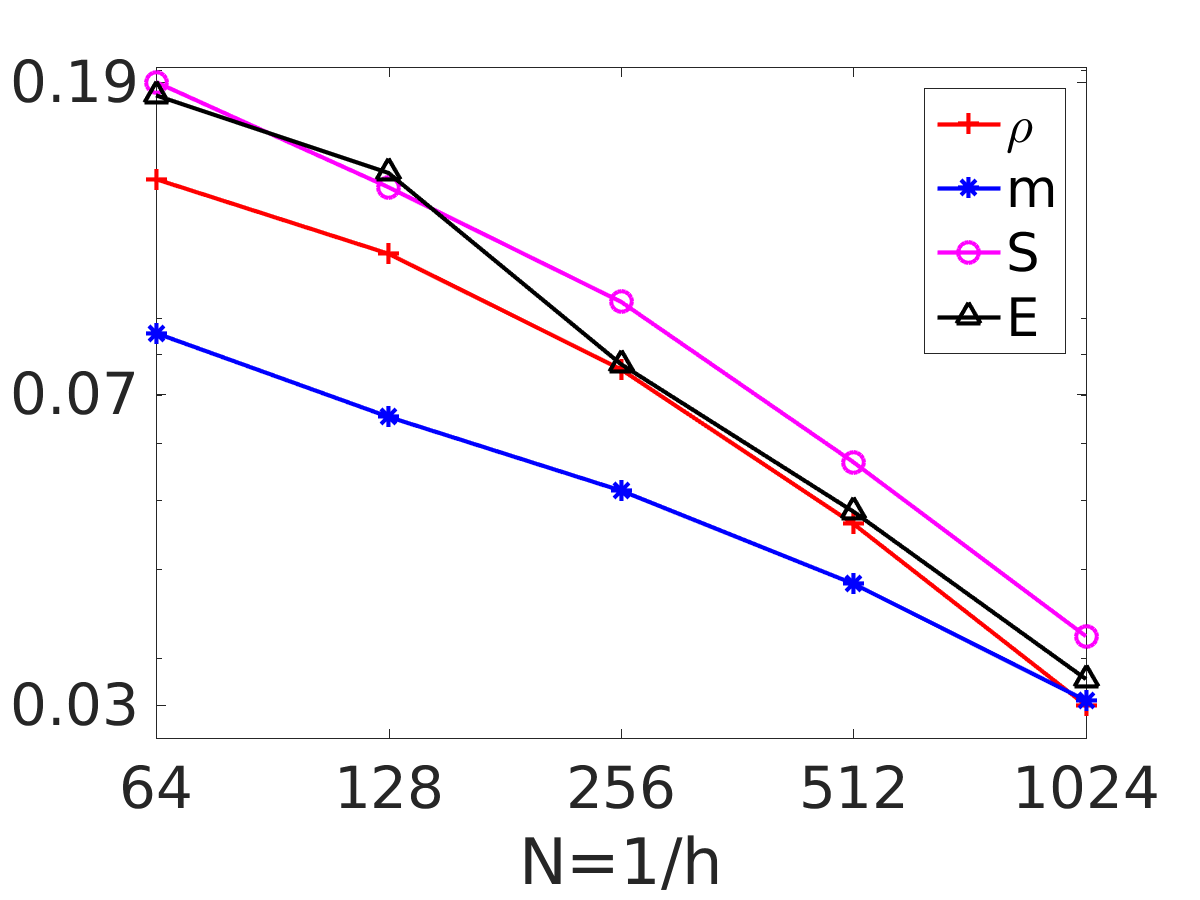}
\includegraphics[width=0.24\textwidth]{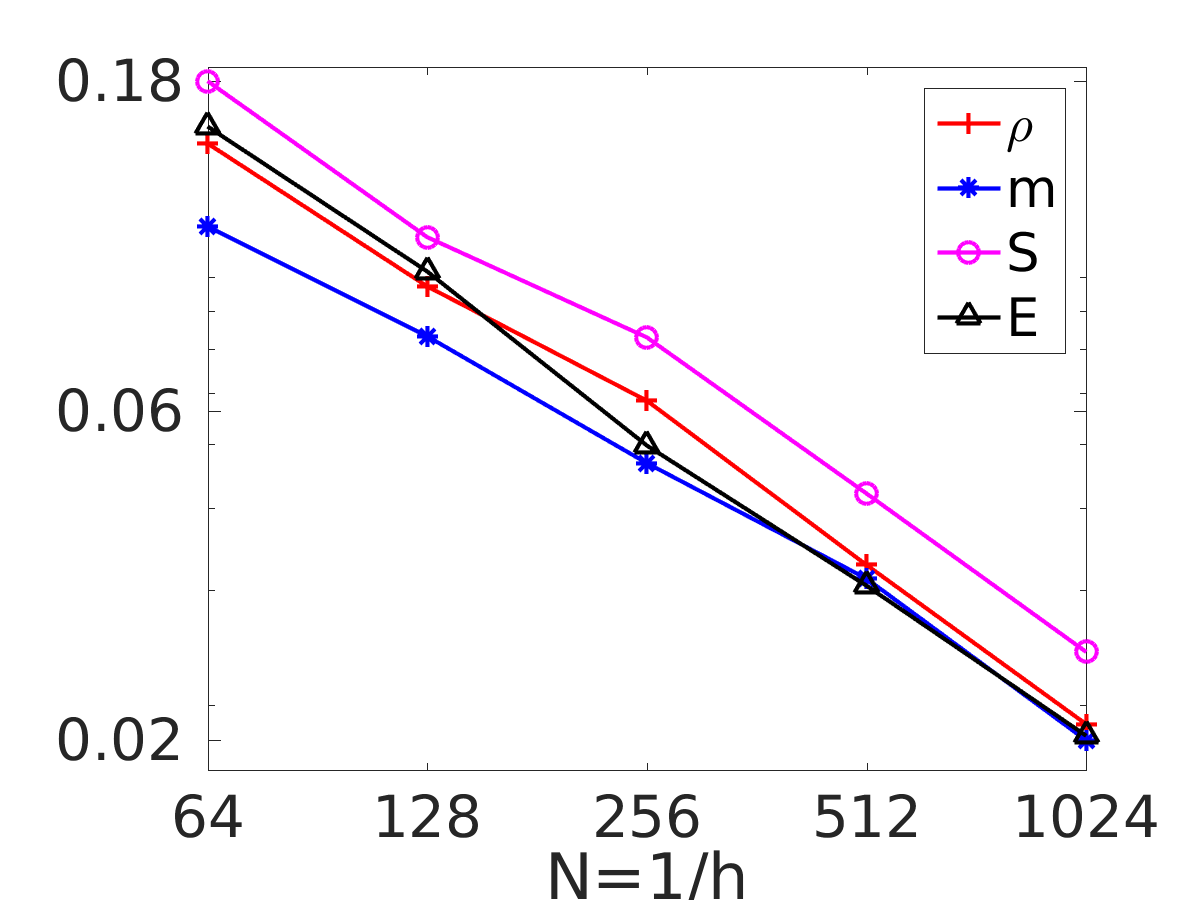}
\includegraphics[width=0.24\textwidth]{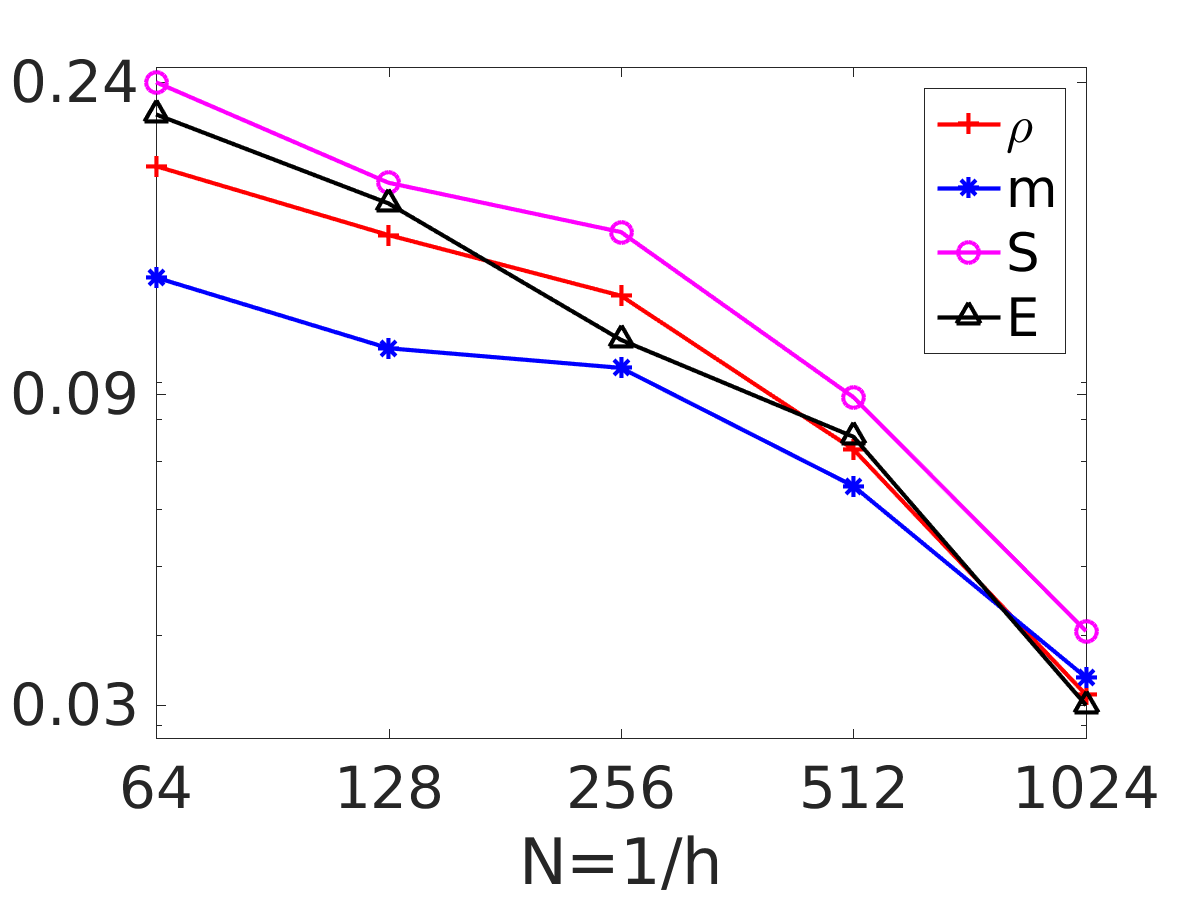}
\caption{GRP scheme}\label{fig8_GRP}
\end{subfigure}
\caption{Experiment~2, convergence study for the Richtmyer-Meshkov problem: $E_1, E_2, E_3$ and $E_4$
errors (left to right). }
\label{fig3}
\end{figure}


\begin{table}[h!]
\caption{Experiment~2, convergence study for the Richtmyer-Meshkov problem: $E_1, E_2, E_3$ and $E_4$
errors for density (left to right).}
\label{tab2}
\begin{tabular}{|c|c|c|c|c|c|c|c|c|}
  \hline
  \multirow{2}{*}{mesh density} & \multicolumn{2}{c|}{$E_1$}  & \multicolumn{2}{c|}{$E_2$} & \multicolumn{2}{c|}{$E_3$} & \multicolumn{2}{c|}{$E_4$} \\
  \cline{2-9}
      & error   & order & error   & order & error   & order  & error   & order \\
   \hline
\multicolumn{9}{c}{(a) \quad FLM  scheme} \\\hline
64	 & 2.19e-01 	 & - 	 & 1.27e-01 	 & - 	 & 2.14e+00 	 & - 	 & 1.80e-01 	 & - 	 \\
128	 & 2.42e-01 	 & -0.14 	 & 1.02e-01 	 & 0.32 	 & 1.02e+00 	 & 1.07 	 & 1.29e-01 	 & 0.48 	 \\
256	 & 2.56e-01 	 & -0.08 	 & 7.25e-02 	 & 0.49 	 & 6.93e-01 	 & 0.56 	 & 8.86e-02 	 & 0.54 	 \\
512	 & 2.53e-01 	 & 0.02 	 & 5.19e-02 	 & 0.48 	 & 4.30e-01 	 & 0.69 	 & 6.45e-02 	 & 0.46 	 \\
1024	 & 2.46e-01 	 & 0.04 	 & 3.25e-02 	 & 0.68 	 & 2.47e-01 	 & 0.80 	 & 4.05e-02 	 & 0.67 	 \\ \hline
\multicolumn{9}{c}{(b) \quad upwind FV scheme} \\\hline
64	 & 2.65e-01 	 & - 	 & 1.86e-01 	 & - 	 & 2.53e+00 	 & - 	 & 2.32e-01 	 & - 	 \\
128	 & 2.70e-01 	 & -0.03 	 & 1.18e-01 	 & 0.66 	 & 1.02e+00 	 & 1.31 	 & 1.47e-01 	 & 0.66 	 \\
256	 & 2.69e-01 	 & 0.00 	 & 8.28e-02 	 & 0.51 	 & 6.69e-01 	 & 0.61 	 & 1.01e-01 	 & 0.54 	 \\
512	 & 2.63e-01 	 & 0.03 	 & 5.52e-02 	 & 0.59 	 & 4.18e-01 	 & 0.68 	 & 7.11e-02 	 & 0.51 	 \\
1024	 & 2.55e-01 	 & 0.04 	 & 3.34e-02 	 & 0.72 	 & 2.41e-01 	 & 0.79 	 & 4.40e-02 	 & 0.69 	 \\ \hline
\multicolumn{9}{c}{(c) \quad GRP scheme} \label{tab2_GRP} \\\hline
64	 & 2.01e-01 	 & - 	 & 1.41e-01 	 & - 	 & 1.44e-01 	 & - 	 & 1.80e-01 	 & - 	 \\
128	 & 1.98e-01 	 & 0.02 	 & 1.11e-01 	 & 0.35 	 & 8.71e-02 	 & 0.72 	 & 1.44e-01 	 & 0.32 	 \\
256	 & 2.23e-01 	 & -0.17 	 & 7.61e-02 	 & 0.54 	 & 5.83e-02 	 & 0.58 	 & 1.19e-01 	 & 0.28 	 \\
512	 & 2.27e-01 	 & -0.03 	 & 4.62e-02 	 & 0.72 	 & 3.27e-02 	 & 0.83 	 & 7.26e-02 	 & 0.71 	 \\
1024	 & 1.98e-01 	 & 0.20 	 & 2.57e-02 	 & 0.85 	 & 1.87e-02 	 & 0.81 	 & 3.30e-02 	 & 1.14 	 \\
\hline
\end{tabular}
\end{table}

\begin{figure}[hb]
\begin{subfigure}{0.245\linewidth} \centering
\includegraphics[width=1\textwidth]{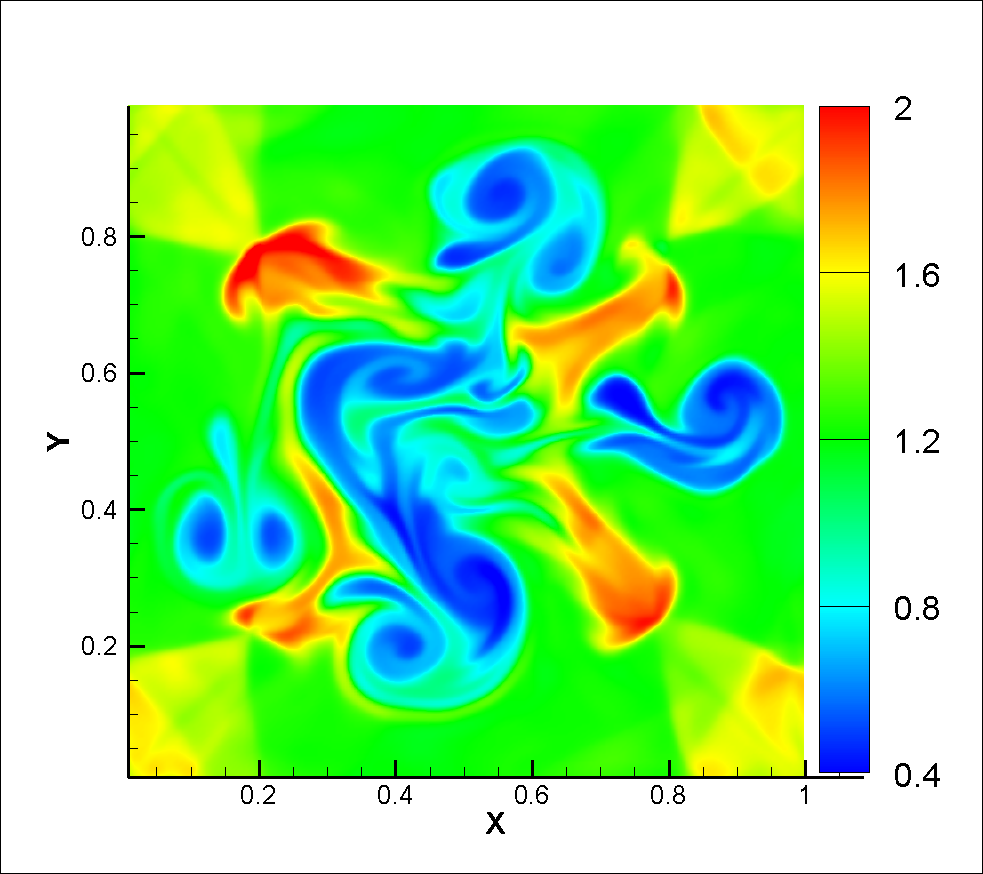}
\caption{$n=256$}
\end{subfigure}
\begin{subfigure}{0.245\linewidth} \centering
\includegraphics[width=1\textwidth]{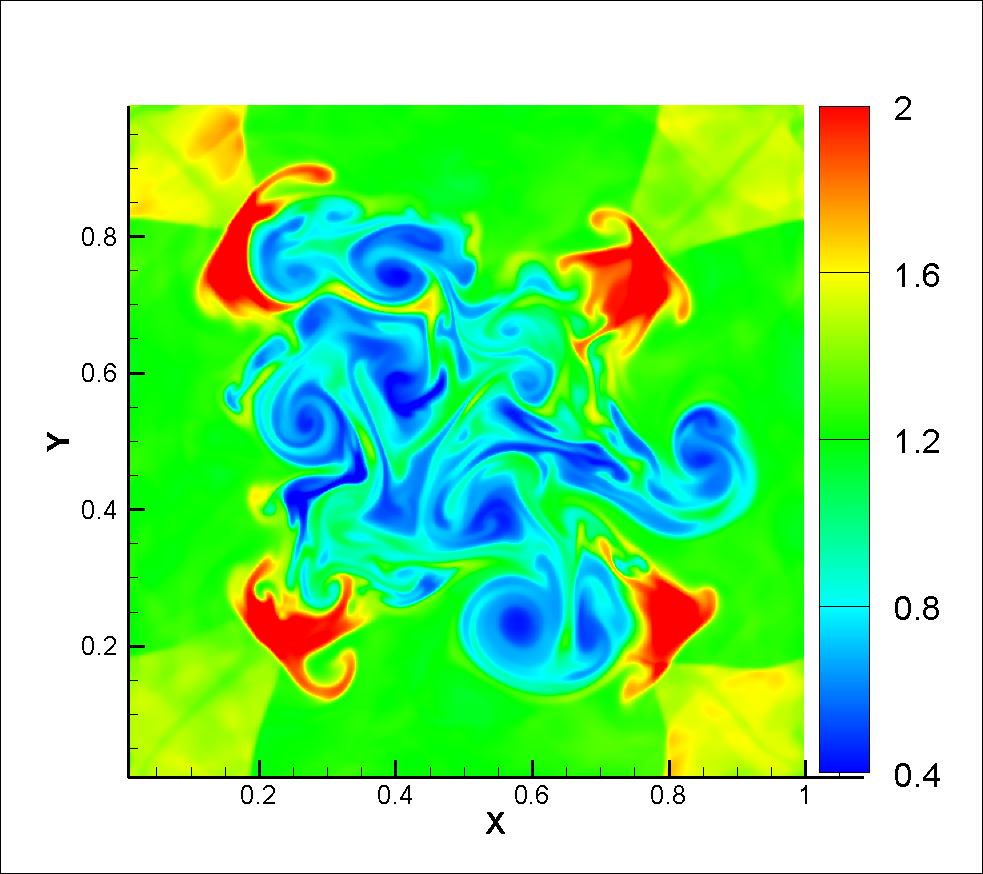}
\caption{$n=512$}
\end{subfigure}
\begin{subfigure}{0.245\linewidth} \centering
\includegraphics[width=1\textwidth]{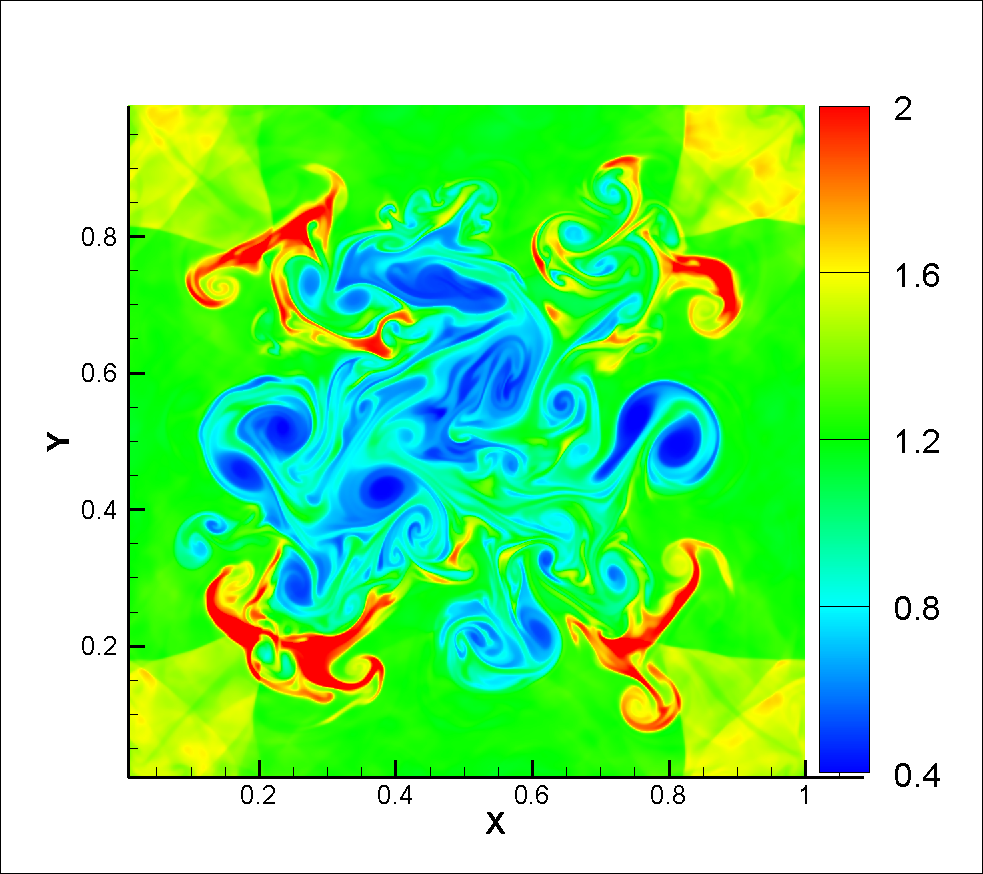}
\caption{$n=1024$}
\end{subfigure}
\begin{subfigure}{0.245\linewidth} \centering
\includegraphics[width=1\textwidth]{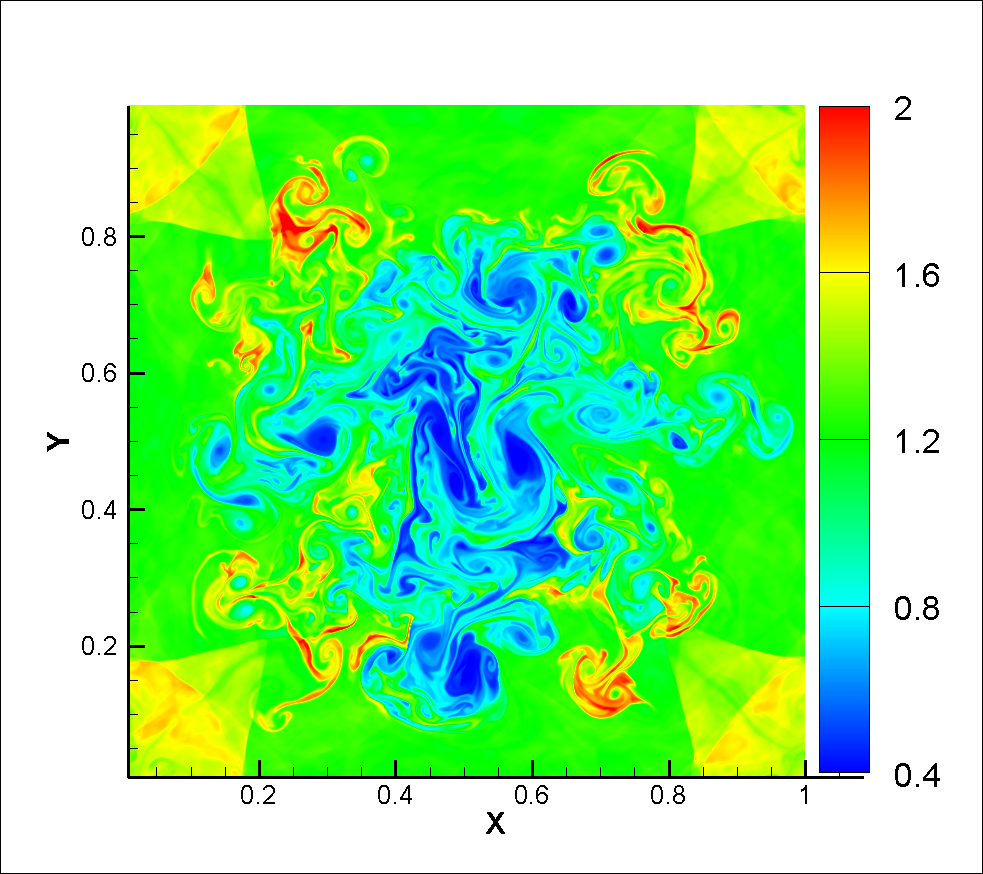}\caption{$n=2048$}
\end{subfigure}
\caption{Experiment~2, density computed by the GRP scheme at $T = 4$ for the Richtmyer-Meshkov
problem on a mesh with $n \times n$ cells.}
\label{fig_w4}
\end{figure}

\begin{figure}[hb]
\begin{subfigure}{0.245\linewidth} \centering
\includegraphics[width=1\textwidth]{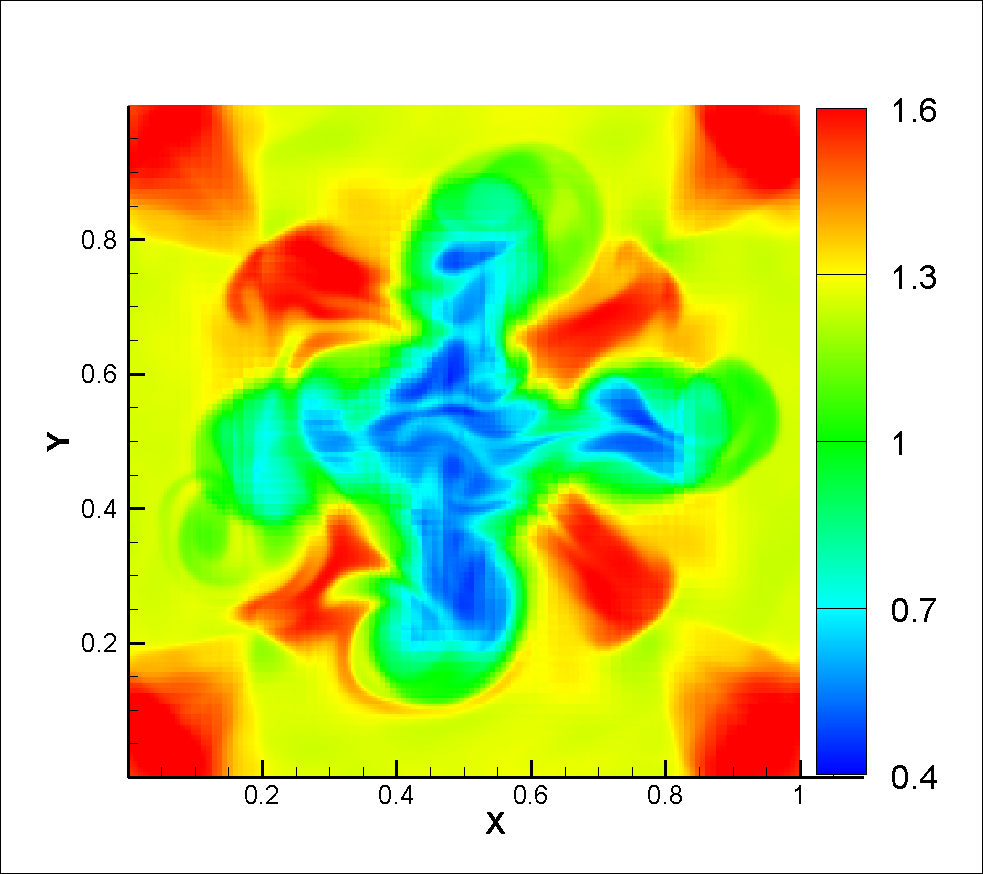}
\caption{$up\ to\ n=256$}
\end{subfigure}
\begin{subfigure}{0.245\linewidth} \centering
\includegraphics[width=1\textwidth]{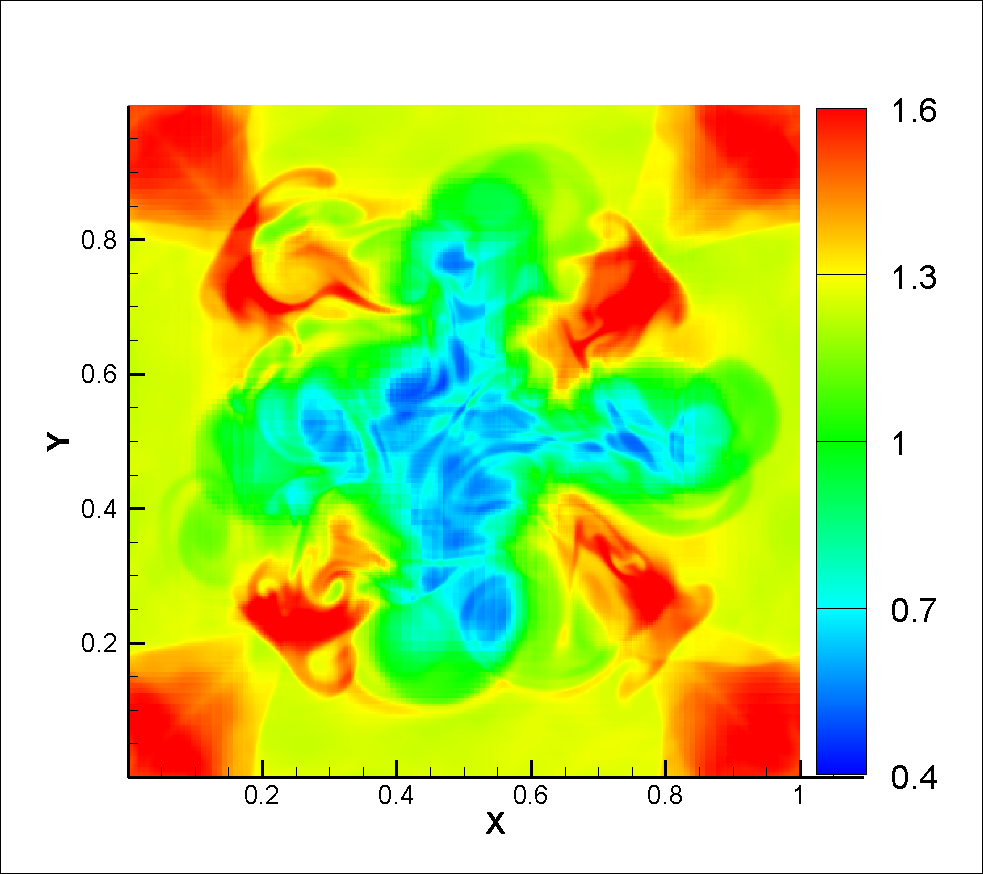}
\caption{$up\ to\ n=512$}
\end{subfigure}
\begin{subfigure}{0.245\linewidth} \centering
\includegraphics[width=1\textwidth]{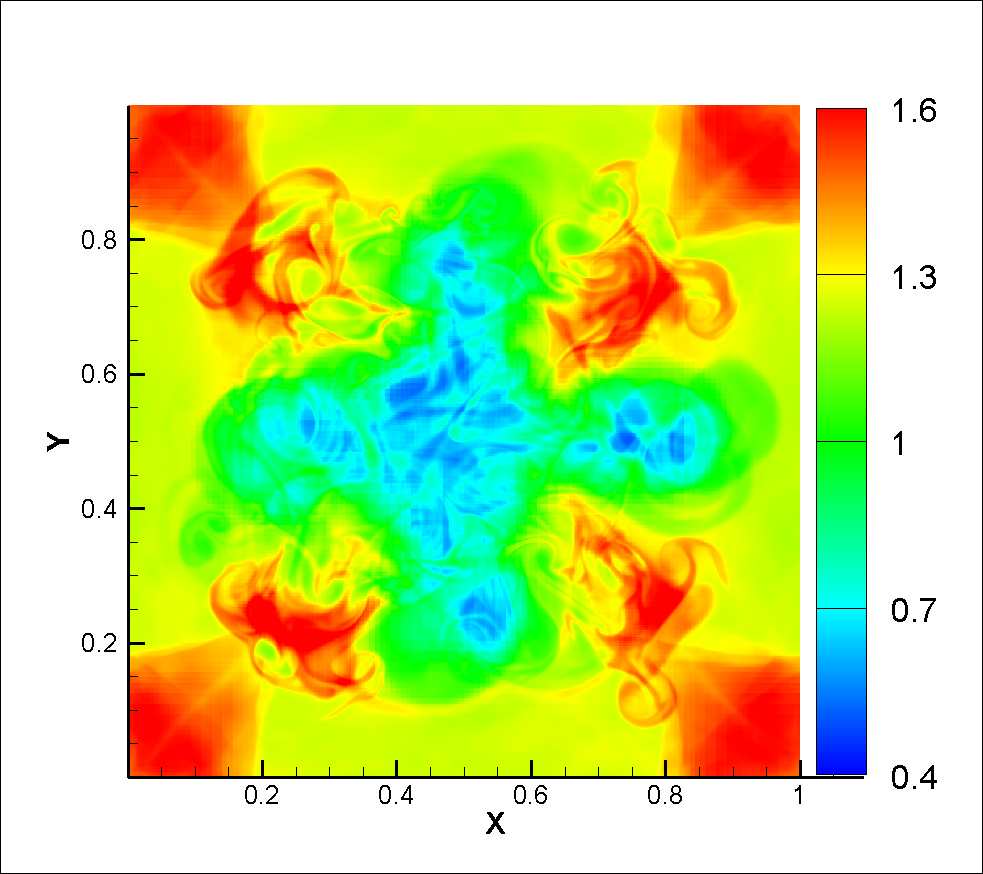}
\caption{$up\ to\ n=1024$}
\end{subfigure}
\begin{subfigure}{0.245\linewidth} \centering
\includegraphics[width=1\textwidth]{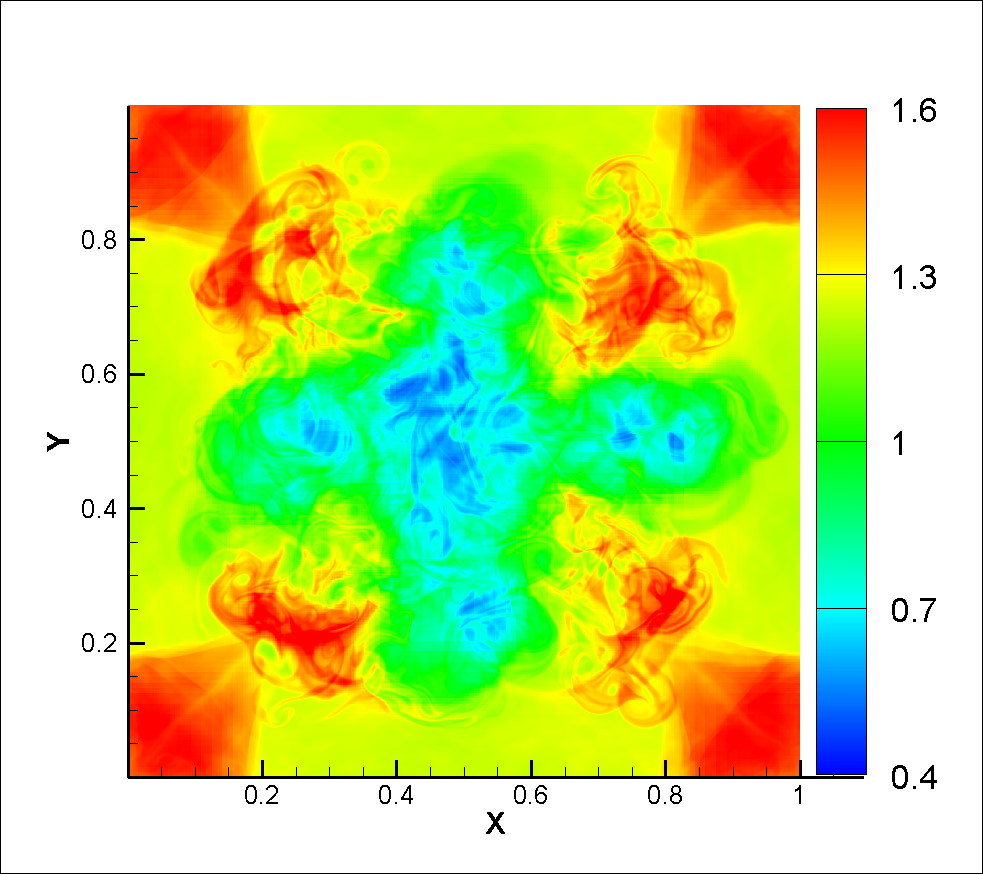}\caption{$up\ to\ n=2048$}
\end{subfigure}
\caption{Experiment~2, Ces\`aro averages of the density computed by the GRP method on  meshes with $j\times j$ cells, $j = 64, 128, \dots, n,$  for the Richtmyer-Meshkov problem.}
\label{fig_w5}
\end{figure}

\begin{figure}[hb]
\begin{subfigure}{0.245\linewidth} \centering
\includegraphics[width=1\textwidth]{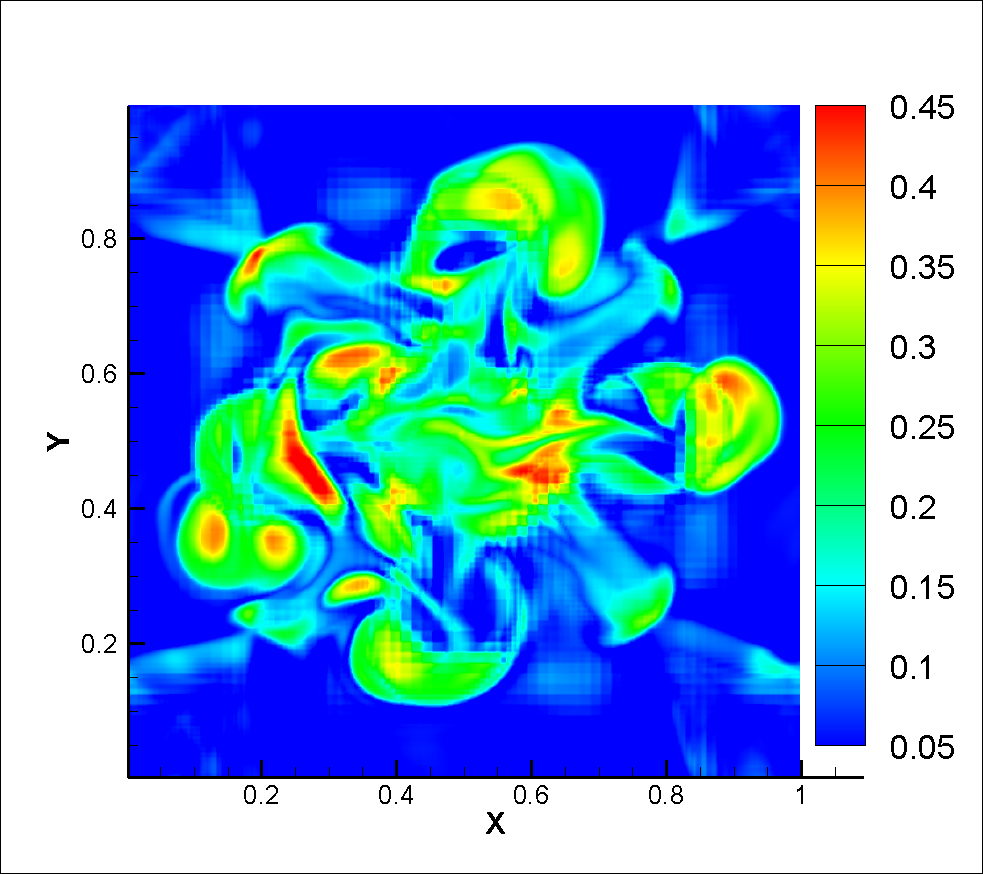}
\caption{$up\ to\ n=256$}
\end{subfigure}
\begin{subfigure}{0.245\linewidth} \centering
\includegraphics[width=1\textwidth]{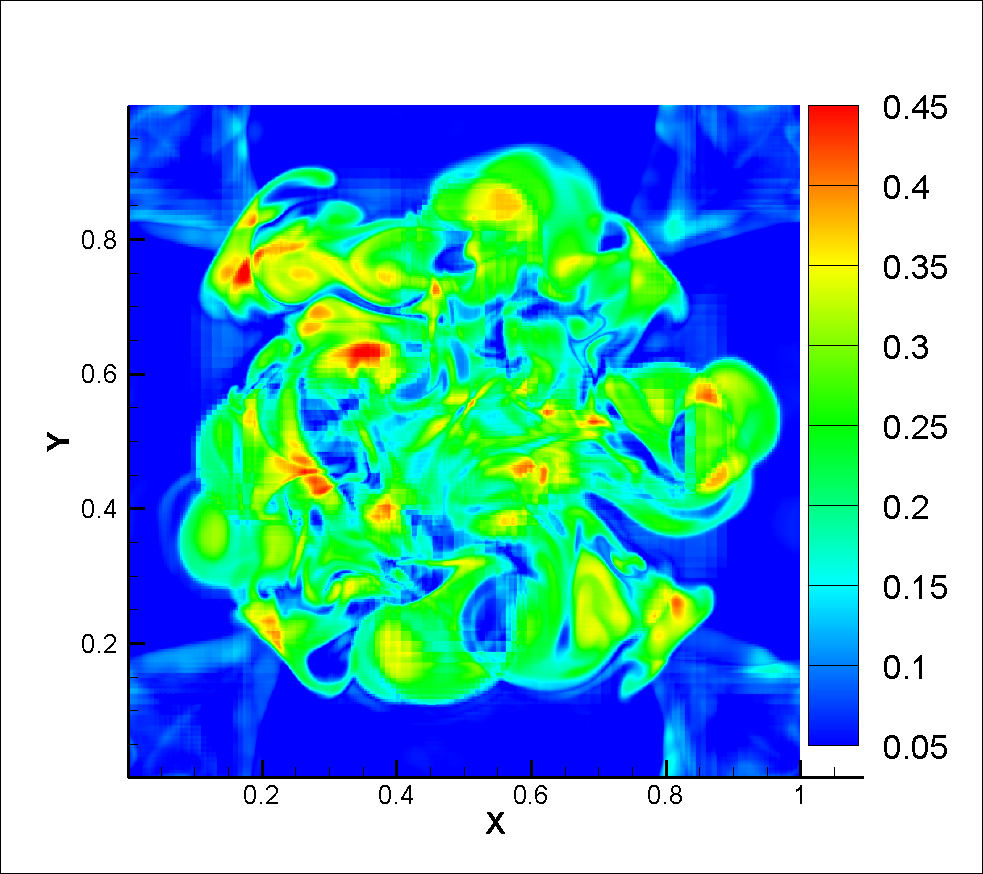}
\caption{$up\ to\ n=512$}
\end{subfigure}
\begin{subfigure}{0.245\linewidth} \centering
\includegraphics[width=1\textwidth]{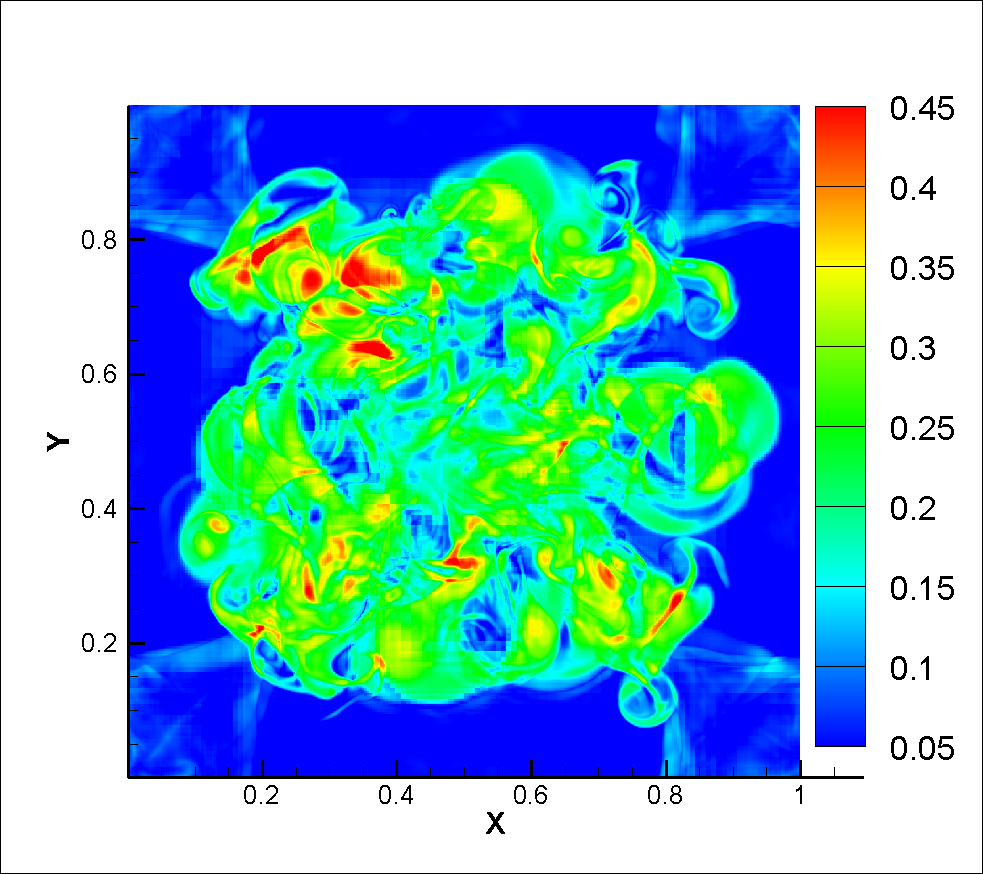}
\caption{$up\ to\ n=1024$}
\end{subfigure}
\begin{subfigure}{0.245\linewidth} \centering
\includegraphics[width=1\textwidth]{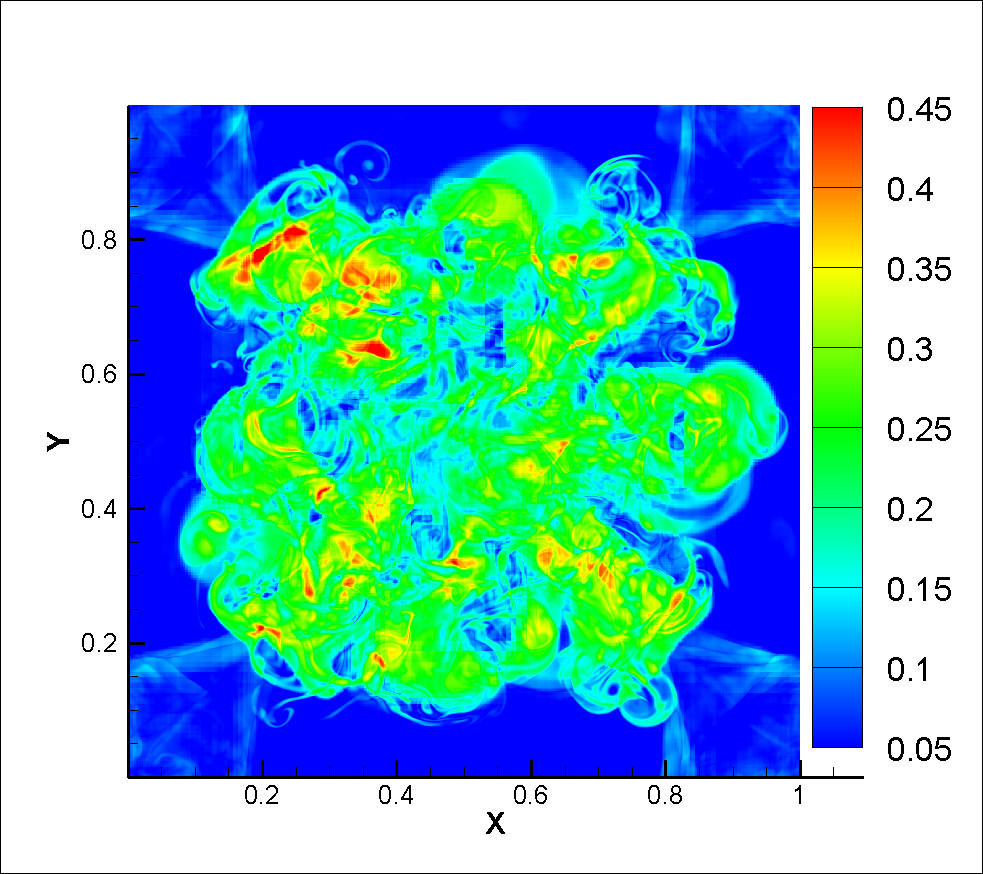}\caption{$up\ to\ n=2048$}
\end{subfigure}
\caption{Experiment~2, first variance of the density computed by the GRP method on  meshes with $j\times j$ cells, $j = 64, 128, \dots, n,$  for the Richtmyer-Meshkov problem.}
\label{fig_w6}
\end{figure}

\clearpage\newpage

\section{Conclusion}

We have developed a new strategy for computing dissipative measure--valued (DMV) solutions of the Euler system of gas dynamics. It is a well-known fact that the Euler equations in two- and three-space dimensions are essentially ill-posed, which is demonstrated by the fact that infinitely many admissible weak solutions exist even for smooth initial data. Therefore the question of the  convergence of numerical schemes, which is studied in the present paper, is of fundamental importance.

We have introduced the concept of consistent approximations and showed that they generate DMV solutions, see Section~\ref{A}.
An example of such a consistent approximation is the FLM finite volume method developed in \cite{FLM_18}, see Section~\ref{Brenner}.
We have also presented an interesting property, that can be seen as a sharper version of the Lax-Wendroff theorem:
a consistent numerical approximations converge either strongly or
their (weak) limit is not a weak solution of the Euler system, see Section~\ref{w}.

In view of these facts a suitable concept of generalized solutions of the Euler system is inevitable.
Applying $\mathcal{K}$-convergence to the Dirac distributions concentrated on the  numerical solutions we showed in Section~\ref{K} that a limiting DMV solution can be effectively computed by the Ces\`aro averages of the consistent approximations.
To this goal, we have studied the convergence of observable quantities, such as the expected value and the first variance, see \eqref{K2}, \eqref{K3}, \eqref{K5}, \eqref{K6}, \eqref{K51},
and showed that the density, momentum and the entropy converge strongly in $L^q((0,T) \times \Omega)$ for  some $q \geq 1$. Note that the energy converges {a.e.} in $(0,T) \times \Omega$.

As an added benefit, we have also obtained that the Ces\`aro averages of the Dirac distributions converge in the Wasserstein distance $W_q$ and  strongly in $L^q((0,T) \times \Omega)$, for some $q \geq 1,$ see Theorem~\ref{Tmain}.

In Section~\ref{Brenner} we have described two numerical schemes for the Euler equations, the FLM finite volume method \cite{FLM_18} and the standard second order GRP method \cite{ben2003, ben2006 ,ben2007}.  Both methods  confirm our theoretical results on $\K$-convergence. This is demonstrated in Section~\ref{numerics}, where we solve numerically two well-known benchmarks, the Kelvin-Helmholtz and the Richtmyer-Meshkov instability problems. As expected,
single solutions computed by both methods do not converge due to new fine scale structures arising on refined meshes. On the other hand, we have
demonstrated $\K$-convergence of the corresponding Dirac distributions as well as the strong convergence of Ce\`saro averages of the numerical solutions and their first variance.

$\K$-convergence can be seen as a new tool in numerical analysis of ill-posed partial differential equations. In particular, it can  be applied to other well-established numerical schemes for the Euler equations in order to study their convergence. In future it will be interesting to extend the concept of  dissipative measure--valued solutions and $\K$-convergence to general hyperbolic conservation laws.

\end{document}